\documentclass{imanum}
\usepackage{amsmath,amssymb,amsthm,graphicx,amsxtra}
\usepackage{graphicx}
\usepackage{mathtools}
\usepackage{xcolor}
\usepackage{enumitem} 
\allowdisplaybreaks
\jno{drnxxx}

\renewcommand{\d}{\/\mathrm{d}\/}

\def\L{\mathrm{L}}

\def\C{\mathrm{C}}

\def\B{\mathrm{B}}
\def\D{\mathrm{D}}

\def\H{\mathrm{H}}

\newcommand{\R}{\mathbb{R}}

\renewcommand{\d}{\/\mathrm{d}\/}

\begin{document}
	
	\title{Finite element approximation for a delayed generalized Burgers-Huxley equation with weakly singular kernels: Part I Well-posedness, Regularity and Conforming approximation}
	\shorttitle{Generalized Burgers-Huxley equation with weakly singular kernel}
	
	\author{%
		{\sc
			Sumit Mahajan\thanks{Email: sumit{\_}m@ma.iitr.ac.in}
			Arbaz Khan \thanks{Corresponding author. Email: arbaz@ma.iitr.ac.in}, 
			and
			Manil T. Mohan\thanks{Email:  maniltmohan@ma.iitr.ac.in, maniltmohan@gmail.com.
		}}  \\[2pt]
		Department of Mathematics, Indian Institute of Technology Roorkee, Roorkee 247667, India 
}
\shortauthorlist{S. Mahajan, A. Khan, M.T. Mohan}

\maketitle\\
	\begin{abstract}
	{The analysis of a delayed generalized Burgers-Huxley equation (a non-linear advection-diffusion-reaction problem) with  weakly singular kernels is carried out in this work. Moreover,  numerical approximations are performed using the conforming finite element method (CFEM). The existence, uniqueness and regularity results for the continuous problem have been discussed in detail using the Faedo-Galerkin approximation technique.	For the numerical studies, we first propose a semi-discrete conforming finite element scheme for space discretization and discuss its error estimates under minimal regularity assumptions. We then employ a backward Euler discretization in time and CFEM in space to obtain a fully-discrete approximation. Additionally, we derive a prior error estimates for the fully-discrete approximated solution. Finally, we present  computational results that support the derived theoretical results.}	Burgers-Huxley equation, Galerkin approximation, weak solution, A priori error analysis, Conforming finite element method.
	\end{abstract}
\section{Introduction}\setcounter{equation}{0} 
The \emph{generalized Burgers-Huxley equation (GBHE) with weakly singular kernels} describes a prototype model for describing the interaction between reaction mechanisms, convection effects and diffusion transports with delays. The GBHE has various applications in different fields such as traffic flows, nuclear waste disposal, the mobility of fish populations, motion of the domain wall of a ferroelectric material in an electric field, etc. In this work, we consider a delayed generalized Burgers-Huxley equation with weakly singular kernels defined on $\Omega \times (0,T)$, where $\Omega\subset \mathbb{R}^d \ (d = 2, 3)$ be an open bounded and simply connected convex domain with Lipschitz boundary $\partial\Omega$. The delayed GBHE under our consideration is given by
\begin{equation}\label{GBHE}
	\left\{\begin{aligned}
		\frac{\partial u(x,t)}{\partial t}&+\alpha u(x,t)^{\delta}\sum_{i=1}^d\frac{\partial u(x,t)}{\partial x_i}-\nu\Delta u(x,t)-\eta\int_{0}^{t} K(t-s)\Delta u(x,s) \mathrm{~d}s\\&=\beta u(x,t)(1-u(x,t)^{\delta})(u(x,t)^{\delta}-\gamma)+f(x,t),  \ (x,t)\in\Omega\times(0,T),\\ u(x,t)&=0, \ x \in \partial\Omega ,\ t\in[0,T),\\
		u(x,0)&=u_0(x), \ x\in{\Omega},
	\end{aligned}\right.
\end{equation}
where $K(t)$ is a weakly singular kernel, $f(\cdot,\cdot)$ is an external forcing, $\alpha,\beta,\nu,\delta,\eta$ are parameters such that $\alpha>0$ is the advection coefficient, $\beta>0$, $\delta\geq 1 $ is the retardation time, $\eta\geq 0$ is the relaxation time, $\gamma\in(0,1)$ and $d=2,3$. The integration term is used to study the delayed effect.

For $\eta=0$, the equation (\ref{GBHE}) becomes  
\begin{align*}
	\frac{\partial u}{\partial t}+\alpha u^{\delta}\sum_{i=1}^d\frac{\partial u}{\partial x_{i}}-\nu\Delta u = \beta u(1-u^{\delta})(u^{\delta}-\gamma)+f,
\end{align*}
which is the generalized Burgers-Huxley equation (GBHE). A wide literature is available on the GBHE and related models. Solitary and travelling wave solutions of the generalized Burgers-Huxley equation using a relevant non-linear transformation is obtained (\citet[see][]{XYW,JS,TSE,AMW1}, etc). The global solvability results for the stochastic generalized Burgers-Huxley (SGBH) equation have also been proved in \cite{KMo}. Many researchers have established numerical solutions of the GBHE by using various numerical techniques such as  \cite{HNS} used an Adomian decomposition method, \citet{KHA} provided a numerical solution of GBHE based on  the collocation method using mesh-less radial basis functions, \cite{SGZ} performed numerical studies using a fourth-order Runge-Kutta (RK4) scheme in time combined with up to tenth-order finite difference scheme in space, \cite{KSM} used a three-step Taylor–Galerkin finite element scheme to approximate the solution of the singularly perturbed GBHE. 

For $\delta=1$, $\eta=0$, $\alpha\neq 0$ and $\beta\neq 0$, the equation (\ref{GBHE})
is known as the \emph{Burgers-Huxley equation} (see \cite{XYW,AMW3}, etc.). 
For $\alpha=0$, $\eta=0$ and $\delta=1$, the equation (\ref{GBHE}) takes the form 
which is known as the \emph{Huxley equation} which describes the nerve pulse propagation in nerve fibres and wall motion in liquid crystals \citep[see][]{XYW1}. For $\beta = 0$, $\delta=1$, and $\alpha = 1$, the equation (\ref{GBHE}) reduces to the well-known viscous Burgers’ equation which represents the mathematical modelling of the turbulence phenomena \citep[see][]{JMB}, far field of wave propagation in non-linear dissipative systems, shock wave theory, vorticity transportation, heat conduction, wave processes in thermoelastic medium, dispersion in porous media, hydrodynamic turbulence, elasticity, gas dynamics, etc. 

The kernel function $K(\cdot)$ is  \emph{weakly singular} such  that $K\in\L^1(0,T)$.  Moreover, we assume that the function $K(\cdot)$ is a \emph{positive kernel}, that is,  if for any $T> 0$, we have
\begin{align}\label{pk}
	\int_0^T \int_0^t K(t-s)u(s)u(t)ds \mathrm{~d}t\geq 0,\  \ \text{ for all }\ u\in \L^2(0, T).
\end{align}
One primary example of the weakly singular kernel is given by 
$$K(t) = \frac{1}{\Gamma(\alpha)}\frac{1}{t^{1-\alpha}}, \quad 0< \alpha < 1,$$ where  $\Gamma(\alpha)=\int_{0}^{\infty}t^{\alpha-1}e^{-t}\mathrm{d}t$ is the \emph{Euler Gamma function}.
The convolution integral can be interpreted as a fractional integral of order $\alpha$ \citep[see, e.g.,][]{AEL, OS}. By varying the fractional integral exponent, we obtain a link between parabolic behavior  (in the limit $\alpha = 0, $ the kernel $K$ approaches to the delta distribution) and hyperbolic behavior (specializing $\alpha =1 $). The global solvability and numerical studies, for the one dimensional GBHE and the two-dimensional stationary generalized Burgers- Huxley equation, have been recently established in  \cite{MKH} and \cite{KMR}, respectively. The literature relevant to the finite element analysis of integro-differential equations of parabolic type is vast. The time discretization for integro-differential equations with a smooth kernel was introduced in \cite{STo}. A prior error estimates for linear parabolic equations can be found in \cite{CLY} and \cite{AKP}, and for both parabolic and hyperbolic equations in \cite{YFa}. Error estimates for the semi-discrete finite element method (FEM) have been discussed in \cite{TZh}. Time discretization and error estimation using the Crank-Nicolson scheme and Pad\'{e}-type scheme are discussed in \cite{CJL} and \cite{ZNy}, respectively.

Fully-discrete FEM for parabolic integro-differential equations with a weakly singular kernel using positive quadrature rule has been discussed in \cite{CTW} and \cite{MTh} . A second-order accurate numerical method for a semi-linear integro-differential equation using the Crank-Nicolson scheme with a weakly singular kernel is discussed in \cite{MMh}. Recently, \cite{ZXD} used the weak Galerkin finite element scheme for the parabolic integro-differential equation with a weakly singular kernel to obtain the optimal convergence order in $\L^2$.

To the best of the authors' knowledge, this work appears to be the first in the field of theoretical and numerical partial differential equations to investigate the solvability of the model generalized Burgers-Huxley equation with memory \eqref{GBHE}. This research article makes significant contributions from both theoretical and numerical perspectives:
\begin{itemize}
	\item \textbf{Well-posedness results}: From theoretical point of view, the well-posedness of the generalized Burgers-Huxley equation with memory (Kernel) is proven for the first time in the literature (Section \ref{sec2.3}). The analysis is carried out in an open bounded and simply connected convex domain,  $\Omega\subset \mathbb{R}^d (d = 2, 3)$, with a Lipschitz boundary $\partial\Omega$. 
	\item \textbf{Regularity results}: The novel regularity results for the equation \eqref{GBHE} in a convex domain or a domain with a $\mathrm{C}^2$-boundary under smoothness assumptions on the initial data and external forcing have been obtained in Section \ref{sec2.4}. These results are significant from a numerical analysis standpoint as well.
	\item \textbf{Semi-discrete error estimates}: This work provides, for the first time, semi-discrete error estimates for non-linear integro-differential equations with weakly singular kernels under minimal regularity assumptions.
	\item \textbf{Fully-discrete FEM error estimates}: The major contribution in the numerical aspect lies in the error estimates for the fully-discrete finite element method (FEM). Unlike the previous literature that necessitates $u_{tt}\in \L^2(0,T;\L^2(\Omega))$ (\cite{VTh}) to handle the time derivative term, which requires the boundary to be at least $\mathrm{C}^4$, our results do not impose such higher regularity requirements. Notably, the minimum regularity result has been established for the linear parabolic equation \cite{LSW}, but our work is the first to extend it for a non-linear equation with weakly singular kernels.
	
\end{itemize}
We show the existence and uniqueness of a global weak solution to the system (\ref{GBHE}), using a Faedo-Galerkin method. Regularity results of the weak solution are obtained under smoothness assumptions on the initial data and external forcing. Here, we propose the conforming finite elements (CFEM) scheme to obtain the numerical approximation of the exact solution under minimal regularity and the required regularity is also proved for GBHE with weakly singular kernels. We rigorously derive a priori error estimates indicating first-order convergence in both space and time of the CFEM. We also include a set of computational tests that confirm the theoretical error bounds and show some model equation properties. The finite element approximation using non-conforming and discontinuous Galerkin (DG) finite element method has been discussed in \cite{GBHE}.

The paper is organized as follows: In section \ref{sec1}, we will discuss the notational conventions and provide the abstract formulation of our problem with the necessary functional spaces. The global solvability results for the system (\ref{GBHE}) are  established in Theorem \ref{EWS}; that is, the existence and uniqueness of a global weak solution is proved  by taking the initial data $u_0\in \L^2(\Omega)$ and external forcing $f \in \L^2(0, T; \H^{-1}(\Omega))$.  The uniqueness result of weak solutions under different choices on the initial data is discussed in Theorem \ref{t3}. The regularity results are obtained in Theorem \ref{2M.thm3.2} and Theorem \ref{2M.thm33} under smoothness assumptions on the initial data and the external forcing. In the subsequent section, the semi-discrete and the fully-discrete scheme are introduced, and a priori error estimates are derived for the CFEM in section \ref{sec3}. Finally, numerical results are presented in section \ref{sec4} for 2D and 3D to support our theoretical results.

\section{Mathematical Setting and Solvability}
\label{sec1}
This section lists necessary function spaces and notations used throughout the paper. Also, we define the operators used to obtain the global solvability results of the system (\ref{GBHE}), discuss its properties and establish the existence of a unique weak solution for our system and its regularity. 
\subsection{Functional setting}
We use the standard notation for function spaces. In particular, the Lebesgue spaces are denoted by $\L^p(\Omega)$ for $p\in[1,\infty]$ and the corresponding norm is denoted by $\|\cdot\|_{\L^p}$ (omitting the domain specification whenever clear from the context). For $p=2$, $\L^2(\Omega)$  is a Hilbert space with the inner product denoted by $(\cdot,\cdot)$. The Hilbertian  Sobolev spaces are denoted by $\H^{k}(\Omega)$. Let  $\mathrm{C}_0^{\infty}(\Omega)$ denote  the space of all infinitely differentiable functions with compact support in $\Omega$ and $\H_0^1(\Omega)$ be the closure of $\mathrm{C}_0^{\infty}(\Omega)$ in $\H^1$-norm. As $\Omega$ is a bounded domain, $\|\nabla\cdot\|_{\L^2}$  defines an equivalent norm on $\H^1_0(\Omega)$ by using the Poincaré inequality. Moreover, we have the Gelfand triplet  $\H_0^1(\Omega)\hookrightarrow\L^2(\Omega)\hookrightarrow\H^{-1}(\Omega)$, where $\H^{-1}(\Omega)$ is the dual of $\H_0^1(\Omega)$, and the embedding $\H_0^1(\Omega)\hookrightarrow\L^2(\Omega)$ is compact. The duality paring between $\H_0^1(\Omega)$ and $\H^{-1}(\Omega)$  as well as ${\L}^p(\Omega)$ and its dual ${\L}^{p'}(\Omega)$, where $\frac{1}{p}+\frac{1}{p'}=1$ is denoted by $\langle\cdot,\cdot\rangle$.  The sum space $\H^{-1}(\Omega)+{\L}^{p'}(\Omega)$ is well defined and 
\begin{align*}
	(\H^{-1}(\Omega)+\L^{p'}(\Omega))'=	\H_0^1(\Omega)\cap\L^p(\Omega) \  \text{and} \ (\H_0^1(\Omega)\cap\L^p(\Omega))'=\H^{-1}(\Omega)+\L^{p'}(\Omega),
\end{align*} 
where $\|y\|_{\H_0^1\cap\L^{p}}=\max\{\|y\|_{\H_0^1},\|y\|_{\L^p}\},$ which is equivalent to the norms  $\|y\|_{\H_0^1}+\|y\|_{\L^{p}}$  and $\sqrt{\|y\|_{\H_0^1}^2+\|y\|_{\L^p}^2}$, and  
\begin{align*}
	\|y\|_{\H^{-1}+\L^{p'}}&=\inf\{\|y_1\|_{\H^{-1}}+\|y_2\|_{\L^{p'}}:y=y_1+y_2,  y_1\in\H^{-1}(\Omega) \ \text{and} \ y_2\in\L^{p'}(\Omega)\}\nonumber\\&=
	\sup\left\{\frac{|\langle y_1+y_2,f\rangle|}{\|f\|_{\H_0^1\cap\L^p}}:0\neq f\in\H_0^1(\Omega)\cap\L^p(\Omega)\right\}.
\end{align*}
Note that $\H_0^1(\Omega)\cap\L^p(\Omega)$ and $\H^{-1}(\Omega)+\L^{p'}(\Omega)$ are Banach spaces. Moreover, we have the continuous embedding $\H_0^1(\Omega)\cap\L^p(\Omega)\hookrightarrow\H_0^1(\Omega)\hookrightarrow\L^2(\Omega)\hookrightarrow\H^{-1}(\Omega)\hookrightarrow\H^{-1}(\Omega)+\L^{p'}(\Omega)$. The  following operators will be used in  further analysis:
\subsubsection{Linear operator}
Let $A$ be the unbounded and self-adjoint linear operator on $\L^2(\Omega)$ defined by:
\begin{align*}
	A u:=-\Delta u, \ D(A)= \H^2(\Omega)\cap\H_0^1(\Omega),
\end{align*}
whenever $\Omega$ is convex, bounded open subset of $\mathbb{R}^n$ with a $\mathrm{C}^2$ boundary $\partial\Omega$ (Theorem 3.1.2.1, \cite{GPi}).  Note that  $D(A)\subset \L^p$, for $p\in[1,\infty]$ and $d=2,3$, using the Sobolev Embedding Theorem  \citep[see][]{MZ}. As the embedding  $\H_0^1(\Omega)\hookrightarrow \L^2(\Omega)$ is compact, $A^{-1}$ exists and is a symmetric compact operator on $\L^2(\Omega)$. Then, by using spectral theorem \citep[see][]{RDJL}, there exists a sequence of eigenvalues of $A$, $0<\lambda_1\leq\lambda_2\leq\ldots\to\infty$  and an orthonormal basis $\{w_k\}_{k=1}^{\infty}$ of $\L^2(\Omega)$ consisting of eigenfunctions of $A$ \citep[see][p. 504]{RDJL}. 

An integration by parts yields $$\langle Au,v\rangle=(\nabla u,\nabla v), \ \text{ for all } \ u,v\in\H_0^1(\Omega),$$ so that $A:\H_0^1(\Omega)\to\H^{-1}(\Omega)$. 

\subsubsection{Non-linear operators}
Let us define the non-linear operator $b:(\H_0^1(\Omega)\cap\L^{2(\delta+1)}(\Omega))\times\H_0^1(\Omega)\times (\H_0^1(\Omega)\cap\L^{2(\delta+1)}(\Omega))\to\mathbb{R}$ by\vspace{-2mm} $$b(u,v,w)=\int_{\Omega}(u(x))^{\delta}\sum_{i=1}^{d}\frac{\partial v(x)}{\partial x_i}w(x)\mathrm{~d}x.$$ The map is well-defined since 
$
|b(u,v,w)|\leq\|u\|_{\L^{2(\delta+1)}}^{\delta}\|v\|_{\H_0^1}\|w\|_{\L^{2(\delta+1)}},
$ so that we can define a mapping $B(\cdot,\cdot):(\H_0^1(\Omega)\cap\L^{2(\delta+1)}(\Omega))\times\H_0^1(\Omega)\to \H^{-1}(\Omega)+\L^{\frac{2(\delta+1)}{2\delta+1}}(\Omega)$ by $\langle\B(u,v),w\rangle=b(u,v,w)$. For simplicity, we use the notation  $B(u)=B(u,u)$. 

For all $p\geq 2$ and $u\in\H_0^1(\Omega)\cap\L^{2(\delta+1)}(\Omega)$,  using an integration by parts and boundary conditions,  we have 
\begin{align*}
	b(u,u,|u|^{p-2}u)&=\left(u^{\delta}\sum_{i=1}^{d}\frac{\partial u}{\partial x_i},|u|^{p-2}u\right)=\frac{1}{\delta+2}\sum_{i=1}^{d}\int_{\Omega}\frac{\partial}{\partial {x_i}}(u(x))^{\delta+2}|u(x)|^{p-2}\mathrm{~d} x\nonumber\\&=-\frac{1}{\delta+2}\sum_{i=1}^{d}\int_{\Omega}(u(x))^{\delta+2}\frac{\partial}{\partial {x_i}}|u(x)|^{p-2}\mathrm{~d} x=-\frac{p-2}{\delta+2}\sum_{i=1}^{d}\int_{\Omega}(u(x))^{\delta}|u(x)|^{p-2}\frac{\partial u(x)}{\partial{x_i}}\mathrm{~d}x\nonumber\\&=-\frac{p-2}{\delta+2}\left(u^{\delta}\sum_{i=1}^{d}\frac{\partial u}{\partial x_i},|u|^{p-2}u\right).
\end{align*}
The above expression implies
\begin{align}\label{7a}
	b(u,u,|u|^{p-2}u)=\left(u^{\delta}\sum_{i=1}^{d}\frac{\partial u}{\partial x_i},|u|^{p-2}u\right)=0.
\end{align}

Let us now show that the operator $B(\cdot)$ is locally Lipschitz. For $ u,v,z \in \H_0^1(\Omega)\cap \L^{2(\delta+1)}(\Omega),$ and $w=u-v$, using integration by parts and Taylor's formula for $0<\theta_1<1$, we have 
\begin{align*}
	&	\langle B(u)-B(v),z\rangle  \nonumber\\&= \left(u^{\delta}\sum_{i=1}^{d}\frac{\partial u}{\partial x_i} - v^{\delta}\sum_{i=1}^{d}\frac{\partial v}{\partial x_i},z\right) = \left(u^{\delta}\sum_{i=1}^{d}\frac{\partial w}{\partial x_i},z\right) + \left((u^{\delta}-v^{\delta})\sum_{i=1}^{d}\frac{\partial v}{\partial x_i},z\right)\nonumber \\& = -\delta\left(u^{\delta-1}\sum_{i=1}^{d}\frac{\partial u}{\partial x_i}w,z\right) -\left(u^{\delta}w,\sum_{i=1}^d\frac{\partial z}{\partial x_i}\right)+ \delta\left((\theta_1 u+(1-\theta_1)v)^{\delta-1}\sum_{i=1}^{d}\frac{\partial v}{\partial x_i}w,z\right)\nonumber\\&\leq\left(\delta\|u\|_{\L^{2(\delta+1)}}^{\delta-1}\|\nabla u\|_{\L^2}\|z\|_{\L^{2(\delta+1)}}+\|u\|_{\L^{2(\delta+1)}}^{\delta}\|\nabla z\|_{\L^2}+\delta2^{\delta}(\|u\|^{\delta-1}_{\L^{2(\delta+1)}}+\|v\|^{\delta-1}_{\L^{2(\delta+1)}})\|\nabla v\|_{\L^2}\|z\|_{\L^{2(\delta+1)}}\right)\nonumber\\&\quad\times\|w\|_{\L^{2(\delta+1)}}\nonumber\\&\leq C\left[\left(\|u\|_{\L^{2(\delta+1)}}^{\delta-1}+\|v\|_{\L^{2(\delta+1)}}^{\delta-1}\right)\left(\|u\|_{\H_0^1}+\|v\|_{\H_0^1}\right)+\|u\|_{\L^{2(\delta+1)}}^{\delta}\right]\|w\|_{\L^{2(\delta+1)}}\|z\|_{\H_0^1\cap\L^{2(\delta+1)}},
\end{align*}
so that 
\begin{align}\label{bllc}
	\|B(u)-B(v)\|_{\H^{-1}+\L^{\frac{2(\delta+1)}{2\delta+1}}} &\leq C\left[\left(\|u\|_{\L^{2(\delta+1)}}^{\delta-1}+\|v\|_{\L^{2(\delta+1)}}^{\delta-1}\right)\left(\|u\|_{\H_0^1}+\|v\|_{\H_0^1}\right)+\|u\|_{\L^{2(\delta+1)}}^{\delta}\right]\|w\|_{\L^{2(\delta+1)}}\nonumber\\&\leq C r^{\delta}\|u-v\|_{\L^{2(\delta+1)}},
\end{align}
for all $\|u\|_{\H_0^1\cap\L^{2(\delta+1)}}, \|v\|_{\H_0^1\cap\L^{2(\delta+1)}}\leq r$. Therefore, the operator $B : \H_0^1(\Omega)\cap \L^{2(\delta+1)}(\Omega) \to\H^{-1}(\Omega)+\L^{\frac{2(\delta+1)}{2\delta+1}}(\Omega)$ is locally Lipschitz.

Let us now consider the non-linear operator  $c(u)=u(1-u^{\delta})(u^{\delta}-\gamma)$. It should be noted that 
\begin{align}\label{7}
	(c(u),u)&=(u(1-u^{\delta})(u^{\delta}-\gamma),u)=((1+\gamma)u^{\delta+1}-\gamma u-u^{2\delta+1},u)\nonumber\\&=(1+\gamma)(u^{\delta+1},u)-\gamma\|u\|_{\L^2}^2-\|u\|_{\L^{2(\delta+1)}}^{2(\delta+1)}\nonumber\\&\leq (1+\gamma)\|u\|^{\delta+1}_{\L^{2(\delta+1)}}\|u\|_{\L^2}-\gamma\|u\|_{\L^2}^2-\|u\|_{\L^{2(\delta+1)}}^{2(\delta+1)},
\end{align}
for all $u\in\L^{2(\delta+1)}(\Omega)$.
Using Taylor's formula and H\"older's inequality, for $u,v,z\in\L^{2(\delta+1)}(\Omega)$ with $w=u-v$, $0<\theta_2<1$ and $0<\theta_3<1$, we get 
\begin{align}\label{cllc}
	\langle c(u)-c(v),z\rangle &=\langle (1+\gamma)(u^{\delta+1}-v^{\delta+1})-\gamma(u-v)-(u^{2\delta+1}-v^{2\delta+1}),z\rangle\\&= (1+\gamma)(\delta+1)( w(\theta_2u+(1-\theta_2)v)^{\delta},z)-\gamma(w,z)\nonumber\\&\quad-(2\delta+1)\langle w(\theta_3u+(1-\theta_3)v)^{2\delta},z\rangle\nonumber\\&\leq (1+\gamma)(\delta+1)2^{\delta-1}\left(\|u\|_{\L^{2(\delta+1)}}^{\delta}+\|v\|_{\L^{2(\delta+1)}}^{\delta}\right)\|w\|_{\L^{2(\delta+1)}}\|z\|_{\L^2}+\gamma\|w\|_{\L^2}\|z\|_{\L^2}\nonumber\\&\quad+(2\delta+1)2^{2\delta-1}\left(\|u\|_{\L^{2(\delta+1)}}^{2\delta}+\|v\|_{\L^{2(\delta+1)}}^{2\delta}\right)\|w\|_{\L^{2(\delta+1)}}\|z\|_{\L^{2(\delta+1)}}\nonumber\\&\nonumber\leq \left((1+\gamma)(\delta+1)2^{\delta}|\Omega|^{\frac{\delta}{2(\delta+1)}}r^{\delta}+\gamma|\Omega|^{\frac{\delta}{\delta+1}}+(2\delta+1)2^{2\delta}r^{2\delta}\right)\|w\|_{\L^{2(\delta+1)}}\|z\|_{\L^{2(\delta+1)}},
\end{align}
for $\|u\|_{\L^{2(\delta+1)}},\|v\|_{\L^{2(\delta+1)}}\leq r$, where $|\Omega|$ is the Lebesgue measure of $\Omega$. Thus, the operator $c(\cdot):\L^{2(\delta+1)}(\Omega)\to\L^{\frac{2(\delta+1)}{2\delta+1}}(\Omega)$ is locally Lipschitz. 

The following lemmas are useful for the further analysis:
\begin{lemma}\citep{MTM}\label{l1}
	Let $K \in \L^1(0,T)$ and $\phi\in \L^2(0,T)$ for some $T>0$. Then
	$$\left(\int_0^{T}\left(\int_0^sK(s-\tau)\phi(\tau)\mathrm{~d}\tau\right)^2\mathrm{~d}s\right)^{\frac{1}{2}}\leq \left(\int_0^{T}|K(s)|\mathrm{~d}s\right)\left(\int_0^{T}\phi^2(s)\mathrm{~d}s\right)^\frac{1}{2}.$$
\end{lemma}

\begin{lemma}[Remark A.2, {\cite{MTM1}}]\label{l2}
	If the kernel $K(\cdot)$ satisfies the  conditions
	\begin{equation}
		\left\{
		\begin{aligned}
			&\nonumber K\in\C^2(0,\infty)\cap\C[0,\infty),\\
			&\nonumber(-1)^kK^{(k)}(t)\geq 0,\textrm{ for all }t>0,k=0,1,2,
		\end{aligned}
		\right.
	\end{equation}
	then  for any  right continuous function $f:[0,T]\to[0,\infty)$  with left limits and for $g:[0,T]\to\R$ such that $(K*g)(\cdot)$ is  well defined, we have 
	\begin{align}\label{211}\int_0^Tf(t)(K*g)(t)g(t)\d t\geq 0.\end{align}
\end{lemma}
\begin{remark}\label{r1}
	By approximating a weakly singular kernel with a smooth kernel having the properties obtained in Lemma \ref{l2}, one can obtain the result \eqref{211} for weakly singular kernels also. For instance, one can approximate $K(t) = \frac{1}{\Gamma(\alpha)}\frac{1}{t^{1-\alpha}}, \ 0< \alpha < 1,$ by $K_{\epsilon}(t) = \frac{1}{\Gamma(\alpha)}\frac{1}{(t+\epsilon)^{1-\alpha}},$ for some $\epsilon>0$. 
\end{remark}

\subsection{Abstract formulation}
By  using the above functional setting, the abstract formulation of the system (\ref{GBHE}) is given by 
\begin{equation}\label{AFGBHEM}
	\left\{
	\begin{aligned}
		\frac{\mathrm{~d}u(t)}{\mathrm{~d}t}+\nu Au(t)+\eta ( K*Au)(t)&=-\alpha B(u(t))+\beta c(u(t))+f(t), \ \text{ in }\ \H^{-1}(\Omega) + \L^{\frac{2(\delta+1)}{2\delta+1}}(\Omega),\\
		u(0)&=u_0\in\L^2(\Omega),
	\end{aligned}
	\right.
\end{equation}
for a.e. $t\in[0,T]$, where $(K*A u)(t)=\int_0^tK(t-s)Au(s)ds.$ 

Let us now provide the definition of weak solution of the system (\ref{AFGBHEM}). 
\begin{definition}[Weak solution]\label{weaksolution}
	For the initial data $u_0\in\L^{2}(\Omega)$ and external forcing $f\in\L^{2}(0,T;\H^{-1}(\Omega))$, we say that a function
	$$u\in \mathrm{L}^{\infty}(0,T; \L^{2}(\Omega))\cap\L^{2}(0,T; \H_0^1(\Omega))\cap\L^{2(\delta+1)}(0,T; \L^{2(\delta+1)}(\Omega)),$$ with $$\partial_tu\in\L^{\frac{2(\delta+1)}{2\delta+1}}(0,T;\H^{-1}(\Omega)+\L^{\frac{2(\delta+1)}{2\delta+1}}(\Omega)),$$ is called \emph{a weak solution} to the system (\ref{AFGBHEM}), if, for a.e. $t\in(0,T)$,  $u(\cdot)$ satisfies: 
	\begin{equation}
		\left\{
		\begin{aligned}
			\nonumber\langle\partial_tu(t),v\rangle+\nu (\nabla u(t),\nabla v)+\alpha b(u(t),u(t),v)+\eta(( K*\nabla u)(s),\nabla v)&=\beta\langle c(u(t)),v\rangle+\langle f(t),v\rangle, \\
			(u(0),v)&=(u_0,v),
		\end{aligned}
		\right.
	\end{equation}
	for any $v\in\H_0^1(\Omega)\cap\L^{2(\delta+1)}(\Omega)$. 
\end{definition}
\subsection{Existence and uniqueness of weak solution}\label{sec2.3}
Let us now discuss the existence and uniqueness of the weak solution for the system (\ref{GBHE}).
\begin{theorem}[Existence of a weak solution]\label{EWS}
	For $u_0\in\L^{2}(\Omega)$ and $f\in\L^{2}(0,T;\H^{-1}(\Omega)),$	there exists a weak solution to the system (\ref{GBHE}) in the sense of Definition \ref{weaksolution}.  
\end{theorem}
\begin{proof}
	We establish the existence of a weak solution to the system (\ref{GBHE}) in the following steps:
	\vskip 2mm
	\noindent\textbf{Step 1:} \emph{Faedo-Galerkin approximatiiron.} Let  the set $\{w_k(x)\}_{k=1}^{\infty}$ be an orthogonal basis of  $\H_0^1(\Omega)$ and orthonormal basis of $\L^2(\Omega)$ (page 504, \cite{RDJL}). We look for a function $u_m:[0,T]\to\H_0^1(\Omega)$ of the form $u_m(t)=\sum\limits_{k=1}^md_m^k(t)w_k$, where $m$ is a fixed positive integer and the coefficients $d_m^k(t)$, for  $0\leq t\leq T$ and $k=1,\ldots,m$ are selected so that $d_m^k(0)=(u_0,w_k)$, for $k=1,2,\ldots,m$ and 
	\begin{align}\label{8}
		&(u_m'(t),w_k)+\nu(\nabla u_m(t),\nabla w_k)+\alpha\left(u_m(t)^{\delta}\sum_{i=1}^{d}\frac{\partial u_m(t)}{\partial x_i},w_k\right)+\eta(( K*\nabla u_m)(t),\nabla w_k)\nonumber\\&=\beta(u_m(t)(1-(u_m(t))^{\delta})((u_m(t))^{\delta}-\gamma),w_k)+\langle f(t),w_k\rangle,
	\end{align}
	for a.e.  $0\leq t\leq T$ and $k=1,\ldots,m$. As the non-linear operators are locally Lipschitz  
	(see (\ref{bllc}) and (\ref{cllc}))
	using Carath\'eodory's existence theorem, 	the finite dimensional problem (\ref{8}) has a local solution $u_m\in C([0,T];\H_m)$, where $\H_m=\mathrm{span}\{w_1,\ldots,w_m\}$.
	\vskip 2mm
	\noindent\textbf{Step 2:} \emph{$\L^2$-energy estimate.} Multiplying (\ref{8}) by $d_m^k(t)$, and then summing for $k=1,\ldots,m$, we have
	\begin{align*}
		&\frac{1}{2}\frac{\mathrm{~d}}{\mathrm{~d}t}\|u_m(t)\|_{\L^2}^2+\nu\|\nabla u_m(t)\|_{\L^2}^2+\alpha\left(u_m(t)^{\delta}\sum_{i=1}^d\frac{\partial u_m(t)}{\partial x_i},u_m(t)\right)+\eta(( K*\nabla u_m)(t),\nabla u_m(t))\nonumber\\&=\beta(u_m(t)(1-(u_m(t))^{\delta})((u_m(t))^{\delta}-\gamma),u_m(t))+\langle f(t),u_m(t)\rangle,
	\end{align*}
	for a.e. $t\in[0,T]$.
	Using the property of the non-linear operator (\ref{7a}) and the estimate (\ref{7}) , we get 

	\begin{align*}
		&\frac{1}{2}\frac{\mathrm{~d}}{\mathrm{~d}t}\|u_m(t)\|_{\L^2}^2+\nu\|\nabla u_m(t)\|_{\L^2}^2+\beta\gamma\|u_m(t)\|_{\L^2}^2+\beta\|u_m(t)\|_{\L^{2(\delta+1)}}^{2(\delta+1)}+\eta((K*\nabla u_m)(t),\nabla u_m(t))\nonumber\\&=\beta(1+\gamma)((u_m(t))^{\delta+1},u_m(t))+\langle f(t),u_m(t)\rangle\nonumber\\&\leq\beta(1+\gamma)\|u_m(t)\|_{\L^{2(\delta+1)}}^{\delta+1}\|u_m(t)\|_{\L^2}+\|f(t)\|_{\H^{-1}}\|\nabla  u_m(t)\|_{\L^2}\nonumber\\&\leq \frac{\beta}{2}\|u_m(t)\|_{\L^{2(\delta+1)}}^{2(\delta+1)}+\frac{\beta(1+\gamma)^2}{2}\|u_m(t)\|_{\L^2}^2+\frac{\nu}{2}\|\nabla u_m(t)\|_{\L^2}^2+\frac{1}{2\nu}\|f(t)\|_{\H^{-1}}^2.
	\end{align*}
	Integrating the above inequality from $0$ to $t$, we have 
	\begin{align}\label{11}
		&\|u_m(t)\|_{\L^2}^2+\nu\int_0^t\|\nabla u_m(s)\|_{\L^2}^2\mathrm{~d} s+\beta\int_0^t\|u_m(s)\|_{\L^{2(\delta+1)}}^{2(\delta+1)}\mathrm{~d}s+2\eta\int_0^t((K*\nabla u_m)(s),\nabla u_m(s))\mathrm{~d} s\nonumber\\&\leq\|u_0\|_{\L^2}^2+\frac{1}{\nu}\int_0^t\|f(s)\|_{\H^{-1}}^2\mathrm{~d}s+\beta(1+\gamma^2)\int_0^t\|u_m(s)\|_{\L^2}^2\mathrm{~d}s,
	\end{align}
	for all $t\in[0,T]$, where we have used the fact that $\|u_m(0)\|_{\L^2}\leq\|u_0\|_{\L^2}$. Using \eqref{pk} and an application of Gronwall's inequality in \eqref{11} yields 
	\begin{align*}
		&\|u_m(t)\|_{\L^2}^2+\nu\int_0^t\|\nabla u_m(s)\|_{\L^2}^2\mathrm{~d} s+\beta\int_0^t\|u_m(s)\|_{\L^{2(\delta+1)}}^{2(\delta+1)}\mathrm{~d}s\leq\left(\|u_0\|_{\L^2}^2+\frac{1}{\nu}\int_0^T\|f(t)\|_{\H^{-1}}^2\mathrm{~d}t\right)e^{\beta(1+\gamma^2)T},
	\end{align*}
for all $t\in[0,T]$. Taking supremum over time $0\leq t\leq T$, we have 
	\begin{align}\label{13} 
		&\sup_{0\leq t\leq T}\|u_m(t)\|_{\L^2}^2+\nu\int_0^T\|\nabla u_m(t)\|_{\L^2}^2\mathrm{~d} t+\beta\int_0^T\|u_m(t)\|_{\L^{2(\delta+1)}}^{2(\delta+1)}\mathrm{~d}t\nonumber\\&\leq\left(\|u_0\|_{\L^2}^2+\frac{1}{\nu}\int_0^T\|f(t)\|_{\H^{-1}}^2\mathrm{~d}t\right)e^{\beta(1+\gamma^2)T},
	\end{align}
	and the right hand side of the inequality \eqref{13} is independent of $m$.
	\vskip 2mm
	\noindent\textbf{Step 3:} \emph{Estimate for the time derivative.} For any $v\in\H_0^1(\Omega)\cap \L^{2(\delta+1)}(\Omega)$, with $\|v\|_{\H_0^1\cap \L^{2(\delta+1)}}\leq 1$, we have 
	\begin{align*}
		\langle u_m',v\rangle= (u_m',v)&=-\nu(\nabla u_m,\nabla v)-\alpha b(u_m,u_m,v)+\eta(K*\nabla u_m,\nabla v)+\beta \langle c(u_m),v\rangle+\langle f,v\rangle.
	\end{align*}
	An application of the Cauchy-Schwarz inequality and H\"older's inequality yield
	\begin{align*}
		|\langle u_m',v\rangle|&\leq \nu|(\nabla u_m,\nabla v)|+\alpha|(b(u_m,u_m,v^1))|+\eta|(K*\nabla u_m,\nabla v)|+\beta|(c(u_m),v)|+|\langle f,v\rangle|\nonumber\\&\leq \left(\nu\|\nabla u_m\|_{\L^2}+\frac{\alpha}{\delta+1}\|u_m\|^{\delta+1}_{\L^{2(\delta+1)}}+\|f\|_{\H^{-1}}+\eta\|K*\nabla u_m\|_{\L^2}\right)\|\nabla v\|_{\L^2}\nonumber\\&\quad+\beta\left((1+\gamma)\|u_m\|^{\delta+1}_{\L^{2(\delta+1)}}+\gamma\|u_m\|_{\L^2}\right)\|v\|_{\L^2}+\beta\|u_m\|^{2\delta+1}_{\L^{2(\delta+1)}}\|v\|_{\L^{2(\delta+1)}}.
	\end{align*}
	Taking supremum over all $v \in \H_0^1(\Omega) \cap \L^{2(\delta+1)}(\Omega)$, we have 
	\begin{align*}
		\|u_m'\|_{\H^{-1}+\L^{\frac{2(\delta+1)}{2\delta+1}}} &\leq \bigg( \nu\|\nabla u_m\|_{\L^2}+\frac{\alpha}{\delta+1}\|u_m\|^{\delta+1}_{\L^{2(\delta+1)}}+\|f\|_{\H^{-1}}+\eta\|K*\nabla u_m\|_{\L^2}\nonumber\\&\qquad+C\beta(1+\gamma)\|u_m\|^{\delta+1}_{\L^{2(\delta+1)}}+C\gamma\|u_m\|_{\L^2}+\beta\|u_m\|^{2\delta+1}_{\L^{2(\delta+1)}} \bigg) .
	\end{align*}  
	Integrating from 0 to T and using Young's inequality, we arrive at 
	\begin{align}\label{19}
		\nonumber\int_{0}^{T}	\|u_m'(t)\|_{\H^{-1}+\L^{\frac{2(\delta+1)}{2\delta+1}}}^{\frac{2(\delta+1)}{2\delta+1}}\mathrm{~d}t \leq&  C\bigg\{ \nu^{\frac{2(\delta+1)}{2\delta+1}}	\int_{0}^{T}\|\nabla u_m(t)\|_{\L^2}^{\frac{2(\delta+1)}{2\delta+1}}\mathrm{~d}t+\left(\frac{\alpha}{\delta+1}\right)^{\frac{2(\delta+1)}{2\delta+1}}	\int_{0}^{T}\|u_m(t)\|^{\frac{2(\delta+1)^2}{2\delta+1}}_{\L^{2(\delta+1)}}\mathrm{~d}t\\\nonumber &+	\int_{0}^{T}\|f(t)\|_{\H^{-1}}^{\frac{2(\delta+1)}{2\delta+1}}\mathrm{~d}t+\eta^{\frac{2(\delta+1)}{2\delta+1}}\int_{0}^{T}\|K*\nabla u_m(t)\|_{\L^2}^{\frac{2(\delta+1)}{2\delta+1}}\mathrm{~d}t\nonumber\\&+(\beta(1+\gamma))^{\frac{2(\delta+1)}{2\delta+1}}	\int_{0}^{T}\|u_m(t)\|^{\frac{2(\delta+1)^2}{2\delta+1}}_{\L^{2(\delta+1)}}\mathrm{~d}t\nonumber\\&+\gamma^{\frac{2\delta}{2\delta+1}}	\int_{0}^{T}\|u_m(t)\|^{\frac{2(\delta+1)}{2\delta+1}}_{\L^2}\mathrm{~d}t+\beta^{\frac{2(\delta+1)}{2\delta+1}}	\int_{0}^{T}\|u_m(t)\|^{2(\delta+1)}_{\L^{2(\delta+1)}} \mathrm{~d}t\bigg\},
	\end{align}
	where $	\int_0^T\|K*\nabla u_m(t)\|_{\L^2}^{\frac{2(\delta+1)}{2\delta+1}} \mathrm{~d}t $ is bounded since (cf. Lemma \ref{l1})
	\begin{align*}
		\int_0^T\|K*\nabla u_m(t)\|_{\L^2}^{\frac{2(\delta+1)}{2\delta+1}} \mathrm{~d}t &\leq T^{\frac{\delta}{2\delta+1}}\left(\int_0^T\|K*\nabla u_m(t)\|_{\L^2}^2 \mathrm{~d}t\right)^{\frac{\delta+1}{2\delta+1}} \nonumber\\&\leq  T^{\frac{\delta}{2\delta+1}} \left(\int_{0}^{T}|K(t)|\mathrm{~d}t\right)^{\frac{2(\delta+1)}{2\delta+1}} \left(\int_0^T\|\nabla  u_m(t)\|_{\L^2}^2 \mathrm{~d}t\right)^{\frac{\delta+1}{2\delta+1}}  < \infty,
	\end{align*}
	and 
	\begin{align*}
		\int_{0}^{T}\|f(t)\|_{\H^{-1}}^{\frac{2(\delta+1)}{2\delta+1}}\mathrm{~d}t \leq T^{\frac{\delta}{2\delta+1}} \left(\int_{0}^{T}\|f(t)\|_{\H^{-1}}^2\mathrm{~d}t\right)^{\frac{\delta+1}{2\delta+1}}.
	\end{align*}
	\vskip 2mm
	\noindent\textbf{Step 4:} \emph{Passing to limit.} As the energy estimates (\ref{13}) and (\ref{19}) are uniformly bounded and independent of $m$, using the Banach-Alaoglu theorem, we can extract a subsequence $\{u_{m_j}\}_{j=1}^{\infty}$ of $\{u_m\}_{m=1}^{\infty}$ such that 
	\begin{equation}\label{2M.25}
		\left\{
		\begin{aligned}
			&u_{m_j}\xrightarrow{w^*} u\ \text{ in }\ \mathrm{L}^{\infty}(0,T;\L^{2}(\Omega)),\\
			&u_{m_j}\xrightarrow{w} u\ \text{ in }\	\mathrm{L}^2(0,T;\H_0^1(\Omega)), \\
			&u_{m_j}\xrightarrow{w} u\ \text{ in }\ \mathrm{L}^{2(\delta+1)}(0,T;\L^{2(\delta+1)}(\Omega)),\\
			&\partial_tu_{m_j}\xrightarrow{w} \partial_tu\ \text{ in }\	\partial_tu\in\L^{\frac{2(\delta+1)}{2\delta+1}}(0,T;\H^{-1}(\Omega)+\L^{\frac{2(\delta+1)}{2\delta+1}}(\Omega)),
		\end{aligned}
		\right.
	\end{equation}
	as $j\to\infty$, for all $\delta\geq 1$. Since we have compact embedding of $\H_0^1(\Omega)\subset\L^2(\Omega)$ (page 284, \cite{LCE}), an application of the Aubin-Lions compactness lemma yields the following strong convergence:
	\begin{align}\label{2M.26}
		u_{m_j}\rightarrow u\ \text{ in }\ \mathrm{L}^{2}(0,T;\L^2(\Omega)),
	\end{align}
	as $j\to\infty$. The above strong convergence also implies $u_{m_{k_j}}\rightarrow u$, for a.e. $(t,x)\in[0,T]\times\Omega$ as $j\to\infty$ (that is, along a subsequence).  Let us now take limit in (\ref{8}) along the subsequence . Choose a function $v\in \C^1([0,T];\H_0^1(\Omega))$ having the form $v(t)=\sum\limits_{k=1}^Nd_m^k(t)w_k,$ where $\{d^k_m(t)\}_{k=1}^N$ are given smooth functions and $N$ is a fixed integer. Let us choose $m\geq N$, multiply (\ref{8}) by $d_m^k(t)$, sum from $k=1$ to $N$ and then integrate it from $0$ to $T$ to find 
	\begin{align}\label{2M.27}
		&\int_0^T\left[\langle \partial_t u_{m}(t),v(t)\rangle  +\nu(\nabla u_m(t),\nabla v(t))+\eta\langle(K*\nabla u_m(t)),\nabla v(t)\rangle+\alpha\left((u_m(t))^{\delta}\sum_{i=1}^d\frac{\partial u_m}{\partial x_i}(t),v(t)\right)\right]\mathrm{~d}t\nonumber\\&=\int_0^T\left[\beta(u_m(t)(1-(u_m(t))^{\delta})((u_m(t))^{\delta}-\gamma),v(t))+\langle f(t),v(t)\rangle\right]\mathrm{~d}t.
	\end{align}
	Setting $m=m_{j}$ (more specifically $m_{k_j}$) and using the convergences given in  (\ref{2M.25}) and (\ref{2M.26}), we can take weak limit in (\ref{2M.27}) to obtain 
	\begin{align}\label{2M.28}
		&\nonumber\int_0^T\left[\langle \partial_t u(t),v(t)\rangle  +\nu(\nabla u(t),\nabla v(t))+\eta\langle(K*\nabla u(t)),\nabla v(t)\rangle+\alpha\left((u(t))^{\delta}\sum_{i=1}^d\frac{\partial u_m}{\partial x_i}(t),v(t)\right)\right]\mathrm{~d}t\\&=\int_0^T\left[\beta(u(t)(1-(u(t))^{\delta})((u(t))^{\delta}-\gamma),v(t))+\langle f(t),v(t)\rangle \right]\mathrm{~d}t,
	\end{align}
	provided if we show that 
	\begin{align*}
		B(u_{m_{k_j}})\xrightarrow{w} B(u) , \ \text{ and }\ 	c(u_{m_{k_j}})\xrightarrow{w} c(u )\ \text{ in }\ \H^{-1}(\Omega) + \L^{\frac{2(\delta+1)}{2\delta+1}}(\Omega), \ \text{ as } \ j\to\infty. 
	\end{align*}
	Using density argument $ C_0^{\infty}(\Omega)$ is dense in $\H_0^1(\Omega)$, we will show that $b(u_{m_{k_j}},u_{m_{k_j}},v)\to b(u,u,v),$ for all $v \in C_0^{\infty}(\Omega)$ to obtain that $B(u_{m_{k_j}})\xrightarrow{w} B(u) \ \text{ in }\ H^{-1}(\Omega) + \L^{\frac{2(\delta+1)}{2\delta+1}}(\Omega)$, as  $j\to\infty$. By an application of Taylor's formula, integration by parts, H\"older's inequality, and convergence, we obtain 
	\begin{align*}
		&	|b(u_{m_{k_j}},u_{m_{k_j}},v)- b(u,u,v)|\nonumber\\&= \left|\frac{1}{\delta+1}\sum_{i=1}^d\int_{\Omega}(u_{m_{k_j}}^{\delta+1}(x)-u^{\delta+1}(x))\frac{\partial v(x)}{\partial x_i}\mathrm{~d} x\right|\nonumber\\&=\left|\sum_{i=1}^d\int_{\Omega}(\theta u_{m_{k_j}}(x)+(1-\theta)u(x))^{\delta}(u_{m_{k_j}}(x)-u(x))\frac{\partial v(x)}{\partial x_i}\mathrm{~d} x\right|\nonumber\\&\leq\|u_{m_{k_j}}-u\|_{\L^2}\left(\|u_{m_{k_j}}\|_{\L^{2(\delta+1)}}^{\delta}+\|u\|_{\L^{2(\delta+1)}}^{\delta}\right)\|\nabla v\|_{\L^{2(\delta+1)}}\nonumber\\&\to 0\ \text{ as } \ j\to\infty, \ \text{ for all } \ v\in C_0^{\infty}(\Omega). 
	\end{align*}
	Again using Taylor's formula, interpolation and H\"older's inequalities, , we find 
	\begin{align*}
		&	|(c(u_{m_{k_j}})-c(u),v)|\nonumber\\&\leq(1+\gamma)\left|\int_{\Omega}(u_{m_{k_j}}^{\delta+1}(x)-u^{\delta+1}(x))v(x)\mathrm{~d} x\right|+\left|\int_{\Omega}(u_{m_{k_j}}(x)-u(x))v(x)\mathrm{~d} x\right| \nonumber\\
		&\quad+\left|\int_{\Omega}(u_{m_{k_j}}^{2\delta+1}(x)-u^{2\delta+1}(x))v(x)\mathrm{~d} x\right|\nonumber\\
		&\leq \left((1+\gamma)(\delta+1)
		\left(\|u_{m_{k_j}}\|_{\L^{2(\delta+1)}}^{\delta}+\|u\|_{\L^{2(\delta+1)}}^{\delta}\right)\|v\|_{\L^{2(\delta+1)}}+\|v\|_{\L^2}\right)\|u_{m_{k_j}}-u\|_{\L^2}\nonumber\\
		&\quad+(1+2\delta)\|u_{m_{k_j}}-u\|_{\L^2}^{\frac{1}{\delta}}\left(\|u_{m_{k_j}}\|_{\L^{2(\delta+1)}}^{1-\frac{1}{\delta}}+\|u\|_{\L^{2(\delta+1)}}^{1-\frac{1}{\delta}}\right)\left(\|u_{m_{k_j}}\|_{\L^{2(\delta+1)}}^{2\delta}+\|u\|_{\L^{2(\delta+1)}}^{2\delta}\right)\|v\|_{\L^{\infty}}\nonumber\\&\to 0\ \text{ as } \ j\to\infty, \ \text{ for all } \ v\in C_0^{\infty}(\Omega).
	\end{align*}
	which completes the convergence proof.  
	\vskip 2mm
	\noindent\textbf{Step 5:} \emph{Initial data.} Let us now show that $u(0)=u_0$. From (\ref{2M.28}), we note that
	\begin{align}\label{2M.31}
		&\int_0^T\left[-\langle u(t), v'(t)\rangle  +\nu(\nabla u(t),\nabla v(t))+\eta\langle(K*\nabla u(t)),\nabla v(t)\rangle+\alpha\left((u(t))^{\delta}\sum_{i=1}^2\frac{\partial u}{\partial x_i}(t),v(t)\right)\right]\mathrm{~d}t\nonumber\\&=\int_0^T\left[\beta(u(t)(1-(u(t))^{\delta})((u(t))^{\delta}-\gamma),v(t))+(f(t),v(t))\right]\mathrm{~d}t+(u(0),v(0)),
	\end{align} 
	for each $v\in \C^1([0,T];\H_0^1(\Omega))$ with $v(T)=0$. In a similar fashion, we deduce from (\ref{2M.27}) that 
	\begin{align*}
		&\int_0^T\left[-\langle u_{m}(t),v'(t)\rangle  +\nu(\nabla  u_m(t),\nabla v(t))+\eta\langle(K*\nabla  u_m(t)),\nabla v(t)\rangle+\alpha\left((u_m(t))^{\delta}\sum_{i=1}^2\frac{\partial u_m}{\partial x_i}(t),v(t)\right)\right]\mathrm{~d}t\nonumber\\&=\int_0^T\left[\beta(u_m(t)(1-(u_m(t))^{\delta})((u_m(t))^{\delta}-\gamma),v(t))+(f(t),v(t))\right]\mathrm{~d}t+(u_m(0),v(0)).
	\end{align*}
	Taking $m=m_j$ in the above expression, using the above convergences ((\ref{2M.25}) and (\ref{2M.26})) and passing then to limit, we find 
	\begin{align}\label{2M.33}
		&\int_0^T\left[-\langle u(t), v'(t)\rangle  +\nu(\nabla u(t),\nabla v(t))+\eta\langle(K*\nabla u(t)),\nabla v(t)\rangle+\alpha\left((u(t))^{\delta}\sum_{i=1}^2\frac{\partial u}{\partial x_i}(t),v(t)\right)\right]\mathrm{~d}t\nonumber\\&=\int_0^T\left[\beta(u(t)(1-(u(t))^{\delta})((u(t))^{\delta}-\gamma),v(t))+(f(t),v(t))\right]\mathrm{~d}t+(u_0,v(0)),
	\end{align} 
	since $u_m(0)\to u_0$ as $m\to\infty$ in $\L^2(\Omega)$. As $v(0)$ is arbitrary, comparing (\ref{2M.31}) and (\ref{2M.33}), one can easily conclude that $u(0)=u_0$ in $\L^2(\Omega)$. 
	\vskip 2mm
	\noindent\textbf{Step 6:} \emph{Energy equality.} As $u\in \L^{2}(0,T;\H_0^1(\Omega))\cap\L^{2(\delta+1)}(0,T;\L^{2(\delta+1)}(\Omega))$ with $\partial_tu\in\L^{\frac{2(\delta+1)}{2\delta+1}}(0,T;\newline\H^{-1}(\Omega)+\L^{\frac{2(\delta+1)}{2\delta+1}}(\Omega))$ using an argument similar to Theorem 7.2 (\cite{JCR} (Exercise 8.2)) we have $u\in \mathrm{C}^0([0,T];\L^2(\Omega))$ and the following energy equality holds
	\begin{align*}
		&\|u(t)\|_{\L^2}^2+2\nu\int_0^t\|\nabla u(s)\|_{\L^2}^2\mathrm{~d} s + 2\eta\int_0^t(K*\nabla u(s), \nabla u(s)) \mathrm{~d} s+ 2\beta\gamma \int_{0}^{t}\|u(s)\|^2_{\L^2}\mathrm{~d} s \nonumber\\& +2\beta\int_0^T\|u_m(t)\|_{\L^{2(\delta+1)}}^{2(\delta+1)}\mathrm{~d}s = \|u_0\|_{\L^2}^2 + 2\beta\int_{0}^t(u^{\delta+1}(s),u(s)) \mathrm{~d} s+\int_0^t\langle f(s), u(s)\rangle \mathrm{~d}s,
	\end{align*} 
	for all $t\in[0,T].$ This proves the existence of weak solution. 
\end{proof}

\begin{remark}\label{r2}
	If we take the initial data $u_0\in\L^{2(\delta+1)}(\Omega)$, and the external forcing, $f\in\L^2(0,T;\L^2(\Omega))$, then we can obtain the $\L^{2(\delta+1)}$-energy estimate as follows:
	
	Multiplying (\ref{8}) by $d_m^k(t)|u_m(t)|^{2\delta}$, and then summing for $k=1,\ldots,m$, we obtain for a.e. $t\in[0,T]$
	\begin{align*}
		&(\partial_t u_m(t),|u_m(t)|^{2\delta}u_m(t))+\nu(2\delta+1)(\nabla u_m(t),|u_m(t)|^{2\delta}\nabla u_m(t))\nonumber\\&\quad +\eta(2\delta+1)((K*\nabla u_m)(t),|u_m(t)|^{2\delta}\nabla u_m(t))+\alpha\left((u_m(t))^{\delta}\sum_{i=1}^{d}\frac{\partial u_m}{\partial x_i}(t),|u_m(t)|^{2\delta}u_m(t)\right)\nonumber\\&= \beta(u_m(t)(1-(u_m(t))^{\delta})((u_m(t))^{\delta}-\gamma,|u_m(t)|^{2\delta}u_m(t))+(f,|u_m(t)|^{2\delta}u_m(t)).
	\end{align*} 
	Now, using (\ref{7a}), the Cauchy-Schwarz, H\"older's and Young's inequalities, we arrive at
	\begin{align*}
		&\frac{1}{2(\delta+1)}\frac{\mathrm{~d}}{\mathrm{~d}t}\|u_m(t)\|_{\L^{2(\delta+1)}}^{2(\delta+1)}+\nu(2\delta+1)\||u_m(t)|^{\delta}\nabla u_m(t)\|_{\L^2}^2\nonumber\\&\quad+\eta(2\delta+1)((K*\nabla u_m)(t),|u_m(t)|^{2\delta}\nabla u_m(t))+\beta\|u_m(t)\|_{\L^{4\delta+2}}^{4\delta+2}+\beta\gamma\|u_m(t)\|_{\L^{2(\delta+1)}}^{2(\delta+1)}\nonumber\\&\leq\frac{\beta}{2}\|u_m(t)\|_{\L^{4\delta+2}}^{4\delta+2}+\beta(1+\gamma)^2\|u_m(t)\|_{\L^{2(\delta+1)}}^{2(\delta+1)}+\frac{1}{\beta}\|f(t)\|_{\L^2}^2,
	\end{align*}
	for a.e. $t\in[0,T]$.	Integrating the above inequality from $0$ to $t$ and using Lemma \ref{l2} and  Remark \ref{r1}, we have
	\begin{align*}
		&\frac{1}{2(\delta+1)}\|u_m(t)\|_{\L^{2(\delta+1)}}^{2(\delta+1)}+\nu(2\delta+1)\int_0^t\||u_m(s)|^{\delta}\nabla u_m(s)\|_{\L^2}^2\mathrm{~d}s+\frac{\beta}{2}\int_0^t\|u_m(s)\|_{\L^{4\delta+2}}^{4\delta+2}\mathrm{~d}s\nonumber\\&\leq\frac{1}{2(\delta+1)}\|u_0\|_{\L^{2(\delta+1)}}^{2(\delta+1)}+\frac{1}{\beta}\int_0^t\|f(s)\|_{\L^{2}}^2\mathrm{~d}s+\beta(1+\gamma+\gamma^2)\int_0^t\|u_m(s)\|_{\L^{2(\delta+1)}}^{2(\delta+1)}\mathrm{~d}s,
	\end{align*}
	for all $t\in[0,T]$. Applying Gronwall's inequality, we get 
	\begin{align*}
		&\|u_m(t)\|_{\L^{2(\delta+1)}}^{2(\delta+1)}\leq \left(\|u_0\|_{\L^{2(\delta+1)}}^{2(\delta+1)}+\frac{2(\delta+1)}{\beta}\int_0^T\|f(t)\|_{\L^{2}}^2\mathrm{~d}t\right)e^{\beta(1+\gamma+\gamma^2)T},
	\end{align*}
	for all $t\in[0,T]$. Thus,  it is immediate that 
	\begin{align}\label{2M.20}
		\sup_{0\leq t\leq T}\|u_m(t)\|_{\L^{2(\delta+1)}}^{2(\delta+1)}&+\nu(2\delta+1)(2(\delta+1))\int_0^T\||u_m(t)|^{\delta}\nabla u_m(t)\|_{\L^2}^2\mathrm{~d}t+\frac{\beta (2(\delta+1))}{2}\int_0^T\|u_m(t)\|_{\L^{4\delta+2}}^{4\delta+2}\mathrm{~d}t\nonumber\\&\leq \left(\|u_0\|_{\L^{2(\delta+1)}}^{2(\delta+1)}+\frac{2(\delta+1)}{\beta}\int_0^T\|f(t)\|_{\L^{2}}^2\mathrm{~d}t\right)e^{2\beta(1+\gamma+\gamma^2)T},
	\end{align}
	and the right hand side of the above inequality is independent of $m$. Since the right hand side of \eqref{2M.20} is independent of $m$, an application of the Banach-Alaoglu theorem yields $u\in\mathrm{L}^{\infty}(0,T;\L^{2(\delta+1)}(\Omega))\cap\mathrm{L}^{2(2\delta+1)}(0,T;\mathrm{L}^{2(2\delta+1)}(\Omega))$. 
\end{remark}

We will now discuss a result on the uniqueness of the weak solution that we have obtained in Theorem \ref{EWS}. 
\begin{theorem}[Uniqueness of weak solution]\label{t3}
	For $u_0\in\L^{d\delta}(\Omega)$, for $d=2,3$,  the weak solution to the system (\ref{GBHE}) obtained in Theorem \ref{EWS} is unique. 
\end{theorem}
\begin{proof}
	Let $u_1(\cdot)$ and $u_2(\cdot)$ be two weak solutions of the system (\ref{GBHE}) with the same external forcing $f(\cdot,\cdot)$ and initial data $u_0$ with the regularity given in the statement. Then, $w=u_1-u_2$ satisfies: 
	\begin{equation}\label{34}
		\left\{\begin{aligned}
			&\nonumber\frac{\partial w(x,t)}{\partial t}+\alpha \left[u_1(x,t)^{\delta}\sum_{i=1}^d\frac{\partial u_1(x,t)}{\partial x_i}-u_2(x,t)^{\delta}\sum_{i=1}^d\frac{\partial u_2(x,t)}{\partial x_i}\right]-\nu\Delta w(x,t)-\eta\Big(K*\Delta w\Big)(x,t)\\&=\beta \left[u_1(x,t)(1-u_1(x,t)^{\delta})(u_1(x,t)^{\delta}-\gamma)-u_2(x,t)(1-u_2(x,t)^{\delta})(u_2(x,t)^{\delta}-\gamma)\right] , \\
			&w(x,t)=0, \ x \in \partial\Omega ,\ t\in(0,T), \\
			&w(x,0)=0,
		\end{aligned}\right.
	\end{equation}
	for $(x,t)\in\Omega\times(0,T)$ in the weak sense. One can write the above system in the weak sense as 
	\begin{align}\label{235}
		\langle w' ,v\rangle+\alpha (B(u_1)-B(u_2),v)+\nu(\nabla w,\nabla v)+\eta((K*\nabla w),\nabla v)=\beta\langle c(u_1)-c(u_2),v\rangle,
	\end{align}
	for all $v\in \H_0^1(\Omega)\cap \L^{2(\delta+1)}(\Omega)$ and a.e. $t\in[0,T]$. Using $v=w$ in \eqref{235}, we find 
	\begin{align}\label{23}
		\frac{1}{2}\frac{\mathrm{~d}}{\mathrm{~d}t}\|w\|_{\L^2}^2+\nu\|\nabla w\|_{\L^2}^2+\eta( (K*\nabla w),\nabla w)= -\alpha(B(u_1)-B(u_2),w)+\beta\langle c(u_1)-c(u_2),w\rangle.
	\end{align} 
	It can be easily seen that
	\begin{align}\label{24}
		\beta(c(u_1)-c(u_2),w)&=\beta\left[(u_1(1-u_1^{\delta})(u_1^{\delta}-\gamma)-u_2(1-u_2^{\delta})(u_2^{\delta}-\gamma),w)\right]\nonumber\\&=-\beta\gamma\|w\|_{\L^2}^2-\beta(u_1^{2\delta+1}-u_2^{2\delta+1},w)+\beta(1+\gamma)(u_1^{\delta+1}-u_2^{\delta+1},w).
	\end{align}
	Estimating the term $-\beta(u_1^{2\delta+1}-u_2^{2\delta+1},w)$ from (\ref{24}) as
	\begin{align}\label{25}
		-\beta(u_1^{2\delta+1}-u_2^{2\delta+1},w)&= -\beta(u_1^{2\delta},w^2) -\beta(u_2^{2\delta},w^2)-\beta(u_1^{2\delta}u_2-u_2^{2\delta}u_1,w)\nonumber\\&=-\beta\|u_1^{\delta}w\|_{\L^2}^2-\beta\|u_2^{\delta}w\|_{\L^2}^2-\beta(u_1u_2,u_1^{2\delta}+u_2^{2\delta})+\beta(u_1^2,u_2^{2\delta})+\beta(u_2^2,u_1^{2\delta})\nonumber\\&=-\frac{\beta}{2}\|u_1^{\delta}w\|_{\L^2}^2-\frac{\beta}{2}\|u_2^{\delta}w\|_{\L^2}^2-\frac{\beta}{2}(u_1^{2\delta}-u_2^{2\delta},u_1^2-u_2^2)\nonumber\\&\leq -\frac{\beta}{2}\|u_1^{\delta}w\|_{\L^2}^2-\frac{\beta}{2}\|u_2^{\delta}w\|_{\L^2}^2.
	\end{align}
	We estimate the term $\beta(1+\gamma)(u_1^{\delta+1}-u_2^{\delta+1},w)$ from (\ref{24}), using Taylor's formula, H\"older's and Young's inequalities as
	\begin{align}\label{26}
		\beta(1+\gamma)(u_1^{\delta+1}-u_2^{\delta+1},w)&=\beta(1+\gamma)(\delta+1)((\theta_1u_1+(1-\theta)u_2)^{\delta}w,w)\nonumber\\&\leq \beta(1+\gamma)(\delta+1)2^{\delta-1}(\|u_1^{\delta}w\|_{\L^2}+\|u_2^{\delta}w\|_{\L^2})\|w\|_{\L^2}\nonumber\\&\leq\frac{\beta}{4}\|u_1^{\delta}w\|_{\L^2}^2+\frac{\beta}{4}\|u_2^{\delta}w\|_{\L^2}^2+\frac{\beta}{2}2^{2\delta}(1+\gamma)^2(\delta+1)^2\|w\|_{\L^2}^2.
	\end{align}
	Combining (\ref{25})-(\ref{26}) and substituting it in (\ref{24}), we obtain 
	\begin{align}\label{27}
		&\beta\left[(u_1(1-u_1^{\delta})(u_1^{\delta}-\gamma)-u_2(1-u_2^{\delta})(u_2^{\delta}-\gamma),w)\right]\nonumber\\&\leq -\beta\gamma\|w\|_{\L^2}^2-\frac{\beta}{4}\|u_1^{\delta}w\|_{\L^2}^2-\frac{\beta}{4}\|u_2^{\delta}w\|_{\L^2}^2+\frac{\beta}{2}2^{2\delta}(1+\gamma)^2(\delta+1)^2\|w\|_{\L^2}^2.
	\end{align}
	On the other hand, we derive a bound for $-\alpha\langle B(u_1)-B(u_2),w\rangle$, using Taylor's formula, H\"older's, Ladyzhenskaya's and Young's inequalities as 
	\begin{align}\label{30}
		-\alpha\langle B(u_1)-B(u_2),w\rangle&=\frac{\alpha}{\delta+1} \left((u_1^{\delta+1}-u_2^{\delta+1})\left(\begin{array}{c}1\\\vdots\\1\end{array}\right),\nabla w\right)\nonumber\\&=\alpha\left((u_1-u_2)(\theta u_1+(1-\theta) u_2)^{\delta}\left(\begin{array}{c}1\\\vdots\\1\end{array}\right),\nabla w\right)\nonumber\\&\leq 2^{\delta-1} \alpha(\|u_1\|^{\delta}_{\L^{4\delta}}+\|u_2\|^{\delta}_{\L^{4\delta}})\|w\|_{\L^{4}}\|\nabla w\|_{\L^2}\nonumber\\&\leq2^{\delta-1} \alpha(\|u_1\|^{\delta}_{\L^{4\delta}}+\|u_2\|^{\delta}_{\L^{4\delta}})\|w\|_{\L^{2}}^{\frac{4-d}{4}}\|\nabla w\|_{\L^2}^{\frac{4+d}{4}}\nonumber\\&\leq \frac{\nu}{2}\|\nabla w\|_{\L^2}^2+C(\alpha,\nu)\left(\|u_1\|^{\frac{8\delta}{4-d}}_{\L^{4\delta}}+\|u_2\|^{\frac{8\delta}{4-d}}_{\L^{4\delta}}\right)\|w\|_{\L^2}^2.
	\end{align}
	where $C(\alpha, \nu) = \left(\frac{4+d}{4\nu}\right)^{\frac{4+d}{4-d}}\left(\frac{4-d}{8}\right)(2^{\delta-1}\alpha)^{\frac{8}{4-d}}$. Combining (\ref{27}) and (\ref{30}), we have 
	\begin{align}\label{332}
		&\frac{\mathrm{~d}}{\mathrm{~d}t}\|w(t)\|_{\L^2}^2+\nu\|\nabla w(t)\|_{\L^2}^2+\frac{\beta}{2}\|u_1^{\delta}(t)w(t)\|_{\L^2}^2+\frac{\beta}{2}\|u_2^{\delta}(t)w(t)\|_{\L^2}^2+\beta\gamma\|w(t)\|_{\L^2}^2\nonumber\\&\leq \beta 2^{2\delta+1}(1+\gamma)^2(\delta+1)^2\|w(t)\|_{\L^2}^2+C(\alpha,\nu)\left(\|u_1(t)\|^{\frac{8\delta}{4-d}}_{\L^{4\delta}}+\|u_2(t)\|^{\frac{8\delta}{4-d}}_{\L^{4\delta}}\right)\|w(t)\|_{\L^2}^2,
	\end{align}
	for a.e. $t\in[0,T]$.	An application of Gronwall's inequality in (\ref{332}) gives for all $t\in[0,T]$
	\begin{align}\label{333}
		\|w(t)\|_{\L^2}^2\leq\|w(0)\|_{\L^2}^2e^{ \beta 2^{2\delta}(1+\gamma)^2(\delta+1)^2T}\exp\left\{C(\alpha,\nu)\int_0^T\left(\|u_1(t)\|^{\frac{8\delta}{4-d}}_{\L^{4\delta}}+\|u_2(t)\|^{\frac{8\delta}{4-d}}_{\L^{4\delta}}\right)\mathrm{~d} t\right\}.
	\end{align}
	For $d=2$ and $u_0\in\L^{2\delta}(\Omega)$, from Remark \ref{r2}, it is clear that $u_1,u_2\in\mathrm{L}^{4\delta}(0,T;\L^{4\delta}(\Omega))$, and hence the uniqueness follows. For $d=3,$ $u_0\in\L^{3\delta}(\Omega)$, the integral appearing in the exponential is finite by using interpolation inequality as 
	\begin{align*}
		\int_0^T\|u_1(t)\|_{\L^{4\delta}}^{8\delta} \mathrm{~d} t \leq\sup_{t\in[0,T]}\|u_1(t)\|_{\L^{3\delta}}^{3\delta}  \int_{0}^{T}\|u_1(t)\|_{\L^{5\delta}}^{5\delta}\mathrm{~d} t< \infty.
	\end{align*}
	Hence, the uniqueness follows if we take the initial data $u_0 \in \L^{d\delta}(\Omega)$. 
\end{proof}
\begin{remark}\label{rk2.4}
	For $d=3$,  using Gagliardo-Nirenberg's inequality in \eqref{30}, one can find a sufficient criterion for the uniqueness of weak solution for $u_0\in\L^{3\delta}(\Omega)$  as $$\int_0^T\|u(t)\|_{\L^{r\delta}}^{s\delta}\d t<\infty,\ \text{ where }\ \frac{3}{r}+\frac{2}{s}=1, \ r\in(3,\infty).$$ 
\end{remark}
\begin{remark}\label{rem2.4}
	If we consider our initial data $u_0 \in \L^{2}(\Omega)$ and $\beta\nu>(2^{\delta}\alpha)^2,$ then there exists a unique weak solution to the system (\ref{GBHE}).  One can  estimate $-\alpha\langle B(u_1)-B(u_2),w\rangle$ in  \eqref{30}  using integrating by parts, Taylor's formula,  H\"older's and Young's inequalities as: 
	\begin{align}\label{28}
		-\alpha\langle B(u_1)-B(u_2),w\rangle&=\frac{\alpha}{\delta+1} \left((u_1^{\delta+1}-u_2^{\delta+1})\left(\begin{array}{c}1\\\vdots\\1\end{array}\right),\nabla w\right)\nonumber\\&=\alpha\left((u_1-u_2)(\theta u_1+(1-\theta) u_2)^{\delta}\left(\begin{array}{c}1\\\vdots\\1\end{array}\right),\nabla w\right)\nonumber\\&\leq 2^{\delta-1}\alpha\|\nabla w\|_{\L^2}\left(\|u_1^{\delta}w\|_{\L^2}+\|u_2^{\delta}w\|_{\L^2}\right)\nonumber\\&\leq\frac{\nu}{2}\|\nabla w\|_{\L^2}^2+\frac{2^{2\delta}\alpha^2}{4\nu}\|u_1^{\delta} w\|_{\L^2}^2+\frac{2^{2\delta}\alpha^2}{4\nu}\|u_2^{\delta} w\|_{\L^2}^2. 
	\end{align}
	Combining (\ref{27}) and (\ref{28}) in (\ref{23}), we find 
	\begin{align*}
		& \frac{1}{2}\frac{\mathrm{~d}}{\mathrm{~d}t}\|w(t)\|_{\L^2}^2+\frac{\nu}{2}\|\nabla w(t)\|_{\L^2}^2+\frac{\beta}{4}\|u_1^{\delta}(t)w(t)\|_{\L^2}^2+\frac{\beta}{4}\|u_2(t)^{\delta}w(t)\|_{\L^2}^2+\beta\gamma\|w(t)\|_{\L^2}^2+\eta( (K*\nabla w)(t),\nabla w(t))\nonumber\\&\leq \frac{\beta}{2}2^{2\delta}(1+\gamma)^2(\delta+1)^2\|w(t)\|_{\L^2}^2+\frac{\alpha^2}{4\nu}2^{2\delta}\|u_1^{\delta}w(t)\|_{\L^2}^2+\frac{\alpha^2}{4\nu}2^{2\delta}\|u_2(t)^{\delta}w(t)\|_{\L^2}^2. 
	\end{align*}
	From the above expression, it is immediate that 
	\begin{align*}
		&\frac{\mathrm{~d}}{\mathrm{~d}t}\|w(t)\|_{\L^2}^2+\nu\|\nabla w(t)\|_{\L^2}^2+\left(\frac{\beta}{2}-\frac{\alpha^2}{2\nu}2^{2\delta}\right)\|u_1(t)^{\delta}w(t)\|_{\L^2}^2\nonumber\\&\quad+\left(\frac{\beta}{2}-\frac{\alpha^2}{2\nu}2^{2\delta}\right)\|u_2(t)^{\delta}w(t)\|_{\L^2}^2+\eta( (K*\nabla w)(t),\nabla w(t))\nonumber\\&\leq \beta 2^{2\delta}(1+\gamma)^2(\delta+1)^2\|w(t)\|_{\L^2}^2. 
	\end{align*}
	For $\beta\nu>(2^{\delta}\alpha)^2$, integrating the above inequality from $0$ to $t$, (\ref{pk}) and then applying  of Gronwall's inequality results to
	\begin{align*}
		& \|w(t)\|_{\L^2}^2\leq\|w(0)\|_{\mathrm{L}^2}^2e^{ \beta 2^{2\delta}(1+\gamma)^2(\delta+1)^2T},
	\end{align*} 
	for all $t\in[0,T]$. Since $w_0=0$ and $u_1$ and $u_2$ are weak solutions of the system (\ref{GBHE}), uniqueness follows easily.
\end{remark}

\subsection{Regularity results}\label{sec2.4}
If we take our initial data $u_0\in\H_0^1(\Omega)\cap \L^{2(\delta+1)}(\Omega)$, and the external forcing, $f\in\L^2(0,T;\L^2(\Omega))$, where $\Omega\subset\mathbb{R}^d, d=2,3,$ is either convex, or a domain with $\mathrm{C}^2$-boundary, then we can obtain the $\H_0^1(\Omega)$-energy estimate. Under  smoothness assumptions on the domain, initial data and external forcing, let us now establish the regularity of the weak solution obtained in Theorem \ref{EWS}. 
\begin{theorem}[Regularity]\label{2M.thm3.2}
	Let $u$ be a weak solution to the system (\ref{GBHE}). If $\Omega\subset\mathbb{R}^d,d=2,3,$ is either convex, or a domain with $\mathrm{C}^2$-boundary and let $u_0\in\H_0^1(\Omega)\cap \L^{2(\delta+1)}(\Omega)$ and $f\in\L^2(0,T;\L^2(\Omega))$ be given. Then, $u$ is a strong solution with 
	\begin{align*}
		u\in\mathrm{C}([0,T];\H_0^1(\Omega))\cap\mathrm{L}^2(0,T;\H^2(\Omega))\cap\mathrm{L}^{2(\delta+1)}(0,T;\mathrm{L}^{6(\delta+1)}(\Omega)), 
	\end{align*}
	and $ \partial_tu\in\mathrm{L}^2(0,T;\L^2(\Omega))$ and the following system is satisfied:
	\begin{equation}\label{2M.336}
		\left\{
		\begin{aligned}
			\nonumber\frac{\mathrm{~d}u(t)}{\mathrm{~d}t}+\nu Au(t)+\eta ( K*Au)(s)&=-\alpha B(u(t))+\beta c(u(t))+f(t), \ \text{ in }\ \L^2(\Omega),\\
			u(0)&=u_0\in\L^2(\Omega),
		\end{aligned}
		\right.
	\end{equation}
	for a.e. $t\in[0,T]$. 
\end{theorem}
\begin{proof}
	
	\vskip 2mm
	\noindent\textbf{Step 1:} \emph{$\H_0^1$-energy estimate.}
	Let $\lambda_k$ denote the $k^{\mathrm{th}}$ eigenvalue of $-\Delta$ in $\H_0^1(\Omega)$. Multiplying the identity (\ref{8}) by $\lambda_kd_k^m(t)$ and summing it from $k=1,2,\ldots,m$, we have
	\begin{align}\label{2m.14}
	(\partial_t u_m,Au_m)-\nu(\Delta u_m, Au_m)-\eta(K*\Delta u_m, Au_m)=-\alpha (B(u_m ),Au_m )+\beta(c(u_m),Au_m)+(f,Au_m). 
	\end{align}
	Now, using integrating by parts and applying H\"older's and Young's inequalities, we find
	\begin{align*}
		\alpha |(B(u_m ),Au_m )|\leq\alpha \|B(u_m)\|_{\L^2}\|A u_m\|_{\L^2}\leq\alpha\|u_m^{\delta}\nabla  u_m\|_{\L^2}\|Au_m\|_{\L^2}\leq\frac{\nu}{4}\|Au_m\|_{\L^2}^2+\frac{\alpha^2}{\nu}\|u_m^{\delta}\nabla  u_m\|_{\L^2}^2.
	\end{align*}
	Estimating $\beta(c(u_m),Au_m)$ using H\"older's and Young's inequalities as follows:
	\begin{align*}
		&\beta(c(u_m),Au_m)\nonumber\\&=\beta((1+\gamma)u_m^{\delta+1}-\gamma u_m-u_m^{2\delta+1},Au_m)\nonumber\\
		&=\beta((1+\gamma)(\delta+1)u_m^{\delta}\nabla  u_m-\gamma\nabla  u_m-(2\delta+1)u_m^{2\delta}\nabla  u_m,\nabla  u_m)\nonumber\\
		&=\beta(1+\gamma)(\delta+1)(u_m^{\delta}\nabla  u_m,\nabla  u_m)-\beta\gamma(\nabla  u_m,\nabla  u_m)-\beta(2\delta+1)(u_m^{2\delta}\nabla  u_m,\nabla  u_m)\nonumber\\
		&\leq \beta(1+\gamma)(\delta+1)\|u_m^{\delta}\nabla  u_m\|_{\L^2}\|\nabla  u_m\|_{\L^2}-\beta\gamma\|\nabla u_m\|_{\L^2}^2-\beta(2\delta+1)\|u_m^{\delta}\nabla  u_m\|_{\L^2}^2\nonumber\\
		&\leq\frac{\beta(1+\gamma)^2(\delta+1)^2}{2(2\delta+1)}\|\nabla  u_m\|_{\L^2}^2+\frac{\beta(2\delta+1)}{2}\|u_m^{\delta}\nabla  u_m\|_{\L^2}^2 -\beta\gamma\|\nabla u_m\|_{\L^2}^2-\beta(2\delta+1)\|u_m^{\delta}\nabla  u_m\|_{\L^2}^2\nonumber\\
		&\leq \frac{\beta(1+\gamma)^2(\delta+1)^2}{2(2\delta+1)}\|\nabla  u_m\|_{\L^2}^2-\beta\gamma\|\nabla  u_m\|_{\L^2}^2-\frac{\beta(2\delta+1)}{2}\|u_m^{\delta}\nabla  u_m\|_{\L^2}^2. 
	\end{align*} 
	Combining the above estimates and substituting in (\ref{2m.14}), we get
	\begin{align*}
		& \frac{1}{2}\frac{\mathrm{~d}}{\mathrm{~d}t}\|\nabla u_m(t)\|_{\L^2}^2+\frac{\nu}{2}\|\Delta u_m(t) \|_{\L^2}^2+\frac{\beta(2\delta+1)}{2}\|u_m(t)^{\delta}\nabla  u_m(t)\|_{\L^2}^2+\eta((K*\Delta u_m)(t),\Delta u_m(t))\nonumber\\&\leq \frac{\alpha^2}{\nu}\|u_m(t)^{\delta}\nabla  u_m(t)\|_{\L^2}^2+\frac{\beta((1+\gamma^2)(\delta+1)^2+2\gamma\delta^2)}{2(2\delta+1)}\|\nabla  u_m(t)\|_{\L^2}^2+\frac{1}{\nu}\|f(t)\|_{\L^2}^2.
	\end{align*} 
	Integrating from $0$ to $t$ and using (\ref{pk}), we find
	\begin{align}\label{2M.53}
		&\|\nabla u_m(t)\|_{\L^2}^2+\nu\int_{0}^{t}\|\Delta u_m(s) \|_{\L^2}^2\mathrm{~d}s+\beta(2\delta+1)\int_{0}^{t}\|u_m(s)^{\delta}\nabla  u_m(s)\|_{\L^2}^2\mathrm{~d}s\nonumber\\&\leq\|u_0\|_{\H_0^1}^2+ \frac{2\alpha^2}{\nu}\int_{0}^{t}\|u_m(s)^{\delta}\nabla  u_m(s)\|_{\L^2}^2\mathrm{~d}s+\frac{\beta((1+\gamma^2)(\delta+1)^2+2\gamma\delta^2)}{(2\delta+1)}\int_{0}^{t}\|\nabla  u_m(s)\|_{\L^2}^2\mathrm{~d}s\nonumber\\&\quad+\frac{2}{\nu}\int_{0}^{t}\|f(s)\|_{\L^2}^2\mathrm{~d}s\nonumber\\&\leq C\left(\|u_0\|_{\H_0^1},\|f\|_{\L^2(0,T;\L^2(\Omega))},\alpha,\beta,\gamma,\delta,\nu,T\right),
	\end{align} 
	since $u_m(\cdot)$ satisfies the energy estimates (\ref{13}) and (\ref{2M.20}). Whenever $u_0\in\H_0^1(\Omega)$ and $f\in\L^2(0,T;\L^2(\Omega))$, from (\ref{2M.53}), we have
	\begin{align*}
		u_m\in\mathrm{L}^{\infty}(0,T;\H_0^1(\Omega))\cap\mathrm{L}^2(0,T;\H^2(\Omega)),
	\end{align*}
	by using the elliptic regularity (Theorem 3.1.2.1, \cite{GPi}). Using  Sobolev's inequality, $\|u\|_{\L^6} \leq C\|\nabla u||_{\L^2}$, we have $	\|u_m^{\delta+1}\|_{\L^6}^2 \leq C\|\nabla u_m^{\delta+1}||_{\L^2}^2\leq C(\delta+1)\|u_m^{\delta}\nabla u_m\|_{\L^2},$ so that 
	\begin{align*}
		\int_{0}^t\|u_m(s)\|^{2(\delta+1)}_{\L^{6(\delta+1)}}\mathrm{~d}s= \int_{0}^t\|u_m(s)^{\delta+1}\|_{\L^6}^2 \mathrm{~d}s\leq C(\delta+1)^2\int_{0}^t\|u_m(s)^{\delta}\nabla u_m(s)||_{\L^2}^2\mathrm{~d}s< \infty.
	\end{align*}
	Thus, we infer 
	\begin{align*}
		u_m\in\mathrm{L}^{2(\delta+1)}(0,T;\mathrm{L}^{6(\delta+1)}(\Omega)).
	\end{align*}
	\vskip 2mm
	\noindent\textbf{Step 2:} \emph{Estimate for time derivative.}
	For $m\geq 1$, multiplying (\ref{8}) by $d{_m^k}'(t)$ and sum it over $k=1,2,\ldots,m,$ we have 
	\begin{align}\label{2M.54}
		&\|\partial_t u_m(t)\|_{\L^2}^2+\frac{\nu}{2}\frac{\mathrm{~d}}{\mathrm{~d}t}\|\nabla u_m(t)\|_{\L^2}^2\nonumber\\&=\eta((K*\Delta u_{m})(t),\partial_{t}u_{m}(t))-\alpha\left((u_m(t))^{\delta}\sum_{i=1}^{d}\frac{\partial u_m}{\partial x_i}(t),\partial_{t}u_{m}(t)\right)\nonumber\\&\quad+\beta(u_m(t)(1-(u_m(t))^{\delta})((u_m(t))^{\delta}-\gamma),\partial_tu_m(t))+(f(t),\partial_tu_m(t))\nonumber\\&=\eta((K*\Delta u_{m})(t),\partial_{t}u_{m}(t)) -\alpha\left((u_m(t))^{\delta}\sum_{i=1}^{d}\frac{\partial u_m}{\partial x_i}(t),\partial_{t}u_{m}(t)\right)+\beta(1+\gamma)(u_m(t)^{\delta+1},\partial_tu_m(t))\nonumber\\&\quad-\beta\gamma\frac{\mathrm{~d}}{\mathrm{~d}t}\|u_m(t)\|_{\L^2}^2-\frac{\beta}{2(\delta+1)}\frac{\mathrm{~d}}{\mathrm{~d}t}\|u_m(t)\|_{\L^{2(\delta+1)}}^{2(\delta+1)}+(f(t),\partial_tu_m(t)).
	\end{align}
	Using H\"older's and Young's inequalities, we estimate the first term from the right hand side of the equality (\ref{2M.54}) as 
	\begin{align}\label{2M.55}
		\alpha\left((u_m)^{\delta}\sum_{i=1}^{d}\frac{\partial u_m}{\partial x_i},\partial_{t}u_{m}\right)\leq\frac{1}{4}\|\partial_tu_m\|_{\L^2}^2+\alpha^2\|u_m^{\delta}\nabla u_m\|_{\L^2}^2.
	\end{align}
	Similarly, we estimate $\beta(1+\gamma)(u_m^{\delta+1},\partial_tu_m)$ and $\eta(K*\Delta u_{m},\partial_{t}u_{m})$  as 
	\begin{align}
		\nonumber\beta(1+\gamma)|(u_m^{\delta+1},\partial_tu_m)|&\leq\frac{1}{4}\|\partial_tu_m\|_{\L^2}^2+\beta^2(1+\gamma)^2\|u_m\|_{\L^{2(\delta+1)}}^{2(\delta+1)},\\
		\eta(K*\Delta u_{m},\partial_{t}u_{m})&\leq\frac{\eta^2}{2}\|K*\Delta u_{m}\|_{\L^2}^2+ \frac{1}{2}\|\partial_t u_m\|_{\L^2}^2,
	\end{align}
	The final term from the right hand side of the equality (\ref{2M.54}) can be estimated as 
	\begin{align}\label{2M.5.7}
		|(f,\partial_tu_m)|\leq\|f\|_{\L^2}\|\partial_tu_m\|_{\L^2}\leq\frac{1}{4}\|\partial_tu_m\|_{\L^2}^2+\|f\|_{\L^2}^2. 
	\end{align}
	Combining (\ref{2M.55})-(\ref{2M.5.7}), substituting it in (\ref{2M.54}) and then integrating from $0$ to $t$, we get
	\begin{align*}
		&2\nu\|\nabla u_m(t)\|_{\L^2}^2+\frac{2\beta}{\delta+1}\|u_m(t)\|_{\L^{2(\delta+1)}}^{2(\delta+1)}+4\beta\gamma\|u_m(t)\|_{\L^2}^2+\frac{1}{2}\int_0^t\|\partial_tu_m(s)\|_{\L^2}^2\mathrm{~d}s\nonumber\\&\leq 2\nu\|u_0\|_{\H_0^1}^2+\frac{2\beta}{\delta+1}\|u_0\|_{\L^{2(\delta+1)}}^{2(\delta+1)}+4\beta\gamma\|u_0\|_{\L^2}^2+4\int_0^t\|f(s)\|_{\L^2}^2\mathrm{~d}s+4\alpha^2\int_0^t\|u_m^{\delta}(s)\nabla u_m(s)\|_{\L^2}^2\mathrm{~d}s\nonumber\\&\quad+\frac{\eta^2C}{2}\int_0^t\|\Delta u_m(s)\|_{\L^2}^2\mathrm{~d}s+4\beta^2(1+\gamma)^2\int_0^t\|u_m(s)\|_{\L^{2(\delta+1)}}^{2(\delta+1)}\mathrm{~d}s.
	\end{align*}
	Applying Gronwall's inequality, we arrive at
	\begin{align}\label{2M.5.9}
		&2\nu\|\nabla u_m(t)\|_{\L^2}^2+\frac{2\beta}{\delta+1}\|u_m(t)\|_{\L^{2(\delta+1)}}^{2(\delta+1)}+4\beta\gamma\|u_m(t)\|_{\L^2}^2+\frac{1}{2}\int_0^t\|\partial_tu_m(s)\|_{\L^2}^2\mathrm{~d}s\nonumber\\&\leq\left(2\nu\|u_0\|_{\H_0^1}^2+\frac{2\beta}{\delta+1}\|u_0\|_{\L^{2(\delta+1)}}^{2(\delta+1)}+4\beta\gamma\|u_0\|_{\L^2}^2+4\int_0^T\|f(t)\|_{\L^2}^2\mathrm{~d}t+4\alpha^2\int_0^T\|u_m^{\delta}(t)\nabla u_m(t)\|_{\L^2}^2\mathrm{~d}t\right)\nonumber\\&\quad\times\exp\left({4\beta(1+\gamma)^2(1+\delta)T}\right)e^{\frac{\eta^2C}{2\nu}T},
	\end{align}
	for all $t\in[0,T]$. 
	Using the estimate given in (\ref{2M.53}), we obtain that the right hand side of (\ref{2M.5.9}) is independent of $m$. Thus, it is immediate that $\partial_tu_m\in\mathrm{L}^2(0,T;\L^2(\Omega))$.  Using arguments similar to Steps 4 and 5 in Theorem \ref{EWS}, we obtain the existence of a $\tilde{u}(\cdot)$ such that 
	\begin{align*}
		\tilde{u}&\in\mathrm{L}^{\infty}(0,T;\H_0^1(\Omega)\cap\mathrm{L}^{2(\delta+1)}(\Omega))\cap\mathrm{L}^2(0,T;\H^2(\Omega))\cap\mathrm{L}^{2(\delta+1)}(0,T;\mathrm{L}^{6(\delta+1)}(\Omega)) \ \text{ and }\\ \partial_t\tilde{u}&\in\mathrm{L}^2(0,T;\L^2(\Omega)), 
	\end{align*}
	so that we get $\tilde{u}\in\mathrm{C}([0,T];\H_0^1(\Omega))$ and is a strong solution to the system (\ref{GBHE}). The uniqueness follows from the estimate (\ref{333}) and we have $\tilde{u}=u$.
\end{proof}

\begin{theorem}\label{2M.thm33}
	Let  $1\leq \delta < \infty$ for $d=2$ and, $1\leq \delta \leq 2$ for $d=3$ and assume that $u_0\in\H^2(\Omega)\cap\H_0^1(\Omega)$ 
	\begin{enumerate}[label=(\roman*)]\label{reg}
		\item If $f\in\H^1(0,T;\L^2(\Omega))$, then, we have $\partial_tu\in\L^{\infty}(0,T;\L^2(\Omega))\cap\L^2(0,T;\H_0^1(\Omega)).$
		\item If $f\in\H^1(0,T;\L^2(\Omega))\cap \L^2(0,T;\H^1(\Omega)),$ then $u \in \L^{\infty}(0,T;\H^2(\Omega))$  .
	\end{enumerate}
	
\end{theorem}
\begin{proof}
	We divide the proof into the following steps: 
	\vskip 0.2mm 
	\textbf{Step 1.} $\partial_tu\in\mathrm{L}^{\infty}(0,T;\mathrm{L}^2(\Omega)$, and $\partial_t\nabla u\in\mathrm{L}^{2}(0,T;\mathrm{L}^2(\Omega)$. 
	
	Let  $1\leq \delta < \infty$ for $d=2$ and, $1\leq \delta \leq 2$ for $d=3$. 	Differentiating \eqref{8} with respect to $t$, we find 
	\begin{align}\label{2M.62}
		&(\partial_{tt}u_m(t),w_k)+\nu(\partial_{t}\nabla u_m(t),\nabla w_k)\nonumber\\&+\alpha\sum_{i=1}^{d}\left((u_m(t))^{\delta}\partial_{tx_i} u_m(t)+\delta u_m(t)^{\delta-1}\partial_t u_m(t)\partial{x_i} u_m(t),w_k\right)-\eta((K*\partial_t \Delta u_m)(t)+K(t)\Delta u_m(0), w_k)\nonumber\\&=\beta((\gamma+1)(\delta+1)u_m(t)^{\delta}\partial_tu_m(t)-\gamma\partial_tu_m(t)-(2\delta+1)u_m(t)^{2\delta}\partial_tu_m(t),w_k)+(\partial_tf(t),w_k).
	\end{align}
	Multiplying \eqref{2M.62} with $d{_m^k}'(\cdot)$  and summing it over $k=1,2,\ldots,m$, we obtain 
	\begin{align}\label{2M.63}
		\nonumber&\frac{1}{2}\frac{\mathrm{~d}}{\mathrm{~d}t}\|\partial_tu_m(t)\|_{\L^2}^2+\nu\|\partial_{t}\nabla u_m(t)\|_{\L^2}^2+\beta\gamma\|\partial_tu_m(t)\|_{\L^2}^2+\beta(2\delta+1)\|u_m(t)^{\delta}\partial_tu_m(t)\|_{\L^2}^2\\&\nonumber\quad+\eta((K*\partial_t \nabla u_m)(t),\partial_t\nabla u_m(t))\nonumber\\&=-\alpha\sum_{i=1}^{d}\left((u_m(t))^{\delta}\partial_{tx_i}u_m(t),\partial_tu_m\right)+\beta(\gamma+1)(\delta+1)(u_m(t)^{\delta}\partial_tu_m(t),\partial_tu_m(t))\nonumber\\&\quad+(\partial_tf(t),\partial_tu_m(t)) + \eta(K(t)\Delta u_m(0),\partial_t u_m(t)).
	\end{align} 
	We estimate the term $-\alpha\sum_{i=1}^{d}((u_m)^{\delta}\partial_{tx_i}u_m,\partial_tu_m)$ using integration by parts, H\"older's, Ladyzhenskaya's inequality and Young's inequalities as 
	\begin{align}\label{2M.64}
		\nonumber-&\sum_{i=1}^{d}\alpha\left((u_m)^{\delta}\partial_{tx_i}u_m,\partial_tu_m\right)\\\nonumber&= -\frac{1}{2}\sum_{i=1}^{d}\alpha\left(u_m^{\delta}, \frac{\partial }{\partial x_i}\left(\partial_tu_m\right)^2\right) = \frac{\alpha\delta}{2}\sum_{i=1}^{d}\left(u_m^{\delta-1}\frac{\partial u_m }{\partial x_i}, \left(\partial_tu_m\right)^2\right)\\\nonumber& =\frac{\alpha\delta}{2}\sum_{i=1}^{d} \int_{\Omega} u_m^{\delta-1} \left(\frac{\partial u_m }{\partial x_i}\right)^{\frac{\delta-1}{\delta}}\left(\frac{\partial u_m }{\partial x_i}\right)^{\frac{1}{\delta}}\left(\partial_tu_m\right)^2 \mathrm{~d}x \\\nonumber&\leq \frac{\alpha\delta}{2}\sum_{i=1}^{d} \left(\int_{\Omega} u_m^{2\delta} \left(\frac{\partial u_m }{\partial x_i}\right)^2\mathrm{~d}x\right)^{\frac{\delta-1}{2\delta}}
		\left(\int_{\Omega}\left(\frac{\partial u_m }{\partial x_i}\right)^{\frac{2}{\delta+1}}\left(\partial_tu_m\right)^{\frac{4\delta}{\delta+1}} \mathrm{~d}x\right)^{\frac{\delta+1}{2\delta}}\\\nonumber& \leq \frac{\alpha\delta}{2}\|u_m^{\delta}\nabla u_m\|_{\L^2}^{\frac{\delta-1}{\delta}} \|(\nabla u_m)^{\frac{1}{\delta+1}}\left(\partial_tu_m\right)^{\frac{2\delta}{\delta+1}}\|_{\L^2} ^{\frac{\delta+1}{\delta}}\\\nonumber& \leq \frac{\alpha\delta}{2}\|u_m^{\delta}\nabla u_m\|_{\L^2}^{\frac{\delta-1}{\delta}} \|\nabla u_m\|_{\L^2}^{\frac{1}{\delta}}\|\partial_tu_m\|_{\L^4} ^{2} \\\nonumber &\leq \frac{\alpha\delta}{\sqrt{2}} \|u_m^{\delta}\nabla u_m\|_{\L^2}^{\frac{\delta-1}{\delta}} \|\nabla u_m\|_{\L^2}^{\frac{1}{\delta}}\|\partial_tu_m\|_{\L^2} ^{\frac{4-d}{2}}\|\partial_t\nabla u_m\|_{\L^2} ^{\frac{d}{2}}\\&
		\leq \frac{\nu}{2}  \|\partial_t\nabla u_m\|_{\L^2} ^{2} + \left(\frac{4-d}{4}\right)\left(\frac{d}{2\nu}\right)^{\frac{d}{4-d}}\left(\frac{\alpha\delta}{\sqrt{2}}\right)^{\frac{4}{4-d}} \|u_m^{\delta}\nabla u_m\|_{\L^2}^{\frac{4(\delta-1)}{(4-d)\delta}} \|\nabla u_m\|_{\L^2}^{\frac{4}{\delta(4-d)}}\|\partial_tu_m\|_{\L^2} ^{2}.
	\end{align}
	We estimate $\eta\left|\left(K(t)\Delta u_m(0),\partial_t u_m\right)\right|$ as 
	\begin{align}
		\nonumber\eta\left|\left(K(t)\Delta u_m(0),\partial_t u_m(t)\right)\right|\leq \eta |K(t)|\|\Delta u_m(0)\|_{\L^2}\|\partial_t u_m(t)\|_{\L^2}\leq \frac{\eta}{2}|K(t)|(\|\Delta u_m(0)\|_{\L^2}^2+\|\partial_t u_m(t)\|^2_{\L^2}).
	\end{align}
	Similarly, we estimate the terms $\beta(\gamma+1)(\delta+1)(u_m^{\delta}\partial_tu_m,\partial_tu_m)$ and $(\partial_tf,\partial_tu_m)$ as 
	\begin{align}\label{2M.65}
		\nonumber	\beta(\gamma+1)(\delta+1)|(u_m^{\delta}\partial_tu_m,\partial_tu_m)|&\leq \beta(\gamma+1)(\delta+1)\|u_m^{\delta}\partial_tu_m\|_{\L^{2}}\|\partial_tu_m\|_{\L^2}\\\nonumber&\leq \frac{\beta(2\delta+1)}{2}\|u_m^{\delta}\partial_tu_m\|_{\L^{2}}^2+\frac{\beta(\gamma+1)^2(\delta+1)^2}{2(2\delta+1)}\|\partial_tu_m\|_{\L^2}^2 ,\\
		|(\partial_tf,\partial_tu_m)|\leq \|\partial_tf\|_{\L^2}\|\partial_tu_m\|_{\L^2}&\leq\frac{1}{2\beta\gamma}\|\partial_tf\|_{\L^2}^2+\frac{\beta\gamma}{2}\|\partial_tu_m\|_{\L^2}^2. 
	\end{align}
	Substituting the estimates \eqref{2M.64}-\eqref{2M.65} in \eqref{2M.63}, we have 
		\begin{align*}
			\nonumber&\frac{\mathrm{~d}}{\mathrm{~d}t}\|\partial_tu_m(t)\|_{\L^2}^2+\nu\|\partial_{t}\nabla u_m(t)\|_{\L^2}^2+\beta\gamma\|\partial_tu_m(t)\|_{\L^2}^2+\beta(2\delta+1)\|u_m(t)^{\delta}\partial_tu_m(t)\|_{\L^2}^2\\&\nonumber\quad+2\eta((K*\partial_t \nabla u_m)(t),\partial_t\nabla u_m(t))\\&\nonumber\leq \eta|K(t)|\|\Delta u_m(0)\|_{\L^2}^2 + \frac{1}{\beta\gamma}\|\partial_tf(t)\|_{\L^2}^2 \\&  \quad+ \left(C(\alpha,\nu,\delta) \|u_m(t)^{\delta}\nabla u_m(t)\|_{\L^2}^{\frac{4(\delta-1)}{(4-d)\delta}} \|\nabla u_m(t)\|_{\L^2}^{\frac{4}{\delta(4-d)}}+ \eta|K(t)| +C(\beta,\gamma,\delta)\right)\|\partial_tu_m(t)\|_{\L^2}^2, 
		\end{align*} 
		where $C(\alpha,\nu,\delta)=\left(\frac{4-d}{2}\right)\left(\frac{d}{2\nu}\right)^{\frac{d}{4-d}}\left(\frac{\alpha\delta}{\sqrt{2}}\right)^{\frac{4}{4-d}}$ and $C(\beta,\gamma,\delta)= \frac{\beta(\gamma+1)^2(\delta+1)^2}{(2\delta+1)}$.
		\begin{align*} 
			&\|\nonumber\partial_tu_m(t)\|_{\L^2}^2+\nu\int_0^t\|\partial_{t}\nabla u_m(s)\|_{\L^2}^2\mathrm{~d}s+\beta\gamma\int_0^t\|\partial_tu_m(s)\|_{\L^2}^2\mathrm{~d}s +\beta(2\delta+1)\int_0^t\|u_m(s)^{\delta}\partial_tu_m(s)\|_{\L^2}^2\mathrm{~d}s\\\nonumber&\leq \|\partial_tu_m(0)\|_{\L^2}^2 +\eta\|\Delta u_m(0)\|_{\L^2}^2 \int_0^t|K(s)|\mathrm{~d}s + \frac{1}{\beta\gamma}\int_0^t\|\partial_tf(s)\|_{\L^2}^2\mathrm{~d}s \\& \quad + \int_0^t\left(C(\alpha,\nu,\delta) \|u_m(s)^{\delta}\nabla u_m(s)\|_{\L^2}^{\frac{4(\delta-1)}{(4-d)\delta}} \|\nabla u_m(s)\|_{\L^2}^{\frac{4}{\delta(4-d)}}+ \eta |K(s)| +C(\beta,\gamma,\delta)\right)\|\partial_tu_m(s)\|_{\L^2}^2\mathrm{~d}s.
		\end{align*}
		As $K(t) \in \L^1(0,T)$ and note that $\|u_m(0)\|_{\H^2}\leq C\|u_0\|_{\H^2}$ (see page 385, \cite{LCE}) and an application of Gronwall's inequality yields 
		\begin{align}\label{2M.68} 
			&\|\partial_tu_m(t)\|_{\L^2}^2+\nu\int_0^t\|\partial_{t}\nabla u_m(s)\|_{\L^2}^2\mathrm{~d}s+\beta\gamma\int_0^t\|\partial_tu_m(s)\|_{\L^2}^2\mathrm{~d}s\nonumber\\&\qquad +\beta(2\delta+1)\int_0^t\|u_m(s)^{\delta}\partial_tu_m(s)\|_{\L^2}^2\mathrm{~d}s\nonumber\\&\leq\left(  \|\partial_tu_m(0)\|_{\L^2}^2 + \eta\| u_0\|_{\H^2}\int_0^t|K(s)|\mathrm{~d}s + \frac{1}{\beta\gamma}\int_0^t\|\partial_tf(s)\|_{\L^2}^2\mathrm{~d}s\right)\nonumber\\&\quad\times\exp\left(\int_{0}^{t}\left[C(\alpha,\nu,\delta) \|u_m(s)^{\delta}\nabla u_m(s)\|_{\L^2}^{\frac{4(\delta-1)}{(4-d)\delta}} \|\nabla u_m(s)\|_{\L^2}^{\frac{4}{\delta(4-d)}}+ \eta\|u_0\|_{\H^2} |K(s)| +C(\beta,\gamma,\delta)\right]\right)\nonumber\\&\leq C\left(\|u_0	\|_{\H^2},\|f\|_{\H^{1}(0,T;\L^2(\Omega))},\alpha,\nu,\beta,\gamma,\delta,T\right),
		\end{align}
		for all $t\in[0,T]$, where the first term in exponential is bounded , provided $1\leq \delta <\infty$, for $d=2$ and, $1\leq \delta \leq 2,$ for $d=3$.  Taking supremum over time $0\leq t\leq T$, we have 
		\begin{align*}
			\sup_{0\leq t\leq T}\|\partial_tu_m(t)\|_{\L^2}^2+\nu\int_0^T\|\partial_{t}\nabla u_m(s)\|_{\L^2}^2\mathrm{~d}s\leq C\left(\|u_0\|_{\H^2},\|f\|_{\H^{1}(0,T;\L^2(\Omega))},\alpha,\beta,\gamma,\delta,\nu,T\right).
		\end{align*}
		\vskip 0.2mm 
		\textbf{Step 2.} $\nabla(u^{\delta+1})\in\mathrm{L}^{\infty}(0,T;\mathrm{L}^2(\Omega))$.

		In (\ref{8}), if we take $w_k = -u_m^{2\delta}\Delta u_m$, then we have  for a.e. $t\in[0,T]$
		\begin{align}\label{2r.14}
			\nonumber
			&(\partial_t u_m(t),-u_m(t)^{2\delta}\Delta u_m(t))+\nu(\Delta u_m(t),u_m(t)^{2\delta}\Delta u_m(t))+\eta((K*\Delta u_m)(t),u_m(t)^{2\delta}\Delta u_m(t))\\&=\alpha (B(u_m(t) ),u_m(t)^{2\delta}\Delta u_m(t) )-\beta(c(u_m(t)),u_m(t)^{2\delta}\Delta u_m(t))-(f(t),u_m(t)^{2\delta}\Delta u_m(t)). 
		\end{align}
		Using integration by parts, one can estimate the time derivative term as 
		\begin{align*}
			\nonumber(\partial_t u_m,-u_m^{2\delta}\Delta u_m) &= -\frac{1}{2\delta+1}\left(\partial_t u_m^{2\delta+1},\Delta u_m\right) =\frac{1}{2\delta+1}\left(\nabla\left(\partial_t u_m^{2\delta+1}\right),\nabla u_m\right)\\&\nonumber =\left(\partial_t\left( u_m^{2\delta}\nabla u_m\right),\nabla u_m\right)= \frac{\mathrm{d}}{\mathrm{dt}}\left( u_m^{2\delta}\nabla u_m,\nabla u_m\right) - \left( u_m^{2\delta}\nabla u_m,\partial t\nabla u_m\right)\\&= \frac{\mathrm{d}}{\mathrm{dt}}\|u^{\delta}\nabla u_m\|_{\L^2}^2 - \left( u_m^{2\delta}\nabla u_m,\partial_t\nabla u_m\right).
		\end{align*}
		Using the above estimate, \eqref{2r.14} can be rewritten as 
		\begin{align}\label{2r.15}
			\nonumber
			&\frac{\mathrm{d}}{\mathrm{dt}}\|u_m(t)^{\delta}\nabla u_m(t)\|_{\L^2}^2 +\nu\|u_m(t)^{\delta}\Delta u\|_{\L^2}^2+\eta((K*\Delta) u_m(t),u_m(t)^{2\delta}\Delta u_m(t))\\&= ( u_m(t)^{2\delta}\nabla u_m(t),\partial_t\nabla u_m(t))+\alpha (B(u_m(t) ),u^{2\delta}\Delta u_m(t) )\nonumber\\&\quad-\beta(c(u_m(t)),u_m(t)^{2\delta}\Delta u_m(t))-(f(t),u_m(t)^{2\delta}\Delta u_m(t)),
		\end{align}
		for a.e. $t\in[0,T]$. For the term  $-\beta(c(u_m),u_m^{2\delta}\Delta u_m)$, using integration by parts and Young's inequality, we get
		\begin{align*}
			\nonumber-&\beta(c(u_m),u_m^{2\delta}\Delta u_m) =-  \beta(u_m(1-u_m^{\delta})(u_m^{\delta}-\gamma),u_m^{2\delta}\Delta u_m) \\&\nonumber= -\beta(1+\gamma)(u_m^{\delta+1},u_m^{2\delta}\Delta u_m) + \beta\gamma(u_m,u_m^{2\delta}\Delta u_m)+ \beta(u^{2\delta+1},u_m^{2\delta}\Delta u_m)\\&\nonumber= -\beta(1+\gamma)(u_m^{3\delta+1},\Delta u_m) + \beta\gamma(u_m^{2\delta+1},\Delta u_m)+ \beta(u_m^{4\delta+1},\Delta u_m)\\&\nonumber= \beta(1+\gamma)(3\delta+1)(u_m^{3\delta}\nabla u_m,\nabla u_m) -\beta\gamma(2\delta+1)(u_m^{2\delta}\nabla u_m, \nabla u_m)- \beta(4\delta+1)(u_m^{4\delta}\nabla u_m,\nabla u_m)\\&\nonumber\leq \frac{\beta}{4}(4\delta+1)\|u_m^{2\delta}\nabla u_m\|_{\L^2}^2 + \frac{\beta(1+\gamma)^2(3\delta+1)^2}{(4\delta+1)}\|u_m^{\delta}\nabla u_m\|_{\L^2}^2\nonumber\\&\quad-\beta\gamma(2\delta+1)\|u_m^{\delta}\nabla u_m\|_{\L^2}^2- \beta(4\delta+1)\|u^{2\delta}\nabla u_m\|_{\L^2}^2.
		\end{align*}
		An application of Cauchy-Schwarz and Young's inequalities yields
		\begin{align*}
			\alpha (B(u_m ),u_m^{2\delta}\Delta u_m ) &= \sum_{i=1}^d\left(u_m^{\delta}\frac{\partial u_m}{\partial x_i},u_m^{2\delta}\Delta u_m\right) \leq \|u_m^{2\delta}\nabla u_m\|_{\L^2}\|u_m^{\delta}\Delta u_m \|_{\L^2} \nonumber\\&\leq \frac{1}{\nu}\|u_m\|_{\L^{\infty}}^{2\delta} \|u_m^{\delta}\nabla u_m\|_{\L^2}^2 + \frac{\nu}{4}\|u_m^{\delta}\Delta u_m \|_{\L^2} ^2 ,\\
			(f,u_m^{2\delta}\Delta u_m) & = (u_m^{\delta}f ,u_m^{\delta}\Delta u_m)\leq \frac{\nu}{4}\|u_m^{\delta}\Delta u_m \|_{\L^2} ^2 + \frac{1}{\nu}\|u_m^{\delta}f\|_{\L^2}^2 \leq \frac{\nu}{4}\|u_m^{\delta}\Delta u_m \|_{\L^2} ^2 + \frac{1}{\nu}\|u_m\|_{\L^{\infty}}^{2\delta}\|f\|_{\L^2}^2,\\
			( u_m^{2\delta}\nabla u_m,\partial_t\nabla u_m) &\leq \|u_m^{2\delta}\nabla u_m\|_{\L^2}^2\|\partial_t \nabla u_m\|_{\L^2}^2\leq \frac{\beta}{4}(4\delta+1)\|u_m^{2\delta}\nabla u_m\|_{\L^2}^2 + \frac{1}{\beta(4\delta+1)}\|\partial_t\nabla u_m\|_{\L^2}^2.
		\end{align*}
		Combining the above estimates and integrating from $0$ to $t$, and using Lemma \ref{l2} and  Remark \ref{r1}, we have
		\begin{align}\label{2r.17}
			\nonumber
			&\|u_m(t)^{\delta}\nabla u_m(t)\|_{\L^2}^2 + \frac{\nu}{2}\int_0^t\|u_m(s)^{\delta}\Delta u_m\|_{\L^2}^2\mathrm{~d}s
			+ \frac{\beta}{2}(4\delta+1)\int_0^t\|u_m(s)^{2\delta}\nabla u_m(s)\|_{\L^2}^2\mathrm{~d}s \\&\nonumber\quad+\beta\gamma(2\delta+1)\int_0^t\|u_m(s)^{\delta}\nabla u_m(s)\|_{\L^2}^2\mathrm{~d}s\\& \leq \|u_m(0)^{\delta}\nabla u_m(0)\|_{\L^2}^2 + \frac{1}{\beta(4\delta+1)}\int_0^t\|\partial_t\nabla u_m(s)\|_{\L^2}^2\mathrm{~d}s+  \frac{\beta(1+\gamma)^2(3\delta+1)^2}{(4\delta+1)}\int_0^t\|u_m(s)^{\delta}\nabla u_m(s)\|_{\L^2}^2\mathrm{~d}s \nonumber\\&\quad+\frac{1}{\nu}\int_0^t\|u_m(s)\|_{\L^{\infty}}^{2\delta} \|u_m(s)^{\delta}\nabla u_m(s)\|_{\L^2}^{2\delta}\mathrm{~d}s+ \frac{1}{\nu}\int_0^t\|u_m(s)\|_{\L^{\infty}}^{2\delta}\|f(s)\|_{\L^2}^2\mathrm{~d}s.
		\end{align}
		As $f\in \mathrm{H}^1(0,T;\L^2(\Omega))$ , we have $f\in \L^2(0,T;\L^2(\Omega))$ and $\partial_tf\in L^2(0,T;\L^2(\Omega))$, which implies $f\in \C([0,T], \L^2(\Omega))$, so we have 
		$$
		\int_0^t\|u_m(s)\|_{\L^{\infty}}^{2\delta}\|f(s)\|_{\L^2}^2\mathrm{~d}s \leq \sup_{0\leq t\leq T}\|f(t)\|_{\L^2}^2 \int_0^t\|u_m(s)\|_{\L^{\infty}}^{2\delta}\mathrm{~d}s.
		$$
		Using the last estimate and an application of Gronwall's inequality in \eqref{2r.17} yield
		\begin{align*}
			\nonumber
			&\sup_{0\leq t\leq T}\|u_m(t)^{\delta}\nabla u_m(t)\|_{\L^2}^2 + \frac{\nu}{2}\int_0^t\|u_m(s)^{\delta}\Delta u_m\|_{\L^2}^2\mathrm{~d}s+ \frac{\beta}{2}(4\delta+1)\int_0^t\|u_m(s)^{2\delta}\nabla u_m(s)\|_{\L^2}^2\mathrm{~d}s
			\\& \leq \bigg(\|u_m(0)^{\delta}\nabla u_m(0)\|_{\L^2}^2 + \frac{1}{\beta(4\delta+1)}\int_0^T\|\partial_t\nabla u_m(s)\|_{\L^2}^2\mathrm{~d}s+ \sup_{0\leq t\leq T}\|f(t)\|_{\L^2}^2 \int_0^T\|u_m(s)\|_{\L^{\infty}}^{2\delta}\mathrm{~d}s\bigg)\nonumber\\& \quad\times e^{\left(  \frac{\beta(1+\gamma)^2(3\delta+1)^2}{(4\delta+1)}\right)} \exp\bigg\{\frac{1}{\nu}\int_0^T\|u_m(s)\|_{\L^{\infty}}^{2\delta}\mathrm{~d}s\bigg\}.
		\end{align*}
		For the right hand side to be bounded, we need to show that $\int_0^T\|u_m(t)\|_{\L^{\infty}}^{2\delta}\mathrm{~d}s<\infty$. As $u \in \L^{\infty}(0,T;\H_0^1(\Omega))\cap \L^2(0,T;\H^2(\Omega))$, using Agmon's inequality we have for $3D$ $(d=3)$
		\begin{align*}
			\int_0^T\|u_m(s)\|_{\L^{\infty}}^{2\delta}\mathrm{~d}s \leq C \int_0^T\|u_m(s)\|_{\H^1}^{\delta}\|u_m(s)\|_{\H^2}^{\delta}\mathrm{~d}s \leq  C \sup_{0\leq t\leq T}\|u_m(s)\|_{\H^1}^{\delta}\int_0^T\|u_m(s)\|_{\H^2}^{\delta}\mathrm{~d}s < \infty,
		\end{align*}
		for $1\leq \delta\leq 2$. For the $2D$ $(d=2)$, for $\epsilon >0$, using Agmon's inequality $\|u\|_{\L^{\infty}}\leq C \|u\|_{\H^{1-\epsilon}}^{\frac{1}{1+\epsilon}}\|u\|_{\H^{2}}^{\frac{\epsilon}{1+\epsilon}}$, so 
		\begin{align*}
			\int_0^T\|u_m(s)\|_{\L^{\infty}}^{2\delta}\mathrm{~d}s \leq C \int_0^T \|u_m(s)\|_{\H^{1-\epsilon}}^{\frac{2\delta}{1+\epsilon}}\|u_m(s)\|_{\H^{2}}^{\frac{2\epsilon\delta}{1+\epsilon}}\mathrm{~d}s\leq  C \sup_{0\leq t\leq T}\|u_m(s)\|_{\H^{1-\epsilon}}^{\frac{2\delta}{1+\epsilon}}\int_0^t\|u_m(s)\|_{\H^2}^{\frac{2\epsilon\delta}{1+\epsilon}}\mathrm{~d}s < \infty,
		\end{align*}
		provided $\frac{2\epsilon\delta}{1+\epsilon} \leq 2 \implies \delta \leq \frac{1}{\epsilon}+1 $. Since $\epsilon>0$ is arbitrary, the above estimate is finite for $1\leq \delta < \infty$. Hence  $\nabla u_m^{\delta+1} \in \L^{\infty}(0,T;\L^{2}(\Omega))$.
		\vskip 0.2mm 
		\textbf{Step 3.} $u\in\mathrm{L}^{\infty}(0,T;\mathrm{H}^2(\Omega))$.
		
		Let $\lambda_k$ denote the $k^{\mathrm{th}}$ eigenvalue of $\Delta$ in $\H_0^1(\Omega)$. Multiplying the identity (\ref{8}) by $-\lambda_k^2d_k^m(t)$ and summing it from $k=1,2,\ldots,m$, we have
		\begin{align}
			\nonumber
			&(\partial_t u_m(t),A^2u_m(t))-\nu(Au_m(t), A^2u_m(t))-\eta(K*A u_m(t), A^2u_m(t))\\&=-\alpha (B(u_m(t) ),A^2u_m(t) )+\beta(c(u_m(t)),A^2u_m(t))+(f(t),A^2u_m(t)). 
		\end{align}
		An application of integration by parts yields
		\begin{align}
			\nonumber
			&\frac{1}{2}\frac{\mathrm{d}}{\mathrm{d}t}\|Au_m(t)\|_{\L^2}^2+\nu\|A^{\frac{3}{2}}u_m(t)\|_{\L^2}^2+\eta(K*A^{\frac{3}{2}} u_m(t),A^{\frac{3}{2}}u_m(t))\\&=-\alpha (B(u_m(t)),A^2u_m(t))+\beta(c(u_m(t)),A^2u_m(t))+(f(t),A^2u_m(t)). 
		\end{align}
		Note that $\D(A^{\alpha})=\left\{\begin{array}{cl}\H^{2\alpha}(\Omega)&\text{ if }0<\alpha<1/4,\\
			\H_0^{2\alpha}(\Omega)&\text{ if }1/4<\alpha<1.\end{array}\right.$ Since $u_m\big|_{\partial\Omega}=0$, 
		using integration by parts, Cauchy-Schwarz and Young's inequalities, we have 
		\begin{align*}
			\alpha |(B(u_m ),A^2u_m )|&\leq\alpha \|A^{\frac{1}{2}}B(u_m)\|_{\L^2}\|A^{\frac{3}{2}} u_m\|_{\L^2}\leq\alpha\|A^{\frac{1}{2}}(u_m^{\delta}\nabla  u_m)\|_{\L^2}\|A^{\frac{3}{2}}u_m\|_{\L^2}\nonumber\\
			&\nonumber\leq\frac{\nu}{8}\|A^{\frac{3}{2}}u_m\|_{\L^2}^2+\frac{2\alpha^2}{\nu(\delta+1)}\|A^{\frac{1}{2}}\nabla  u_m^{\delta+1}\|_{\L^2}^2\\&\leq\frac{\nu}{8}\|A^{\frac{3}{2}}u_m\|_{\L^2}^2+\frac{C\alpha^2}{\nu}\left(\|u_m^{\delta} Au_m\|_{\L^2}^2 + \delta\| u_m^{\delta-1}(\nabla u_m)^2\|_{\L^2}^2\right).
		\end{align*} 
		Now we estimate $\| u_m^{\delta-1}(\nabla u_m)^2\|_{\L^2}^2$ using  Sobolev embedding (for $d=2$, $1\leq \delta<\infty$ and for $d=3$, $1\leq\delta\leq 2$), H\"older's and Young's inequalities as
		\begin{align*}
			\|u_m^{\delta-1}(\nabla u_m)^2\|_{\L^2}^2 &= \int_{\Omega} |u_m|^{2(\delta-1)}|\nabla u_m|^4 \mathrm{d}x = \int_{\Omega} |u_m|^{2(\delta-1)}|\nabla u_m|^{\frac{\delta-1}{\delta}}|\nabla u_m|^{\frac{3\delta+1}{\delta}} \mathrm{d}x\\& \leq \left(\int_{\Omega} |u_m|^{4\delta}|\nabla u_m|^2 \mathrm{d}x\right)^{\frac{\delta-1}{2\delta}}\left(\int_{\Omega}|\nabla u_m|^{\frac{2(3\delta+1)}{\delta+1}} \mathrm{d}x\right)^{\frac{\delta+1}{2\delta}}\\& = \|u_m^{2\delta}\nabla u_m\|_{\L^2}^{\frac{\delta-1}{\delta}}\|\nabla u_m\|_{\L^{\frac{2(3\delta+1)}{\delta+1}}}^{\frac{3\delta+1}{\delta}}\\& \leq  C\|u_m^{2\delta}\nabla u_m\|_{\L^2}^{\frac{\delta-1}{\delta}}\|A u_m\|_{\L^2}^{\frac{3\delta+1}{\delta}}=  C\||u_m|^{2\delta}\nabla u_m\|_{\L^2}^{\frac{\delta-1}{\delta}}\|A u_m\|_{\L^2}^{\frac{\delta-1}{\delta}}\|A u_m\|_{\L^2}^{\frac{2(\delta+1)}{\delta}}\\&\leq \frac{1}{2}\|u_m^{2\delta}\nabla u_m\|_{\L^2}^2\|Au_m\|_{\L^2}^2 +C \left(\frac{2(\delta-1)}{2\delta}\right)^{\frac{\delta-1}{\delta+1}}\left(\frac{\delta+1}{2\delta}\right)\|Au_m\|_{\L^2}^4.
		\end{align*}
		Combining, we have 
		\begin{align*}
			\alpha |(B(u_m ),A^2u_m )|&\leq \frac{\nu}{8}\|A^{\frac{3}{2}}u_m\|_{\L^2}^2+\frac{C\alpha^2}{\nu}\|u_m^{\delta} Au_m\|_{\L^2}^2 + \frac{\delta}{2}\|u_m^{2\delta}\nabla u_m\|_{\L^2}^2 \|Au_m\|_{\L^2}^2+ C(\alpha,\nu,\delta)\|Au_m\|_{\L^2}^4.
		\end{align*} 
		
		Again using  Cauchy-Schwarz and Young's inequalities, we get
		\begin{align*}
			&|\beta(c(u_m),A^2u_m)|\nonumber=|\beta((1+\gamma)u_m^{\delta+1}-\gamma u_m-u_m^{2\delta+1},A^2u_m)|\nonumber\\
			&=\beta(1+\gamma)(A^{\frac{1}{2}}u_m^{\delta+1},A^{\frac{3}{2}}u_m)+\beta\gamma (A  u_m, A u_m) + \beta(A^{\frac{1}{2}} u_m^{2\delta+1},A^{\frac{3}{2}}  u_m)\nonumber\\
			&\leq \beta(1+\gamma)(\delta+1)\|u_m^{\delta}A^{\frac{1}{2}}u_m\|_{\L^2}\|A^{\frac{3}{2}}u_m\|_{\L^2}+\beta\gamma \|A u_m\|_{\L^2}^2+\beta(2\delta+1)\|u_m^{2\delta}A^{\frac{1}{2}} u_m\|_{\L^2}\|A^{\frac{3}{2}}u_m\|_{\L^2}\nonumber\\
			&\leq \frac{\beta^2(1+\gamma)^2(\delta+1)^2}{\nu}\|u_m^{\delta}A^{\frac{1}{2}}u_m\|^2_{\L^2} +\beta\gamma \|A u_m\|^2_{\L^2}+\frac{\beta^2(2\delta+1)^2}{\nu}\|u_m^{2\delta}A^{\frac{1}{2}} u_m\|^2_{\L^2} + \frac{\nu}{4}\|A^{\frac{3}{2}}u_m\|_{\L^2}^2.
		\end{align*}
		Finally, for the forcing term using Cauchy-Schwarz and Young's inequalities, we obtain
		\begin{align*}
			|(f, A^2u_m)| \leq |(A^{\frac{1}{2}}f, A^{\frac{3}{2}} u_m)| \leq \frac{1}{\nu}\|A^{\frac{1}{2}}f\|_{\L^2}^2+\frac{\nu}{8}\|A^{\frac{3}{2}}u_m\|^2_{\L^2}.
		\end{align*}
		Combing the above estimates we have 
		\begin{align*}
			\nonumber
			&\frac{1}{2}\frac{\mathrm{d}}{\mathrm{d}t}\|Au_m(t)\|_{\L^2}^2+\frac{\nu}{2}\|A^{\frac{3}{2}}u_m(t)\|_{\L^2}^2+\eta(K*A^{\frac{3}{2}} u_m(t),A^{\frac{3}{2}}u_m(t))\\&\leq \frac{\beta^2(1+\gamma)^2(\delta+1)^2}{\nu}\|u_m(t)^{\delta}A^{\frac{1}{2}}u_m(t)\|^2_{\L^2} +\frac{\beta^2(2\delta+1)^2}{\nu}\|u_m(t)^{2\delta}A^{\frac{1}{2}} u_m(t)\|^2_{\L^2}+ \frac{1}{\nu}\|A^{\frac{1}{2}}f(t)\|_{\L^2}^2 \nonumber\\& \quad +\frac{C\alpha^2}{\nu}\|u_m(t)^{\delta} Au_m(t)\|_{\L^2}^2+\left( \frac{\delta}{2}\|u_m(t)^{2\delta}\nabla u_m(t)\|_{\L^2}^2 +\beta\gamma + C(\alpha,\nu,\delta)\|Au_m(t)\|_{\L^2}^2\right)\|Au_m(t)\|_{\L^2}^2.
		\end{align*}
		Integrating from $0$ to $t$, using positivity of the kernel $K$ and Gronwall's inequality, and taking supremum over  time, we have
		\begin{align*}
			\nonumber
			&\sup_{0\leq t\leq T}\|Au_m(t)\|_{\L^2}^2+\nu\int_0^t\|A^{\frac{3}{2}}u_m(s)\|_{\L^2}^2 \mathrm{d}s \\&\leq\bigg( \frac{2\beta^2(1+\gamma)^2(\delta+1)^2}{\nu}\int_0^T\|u_m(s)^{\delta}A^{\frac{1}{2}}u_m(s)\|^2_{\L^2}\mathrm{d}s +\frac{2\beta^2(2\delta+1)^2}{\nu}\int_0^T\|u_m(s)^{2\delta}A^{\frac{1}{2}} u_m(s)\|^2_{\L^2}\mathrm{d}s\nonumber\\&\qquad+ \frac{2}{\nu}\int_0^T\|A^{\frac{1}{2}}f(s)\|_{\L^2}^2\mathrm{d}s + \|u_0\|_{\H^2}\bigg)\nonumber\\&\quad\times 
			\exp\left( \delta\int_0^T\|u_m(s)^{2\delta}\nabla u_m(s)|_{\L^2}^2 \mathrm{d}s +2\beta\gamma T+ C(\alpha,\nu,\delta)\int_0^T\|Au_m(s)\|_{\L^2}^2\mathrm{d}s \right).
		\end{align*}
		Using the estimate \eqref{2M.53} and  \eqref{2r.17}  the right hand side is bounded. 
		Hence we have $u_m\in\L^{\infty}(0,T;\H^2(\Omega))$, $\ \partial_tu_m\in\L^{\infty}(0,T;\L^2(\Omega))\cap\L^2(0,T;\H_0^1(\Omega))$.
		Passing limit along a subsequence $m_j\to\infty$, we obtain the required bound for $u$.
	\end{proof}
	\begin{remark}
		In Theorem \ref{2M.thm33} $(ii)$, we need $f \in \L^2(0,T;\H^1(\Omega))$ for the memory term. However, if we take $\eta = 0$, that is, GBHE without memory using elliptic regularity we can prove that for $u_0\in\H^2(\Omega)\cap\H_0^1(\Omega)$ and  $f\in\H^1(0,T;\L^2(\Omega))$, we have $u \in \L^{\infty}(0,T;\H^2(\Omega))$ . 
	\end{remark}
	\section{Finite Element Method}\label{sec3}
	\subsection{Semidiscrete Galerkin approximation  and error estimates}\label{2M.se7}\setcounter{equation}{0}
	This section is devoted to the error estimates for the semidiscrete Galerkin approximation. We partition the domain $\Omega$ into shape-regular meshes (triangular or rectangles) denoted by $\mathcal{T}_h$ and consider a finite dimensional space $V_h$, defined as  
	\begin{align}\label{31}
		V_{h}=\left\{v_{h}: v_{h} \in C^0(\bar{\Omega})\cap\H_0^1(\Omega),\left.v_{h}\right|_{T} \in\mathbb{P}_1 (T) \ \text{ for all }\ T \in \mathcal{T}_{h}\right\},\end{align}
	where $\mathbb{P}_1 $ is the space of polynomials which have degree at most $1$. Note that, $V_h$  is a finite dimensional subspace of $\H_0^1(\Omega)$ and $0<h<1$ is the associated small mesh parameter. The aim is to prove the estimate of the type  \citep[see][]{VTh}
	\begin{align*}
		\inf_{\chi\in V_h}\left\{\|u-\chi\|_{\L^2}+h\|\nabla(u-\chi)\|_{\L^2}\right\}\leq Ch^2\|u\|_{\H^2},
	\end{align*}
	for all $u\in\H^2(\Omega)\cap\H_0^1(\Omega)$. The continuous time Galerkin approximation of the problem \eqref{AFGBHEM} is defined in the following way: Find $u_h(\cdot,t)\in V_h,$ such that for $t\in(0,T)$,
		\begin{align}\label{7p1}
	\nonumber	\langle\partial_tu_h(t),\chi\rangle+\nu (\nabla u_h(t),\nabla \chi)+\alpha b(u_h(t),u_h(t),\chi)+\eta((K*\nabla u_h)(t),\nabla \chi)&=\beta(c(u_h(t)),\chi)+\langle f(t),\chi\rangle, \\
			(u_h(0),\chi)&=(u_0^h,\chi), \ \text{ for }\ \chi\in V_h,
		\end{align}
	where $u_0^h$ approximates $u_0$ in $V_h$. Central to the analysis of finite element methods, we define an elliptic projection associated to our model onto $V_h$, the Ritz projection (\cite{VTh}) of $u$ defined by 
	\begin{align*}
		(\nabla R_hu ,\nabla \chi) = (\nabla u,\nabla \chi), \ \text{ for all } \ \chi\in V_h.
	\end{align*}
	By setting $\chi=R_hu$ in the above equality, we obtain  that the Ritz projection is stable, that is, $\|\nabla R_hu\|_{\L^2}\leq\|\nabla u\|_{\L^2}$, for all $u\in\H_0^1(\Omega)$. If we take $\Lambda = R_hu-u$, then, we have the following estimate for $\Lambda$ (\cite{VTh}),
	\begin{align}\label{7a1}
		\|\Lambda(t)\|_{\L^2}+h\|\nabla\Lambda(t)\|_{\L^2}\leq Ch^s\|u(t)\|_{\H^s}, \text{ for  }\ s=1,2.
	\end{align}
	\begin{theorem}[Existence of a discrete solution]
	The discrete equation \eqref{7p1} admits at least one solution $u_h\in V_h$.
	\end{theorem}
\begin{proof}
It follows as a direct consequence of Theorem \ref{EWS}.
\end{proof}\\
	For all $t\in[0,T]$,  the error estimates in this semi-discretization are given in the following theorem,
	\begin{theorem}\label{thm7.1}
		Let $V_h$ be the finite dimensional subspace of $\H_0^1(\Omega)$ with parameter $h$ defined in \eqref{31}. Assume that $u_0\in \L^{d\delta}(\Omega)\cap\H_0^1(\Omega)$ and $f\in\mathrm{L}^2(0,T;\L^2(\Omega))$, and $u(\cdot)$, the solution of  \eqref{weaksolution} has the following regularity: 
		\begin{align*}
			u\in\mathrm{L}^{\infty}(0,T;\mathrm{L}^{2(\delta+1)}(\Omega))\cap\mathrm{L}^2(0,T;\H_0^1(\Omega))\cap\mathrm{L}^{(d+2)\delta}(0,T;\mathrm{L}^{(d+2)\delta}(\Omega)) , \ \partial_tu\in\mathrm{L}^2(0,T;\H_0^1(\Omega)).
		\end{align*}
		Then the error in the semi-discretization satisfies the following estimate:
		\begin{align}\label{7a3}
		&	\|u_h-u\|_{\L^{\infty}(0,T;\L^2(\Omega))}^2 +	\|u_h-u\|_{\L^{2}(0,T;\H_0^1(\Omega))}^2 \nonumber\\&\leq C\bigg\{\|u^h_0-u_0\|_{\L^2}^2+h^2\int_0^T\|\partial_tu(t)\|^2_{\H_0^1}\mathrm{~d} t + h^2\int_0^T\|u(t)\|_{\H_0^1}^2\mathrm{~d} t + h^2\int_0^T\|u(t)\|_{\H^2}^2\mathrm{~d} t\bigg \},
	\end{align}
		where the constant $C$ is independent of $h$ and depends only on  
		$\|u_0\|_{\H_0^1},\nu,\alpha,\beta,\gamma,\eta,\delta,T,\|f\|_{\L^2(0,T;\L^2(\Omega))}$.
		\end{theorem}
	\textbf{Note:} If $u_0\in \L^{d\delta}(\Omega)\cap\H_0^1(\Omega)$ and $f\in\mathrm{L}^2(0,T;\L^2(\Omega))$, then using Theorem \ref{2M.thm3.2} we have $u\in\mathrm{L}^{\infty}(0,T;\mathrm{L}^{2(\delta+1)}(\Omega))\cap\mathrm{L}^2(0,T;\H_0^1(\Omega))\cap\mathrm{L}^{(d+2)\delta}(0,T;\mathrm{L}^{(d+2)\delta}(\Omega))$.\\
	\begin{proof}
		Using \eqref{weaksolution} and \eqref{7p1}, we know that $(u_h-u)(\cdot)$ satisfies
		\begin{align}\label{7p2}
			&	\langle\partial_t(u_h(t)-u(t)),\chi\rangle+\nu (\nabla (u_h(t)-u(t)),\nabla \chi)+\eta( (K*\nabla(u_h-u))(t),\nabla\chi)\nonumber\\&= -\alpha[b(u_h(t),u_h(t),\chi)-b(u(t),u(t),\chi)]+\beta[c(u_h(t),\chi)-c(u(t),\chi)],
		\end{align}
		for all $\chi\in V_h$ and a.e. $t\in[0,T]$. Let us choose $\chi=u_h-W\in V_h$ in \eqref{7p2} to obtain for a.e. $t\in[0,T]$
		\begin{align}\label{7p3}
			&	\langle\partial_t(u_h(t)-u(t)),u_h(t)-W(t)\rangle+\nu (\nabla(u_h(t)-u(t)),\nabla(u_h(t)-W(t)))\nonumber\\&\quad +\eta\langle (K*\nabla(u_h-u))(t),\nabla(u_h(t)-W(t))\rangle\nonumber\\&= -\alpha[b(u_h(t),u_h(t),u_h(t)-W(t))-b(u(t),u(t),u_h(t)-W(t))]\nonumber\\&\quad+\beta[c(u_h(t),u_h(t)-W(t))-c(u(t),u_h(t)-W(t))].
		\end{align}
		Next, we  write $u_h-W$ as $u_h-u+u-W$ in \eqref{7p3} to find
		\begin{align}\label{minreg}
			&	\frac{1}{2}\frac{\mathrm{~d}}{\mathrm{~d}t}\|u_h(t)-u(t)\|_{\L^2}^2+\nu\|\nabla(u_h(t)-u(t))\|_{\L^2}^2+\eta( (K*\nabla(u_h-u))(t),\nabla(u_h(t)-u(t)))\nonumber\\&=-\langle\partial_t(u_h(t)-u(t)),u(t)-W(t)\rangle-(\nabla(u_h(t)-u(t)),\nabla(u(t)-W(t)))\nonumber\\&\quad -\eta\langle (K*\nabla(u_h-u))(t),\nabla(u(t)-W(t))\rangle\nonumber\\&\quad-\alpha\left({u_h(t)}^{\delta}\sum_{i=1}^d\frac{\partial u_h(t)}{\partial x_i}-u(t)^{\delta}\sum_{i=1}^d\frac{\partial u(t)}{\partial x_i},u_h(t)-u(t)\right) \nonumber\\&\quad-\alpha\left({u_h(t)}^{\delta}\sum_{i=1}^d\frac{\partial u_h(t)}{\partial x_i}-u(t)^{\delta}\sum_{i=1}^d\frac{\partial u(t)}{\partial x_i}, u(t)-W(t)\right)\nonumber\\&\quad+\beta\left(u_h(t)(1-u_h(t)^{\delta})(u_h(t)^{\delta}-\gamma)-u(t)(1-u(t)^{\delta})(u(t)^{\delta}-\gamma),u_h(t)-u(t)\right)\nonumber\\&\quad+\beta\left(u_h(t)(1-u_h(t)^{\delta})(u_h(t)^{\delta}-\gamma)-u(t)(1-u(t)^{\delta})(u(t)^{\delta}-\gamma),u(t)-W(t)\right),
		\end{align}
		for a.e. $t\in[0,T]$. 	Using calculations similar to  \eqref{332} yields
		\begin{align}\label{7p5}
			&	\frac{\mathrm{~d}}{\mathrm{~d} t}\|u_h(t)-u(t)\|_{\L^2}^2+\nu\|\nabla(u_h(t)-u(t))\|_{\L^2}^2+2\eta ((K*\nabla(u_h-u))(t),\nabla(u_h(t)-u(t)))\nonumber\\&\quad+\frac{\beta}{2}\|u_h(t)^{\delta}(u_h(t)-u(t))\|_{\L^2}^2+\frac{\beta}{2}\|u(t)^{\delta}(u_h(t)-u(t))\|_{\L^2}^2 + \beta\gamma\|u_h(t)-u(t)\|_{\L^2}^2\nonumber\\&\leq \left[C(\beta,\gamma,\delta) +C(\alpha,\nu)\left(\|u_h(t)\|^{\frac{8\delta}{4-d}}_{\L^{4\delta}}+\|u(t)\|^{\frac{8\delta}{4-d}}_{\L^{4\delta}}\right)\right]\|u_h(t)-u(t)\|_{\L^2}^2 \nonumber\\&\quad -2\partial_t(u_h(t)-u(t),u(t)-W(t))+2(u_h(t)-u(t),\partial_t(u(t)-W(t)))\nonumber\\&\quad-2(\nabla(u_h(t)-u(t)),\nabla(u(t)-W(t)))-2\eta\langle (K*\nabla(u_h-u))(t),\nabla(u(t)-W(t))\rangle\nonumber\\&\quad-2\alpha\left({u_h(t)}^{\delta}\sum_{i=1}^d\frac{\partial u_h(t)}{\partial x_i}-u(t)^{\delta}\sum_{i=1}^d\frac{\partial u(t)}{\partial x_i},u(t)-W(t)\right)\nonumber\\&\quad+2\beta\left(u_h(t)(1-u_h(t)^{\delta})(u_h(t)^{\delta}-\gamma)-u(t)(1-u(t)^{\delta})(u(t)^{\delta}-\gamma),u(t)-W(t)\right)\nonumber\\&= \left(C(\beta,\gamma,\delta)+C(\alpha,\nu)\left(\|u_h(t)\|^{\frac{8\delta}{4-d}}_{\L^{4\delta}}+\|u(t)\|^{\frac{8\delta}{4-d}}_{\L^{4\delta}}\right)\right)\|u_h(t)-u(t)\|_{\L^2}^2 \nonumber\\&\quad -2\partial_t(u_h(t)-u(t),u(t)-W(t)) -2\eta\langle (K*\nabla(u_h-u))(t),\nabla(u(t)-W(t))\rangle+\sum_{i=1}^4J_i,
		\end{align}
		where the positive constants $C(\beta,\gamma,\delta)= \beta 2^{2\delta}(1+\gamma)^2(\delta+1)^2$ and $C(\alpha, \nu) = \left(\frac{4+d}{4\nu}\right)^{\frac{4+d}{4-d}}\left(\frac{4-d}{8}\right)(2^{\delta-1}\alpha)^{\frac{4-d}{8}}$. We estimate $J_1$ and $J_2$ using the Cauchy-Schwarz and Young's inequalities as
		\begin{align}\label{7p6}
			|J_1|&\leq 2\|u_h-u\|_{\mathrm{L}^2}\|\partial_t(u-W)\|_{\mathrm{L}^2}\leq\|u_h-u\|_{\mathrm{L}^2}^2+\|\partial_t(u-W)\|_{\mathrm{L}^2}^2,\\|J_2|&\leq 2\|\nabla(u_h-u)\|_{\L^2}\|\nabla(u-W)\|_{\L^2}\leq\frac{\nu}{4}\|\nabla(u_h-u)\|_{\L^2}^2+\frac{4}{\nu}\|\nabla(u-W)\|_{\L^2}^2\label{1.M}.
		\end{align}
		Using an integration by parts, Taylor's formula, H\"older's and Young's inequalities, we rewrite $J_3$ as
		\begin{align}\label{7p7}
			J_3&=-\frac{2\alpha}{\delta+1}\sum_{i=1}^d\left(\frac{\partial}{\partial{x_i}}({u_h}^{\delta+1}-u^{\delta+1}),u-W\right)=\frac{2\alpha}{\delta+1}\sum_{i=1}^d\left(\left(u_h^{\delta+1}-u^{\delta+1}\right),\frac{\partial}{\partial x_i}(u-W)\right)\nonumber\\&=2\alpha\sum_{i=1}^d\left((\theta u_h+(1-\theta)u)^{\delta}(u_h-u),\frac{\partial}{\partial{x_i}}(u-W)\right)\nonumber\\&\leq 2\alpha\sum_{i=1}^d\|(\theta u_h+(1-\theta)u)^{\delta}(u_h-u)\|_{\L^2}\left\|\frac{\partial}{\partial x_i}{(u-W)}\right\|_{\L^2}\nonumber\\&\leq 2^{\delta}\alpha\left(\|{u_h}^{\delta}(u_h-u)\|_{\L^2}+\|{u}^{\delta}(u_h-u)\|_{\L^2}\right)\|\nabla(u-W)\|_{\L^2}\nonumber\\&\leq\frac{\beta}{8}\|{u_h}^{\delta}(u_h-u)\|_{\L^2}^2+\frac{\beta}{8}\|{u}^{\delta}(u_h-u)\|_{\L^2}^2+\frac{2^{2(\delta+1)}\alpha^2}{\beta}\|\nabla(u-W)\|_{\L^2}^2.
		\end{align}
		Let us rewrite $J_4$ as
		\begin{align}\label{2.M}		J_4&=2\beta(1+\gamma)({u_h}^{\delta+1}-u^{\delta+1},u-W)-2\beta\gamma(u_h-u,u-W)\nonumber\\&\quad-2\beta({u_h}^{2\delta+1}-u^{2\delta+1},u-W):=\sum_{i=5}^7J_i.
		\end{align}
		We estimate $J_5$ using Taylor's formula, H\"older's and Young's inequalities as
		\begin{align}\label{3.M}
			J_5&=2\beta(1+\gamma)(\delta+1)((\theta u_h+(1-\theta)u)^{\delta}(u_h-u),u-W) \nonumber\\&\leq 2^{\delta}\beta(1+\gamma)(\delta+1)\left(\|{u_h}^{\delta}(u_h-u)\|_{\L^2}+\|{u}^{\delta}(u_h-u)\|_{\L^2}\right)\|u-W\|_{\L^2}\nonumber\\&\leq\frac{\beta}{8}\|{u_h}^{\delta}(u_h-u)\|_{\L^2}^2+\frac{\beta}{8}\|{u}^{\delta}(u_h-u)\|_{\L^2}^2+2^{2(\delta+1)}\beta(1+\gamma)^2(\delta+1)^2\|u-W\|_{\L^2}^2.
		\end{align}
		Using the Cauchy-Schwarz and Young's inequalities, we estimate $J_6$ as
		\begin{align}\label{4.M}
			J_6&\leq 2\beta\gamma\|u_h-u\|_{\L^2}\|u-W\|_{\L^2}\leq \beta\gamma\|u_h-u\|_{\L^2}^2+\beta\gamma\|u-W\|_{\L^2}^2.
		\end{align}
		Making use of Taylor's formula, H\"older's, Young's and Gagliardo-Nirenberg inequalities, we estimate $J_7$ as
		\begin{align}\label{7p11}
			J_7&=-2(2\delta+1)\beta\left((\theta u_h+(1-\theta)u)^{2\delta}(u_h-u),u-W\right)\nonumber\\&\leq 2^{2\delta}(2\delta+1)\beta\left(( |u_h|^{2\delta}+|u|^{2\delta})(u_h-u),u-W\right)\nonumber\\&\leq  2^{2\delta}(2\delta+1)\beta\left(\|u_h\|_{\L^{4\delta}}^{2\delta}+\|u\|_{\L^{4\delta}}^{2\delta}\right)\|u_h-u\|_{\L^{\frac{2d}{d-1}}}\|u-W\|_{\L^{2d}}\nonumber\\&\leq  2^{2\delta}(2\delta+1)\beta\left(\|u_h\|_{\L^{4\delta}}^{2\delta}+\|u\|_{\L^{4\delta}}^{2\delta}\right)\|u_h-u\|_{\L^2}^{\frac{1}{2}}\|\nabla(u_h-u)\|_{\L^2}^{\frac{1}{2}}\|u-W\|_{\L^{2d}}\nonumber\\&\leq 2^{2\delta-1}(2\delta+1)\beta\|\nabla(u-W)\|_{\L^2}^2 + 2^{2\delta-1}(2\delta+1)\beta\left(\|u_h\|_{\L^{4\delta}}^{4\delta}+\|u\|_{\L^{4\delta}}^{4\delta}\right)\|u_h-u\|_{\L^2}\|\nabla(u_h-u)\|_{\L^2}	\nonumber\\&\leq 2^{2\delta-1}(2\delta+1)\beta\|\nabla(u-W)\|_{\L^2}^2 + \frac{\nu}{4}\|\nabla(u_h-u)\|_{\L^2}^2 \nonumber\\&\quad + \frac{2^{4\delta-2}(2\delta+1)^2\beta^2}{\nu}\left(\|u_h\|_{\L^{4\delta}}^{8\delta}+\|u\|_{\L^{4\delta}}^{8\delta}\right)\|u_h-u\|_{\L^2}^2,
		\end{align}
		where we used the fact that $\H_0^1(\Omega)\subset\L^{6}(\Omega)$. Note also that
		\begin{align}\label{7p12}
			\|u_h-u\|_{\L^{2(\delta+1)}}^{2(\delta+1)}&=\int_{\Omega}|u_h(x)-u(x)|^{2\delta}|u_h(x)-u(x)|^{2}\mathrm{~d} x\nonumber\\&\leq 2^{2\delta-1}\int_{\Omega}\left(|u_h(x)|^{2\delta}+|u(x)|^{2\delta}\right)|u_h(x)-u(x)|^{2}\mathrm{~d}x\nonumber\\&=2^{2\delta-1}\|{u_h}^{\delta}(u_h-u)\|_{\L^2}^2+2^{2\delta-1}\|{u}^{\delta}(u_h-u)\|_{\L^2}^2.
		\end{align}Combining \eqref{7p6}-\eqref{7p12} and then substituting it in \eqref{7p5}, we deduce 
		\begin{align*}
			&	\frac{\mathrm{~d}}{\mathrm{~d}t}\|u_h(t)-u(t)\|_{\L^2}^2+\frac{\nu}{2}\|\nabla(u_h(t)-u(t))\|_{\L^2}^2+\eta\langle (K*\nabla(u_h-u))(t),\nabla(u_h(t)-u(t))\rangle\nonumber\\&+\frac{\beta}{2^{2\delta+1}}\|u_h(t)-u(t))\|_{\L^{2\delta+2}}^{2\delta+2}\nonumber\\&\leq -2\partial_t(u_h(t)-u(t),u(t)-W(t))+\|\partial_t(u(t)-W(t))\|_{\mathrm{L}^2}^2\nonumber\\&\quad-\eta\langle (K*\nabla(u_h-u))(t),\nabla(u(t)-W(t))\rangle+\left(2^{2\delta-1}(2\delta+1)\beta+\frac{4}{\nu}+\frac{2^{2(\delta+1)}\alpha^2}{\beta}\right)\|\nabla (u(t)-W(t))\|_{\L^2}^2\nonumber\\&\quad+\left(2^{2(\delta+1)}\beta(1+\gamma)^2(\delta+1)^2+\frac{\beta\gamma}{2}\right)\|u(t)-W(t)\|_{\L^2}^2+\bigg(1+2\beta\gamma+C(\beta,\gamma,\delta)\bigg)\|u_h(t)-u(t)\|_{\mathrm{L}^2}^2\nonumber\\&\quad + \left(\frac{2^{4\delta-2}(2\delta+1)^2\beta^2}{\nu}\left(\|u_h(t)\|_{\L^{4\delta}}^{8\delta}+\|u(t)\|_{\L^{4\delta}}^{8\delta}\right)+C(\alpha,\nu)\left(\|u_h(t)\|^{\frac{8\delta}{4-d}}_{\L^{4\delta}}+\|u(t)\|^{\frac{8\delta}{4-d}}_{\L^{4\delta}}\right)\right)\|u_h(t)-u(t)\|_{\mathrm{L}^2}^2.
		\end{align*}
		Integrating the above inequality from $0$ to $t$ and using positivity of kernel $K(\cdot)$, we have
		\begin{align}\label{715}
			&\|u_h(t)-u(t)\|_{\L^2}^2+\frac{\nu}{2}\int_0^t\|\nabla(u_h(s)-u(s))\|_{\L^2}^2\mathrm{~d}s+\frac{\beta}{2^{2\delta+1}}\int_0^t\|u_h(s)-u(s))\|_{\L^{2\delta+2}}^{2\delta+2}\mathrm{~d}s\nonumber\\&\leq \|u_h^0-u_0\|_{\L^2}^2-2(u_h(t)-u(t),u(t)-W(t))+2(u_h^0-u_0,u_0-W(0))\nonumber\\&+\int_0^t\|\partial_t(u(s)-W(s))\|_{\mathrm{L}^2}^2\mathrm{~d}s + \left(2^{2\delta-1}(2\delta+1)\beta+\frac{4}{\nu}+\frac{2^{2(\delta+1)}\alpha^2}{\beta} + \frac{\eta}{4C_K\nu}\right)\int_0^t\|\nabla (u(s)-W(s))\|_{\L^2}^2\mathrm{~d}s\nonumber\\&+\left(2^{2(\delta+1)}\beta(1+\gamma)^2(\delta+1)^2+\frac{\beta\gamma}{2}\right)\int_0^t\|u(s)-W(s)\|_{\L^2}^2 \mathrm{~d}s +\bigg(1+2\beta\gamma+C(\beta,\alpha,\delta)\bigg)\int_0^t\|u_h(s)-u(s)\|_{\mathrm{L}^2}^2 \mathrm{~d}s\nonumber\\& +\int_0^t\left(\frac{2^{4\delta-2}(2\delta+1)^2\beta^2}{\nu}\left(\|u_h(s)\|_{\L^{4\delta}}^{8\delta}+\|u(s)\|_{\L^{4\delta}}^{8\delta}\right)+C(\alpha,\nu)\left(\|u_h(s)\|^{\frac{8\delta}{4-d}}_{\L^{4\delta}}+\|u(s)\|^{\frac{8\delta}{4-d}}_{\L^{4\delta}}\right)\right)\|u_h(s)-u(s)\|_{\mathrm{L}^2}^2\mathrm{~d}s,
		\end{align}
		for all $t\in[0,T]$, where we estimate the memory term using Lemma \ref{l1} as
		\begin{align*}
			&\eta\int_0^t\langle K*\nabla(u_h(s)-u(s)),\nabla(u(s)-W(s))\rangle\mathrm{~d}s\\&\leq\eta \int_0^t\| K*\nabla(u_h(s)-u(s))\|_{\mathrm{L}^2}\|\nabla(u(s)-W(s))\|_{\mathrm{L}^2}\mathrm{~d}s\nonumber\\&\leq\frac{\nu}{4C_K}\int_0^t \| K*\nabla(u_h(s)-u(s))\|_{\mathrm{L}^2}^2\mathrm{~d}s+\frac{C_K\eta^2}{\nu}\int_0^t\|\nabla(u(s)-W(s))\|_{\mathrm{L}^2}^2\mathrm{~d}s\nonumber\\&\leq \frac{\nu}{4}\int_0^t\|\nabla(u_h(s)-u(s))\|_{\mathrm{L}^2}^2\mathrm{~d}s+\frac{C_K\eta^2}{\nu}\int_0^t\|\nabla(u(s)-W(s))\|_{\mathrm{L}^2}^2\mathrm{~d}s,
		\end{align*}
		for the constant $C_K = \left(\int_0^T|K(t)|\mathrm{d}t\right)^2$. Using the Cauchy-Schwarz inequality and Young's inequality, we estimate $-2(u_h-u,u-W)$ and $2(u_h^0-u_0,u_0-W(0))$ as
		\begin{align*}
			-2(u_h-u,u-W)\leq 2\|u_h-u\|_{\L^2}\|u-W\|_{\L^2}\leq\frac{1}{2}\|u_h-u\|_{\L^2}^2+2\|u-W\|_{\L^2}^2,\\
			2(u_h^0-u_0,u_0-W(0))\leq 2\|u_h^0-u_0\|_{\L^2}\|u_0-W(0)\|_{\L^2}\leq\|u_h^0-u_0\|_{\L^2}^2+\|u_0-W(0)\|_{\L^2}^2.
		\end{align*}
		Using the above estimates in \eqref{715} and then applying Gronwall's inequality, we get
		\begin{align}\label{716}
			&\sup_{t\in[0,T]}\|u_h(t)-u(t)\|_{\L^2}^2+\nu\int_0^T\|\nabla(u_h(t)-u(t))\|_{\L^2}^2\mathrm{~d}t+\left(\frac{\beta}{2^{2\delta}}\right)\int_0^T\|u_h(t)-u(t)\|_{\L^{2(\delta+1)}}^{2(\delta+1)}\mathrm{~d}t\nonumber\\&\leq 2\bigg\{2\|u_h^0-u_0\|_{\L^2}^2+\|u_0-W(0)\|_{\L^2}^2+\int_0^T\|\partial_t(u(t)-W(t))\|_{\mathrm{L}^2}^2\mathrm{~d}t\nonumber\\&\quad+\left(2^{2\delta-1}(2\delta+1)\beta+\frac{4}{\nu}+\frac{2^{2(\delta+1)}\alpha^2}{\beta} + \frac{\eta}{4C\nu}\right)\int_0^T\|\nabla(u(t)-W(t))\|_{\L^2}^2\mathrm{~d}t\nonumber\\&\quad+\bigg(2^{2(\delta+1)}\beta(1+\gamma)^2(\delta+1)^2+\frac{\beta\gamma}{2}\bigg)\int_0^T\|u(t)-W(t)\|_{\L^2}^2\mathrm{~d}t\bigg\}\times e^{\left(1+2\beta\gamma+C(\beta,\gamma,\delta)\right)T}\nonumber\\&\quad\times\exp\bigg\{\int_0^T\left(\frac{2^{4\delta-2}(2\delta+1)^2\beta^2}{\nu}\left(\|u_h(t)\|_{\L^{4\delta}}^{8\delta}+\|u(t)\|_{\L^{4\delta}}^{8\delta}\right)+C(\alpha,\nu)\left(\|u_h(t)\|^{\frac{8\delta}{4-d}}_{\L^{4\delta}}+\|u(t)\|^{\frac{8\delta}{4-d}}_{\L^{4\delta}}\right)\right)\mathrm{~d}t\bigg\}.
		\end{align}
		The term appearing in the exponential is uniformly bounded  and is independent of $h$ for $u_0\in \L^{d\delta}$ (see \ref{333}). Using \eqref{7a1}, we get
		\begin{align*}
			\|u_0-W(0)\|_{\L^2}^2&=\|u_0-R^hu_0\|^2_{\L^2}\leq Ch^2\|u_0\|_{\H_0^1}, \\
			\int_0^T\|u(t)-W(t)\|_{\L^2}^2\mathrm{~d}t&=\int_0^T\|u(t)-R^hu(t)\|_{\L^2}^2\mathrm{~d}t\leq Ch^4\int_0^T \|u(t)\|_{\H^2}^2\mathrm{~d} t,\\
			\int_0^T\|\nabla(u(t)-W(t))\|_{\L^2}^2\mathrm{~d}t&=\int_0^T\|\nabla(u(t)-R^hu(t))\|_{\L^2}^2\mathrm{~d}t\leq Ch^2\int_0^T\|u(t)\|_{\H^2}^2 \mathrm{~d} t,\\
			\int_0^T\|\partial_t(u(t)-W(t))\|_{\L^2}^2\mathrm{~d}t&=\int_0^T\|\partial_t(u(t)-R^hu(t))\|_{\L^2}^2\mathrm{~d}t\leq Ch^2\int_0^T \|\partial_tu(t)\|_{\H_0^1}^2\mathrm{~d} t.
		\end{align*}
		Thus, \eqref{716} implies that
		\begin{align*}
			&
			\nu\int_0^T\|\nabla(u^h(t)-u(t))\|_{\L^2}^2\mathrm{~d}t\leq C(\|u_0\|_{\H_0^1},\nu,\alpha,\beta,\gamma,\eta,\delta,T,\|f\|_{\L^2(0,T;\L^2(\Omega))})\nonumber\\&\times\bigg\{\|u^h_0-u_0\|_{\L^2}^2+h^2\int_0^T\|\partial_tu(t)\|^2_{\H_0^1}\mathrm{~d} t + h^2\int_0^T\|u(t)\|_{\H_0^1}^2\mathrm{~d} t + h^2\int_0^T\|u(t)\|_{\H^2}^2\mathrm{~d} t\bigg \},
		\end{align*}
		and hence the estimate \eqref{7a3} follows.
	\end{proof}
	
	\begin{remark}
		1. For $d=2$ with $\delta\in[1,\infty)$ and $d=3$ with $\delta\in[1,2]$, we infer from Theorem \ref{2M.thm33} $\partial_tu\in\mathrm{L}^2(0,T;\H_0^1(\Omega))$. For large values of $\alpha,\beta,\gamma$ (cf. Remark \ref{rem2.4}), one can get the same regularity for the case $d=3$ with $\delta\in(2,\infty)$.
		
		2. 	In the above Theorem we need $\partial_tu\in \L^2(0,T;\H_0^1(\Omega)) $. However, if we assume $u_h \in \H^1(0,T;V_h)$ and  $P_h$ denotes the $\L^2(\Omega)$ projection from $\L^2(\Omega)$ onto $V_h$ as in \cite{LSW} and \cite{KLS}, namely, for each $v \in \L^2(\Omega)$
		\begin{align}\label{L2proj}
			\left(P_h v-v, w^h\right)=0 \quad\text{for all}\quad w^h \in V_h,\text{a.e. } t.
		\end{align}
		Also, there exists a constant $C$ independent of $v$ and $h$ such that
		\begin{align*}
			\|u-P_hv\|_{\H^r} &\leq C \|u-R_hv\|_{\H^r},\quad \text{for all}\quad u\in \H_0^1(\Omega), \   r \in[0,1],
		\end{align*}
		We can estimate the first term on the right hand side of \eqref{minreg} using the $\L^2$ projection \eqref{L2proj} repeatedly as 
		\begin{align*}
			\nonumber	-\langle\partial_t&(u_h(t)-u(t)),u(t)-P_hu(t)\rangle = \langle\partial_t u(t), u(t) - P_hu(t)\rangle\\& = \langle\partial_t (u(t)-P_hu(t)), u(t) - P_hu(t)\rangle = \frac{1}{2}\frac{\mathrm{d}}{\mathrm{d}t}\|u(t)-P_hu(t)\|_{\L^2}^2
		\end{align*}
		for a.e. $t\in[0,T]$,	which upon integration can be estimated as 
		\begin{align*}
			\int_0^T\|u(t)-P_hu(t)\|_{\L^2}^2\mathrm{~d}t&\leq\int_0^T\|u(t)-R^hu(t)\|_{\L^2}^2\mathrm{~d}t\leq Ch^4\int_0^T \|u(t)\|_{\H^2}^2\mathrm{~d} t.
		\end{align*}
		So, we need only our initial data $u_0\in \L^{d\delta}(\Omega)\cap \H_0^1(\Omega)$ and $\partial_tu\in \L^2(0,T;\L^2(\Omega)).$
	\end{remark}
\subsection{Fully-discrete Galerkin approximation and error estimates}

In this section, we introduce a fully-discrete conforming  finite element scheme. We partition the time interval $[0,T]$ into $0 =t_0<t_1<t_2\cdots < t_N=T$ with uniform time stepping $\Delta t$, that is, $t_k = k\Delta t$. We replace the time derivative with a backward Euler discretization and approximating the memory term in the following way: We shall approximate $\psi$ in $J(\psi) = \int_0^{t} K(t-s)\psi(s) \mathrm{d}s$ by the piecewise constant function taking value $\psi^k=\psi(t_k)$ in $(t_{k-1},t_{k})$ and the quadrature rule as 
\begin{align*}
	J(\psi) = \int_0^{t} K(t-s)\psi(s) \mathrm{d}s \approx	\frac{1}{(\Delta t)^2} \int_{t_{k-1}}^{t_k}\int_0^t K(t-s)\Delta t \psi(s)\mathrm{~d}s\mathrm{d}t  =  \sum_{j=1}^k \omega_{kj}\Delta t \psi^j,
\end{align*}
where $\omega_{k j}=\frac{1}{(\Delta t)^2 } \int_{t_{k-1}}^{t_k} \int_{t_{j-1}}^{\min \left(t, t_j\right)} K(t-s) \mathrm{~d}s \mathrm{~d}t$, for $1\leq k\leq N.$
We denote $C$ as a generic constant which  is independent of $h$ and $\Delta t$. We then define a generic finite element approximate solution $u_{kh}$ by 
\begin{align}\label{gfds}
	u_{kh} |_{[t_{k-1},t_k]} = u_h^{k-1}+\left(\frac{t-t_{k-1}}{\Delta t}\right)(u_h^{k}-u_h^{k-1}), \qquad 1\leq k\leq N, \qquad \qquad \text{for } t\in [t_{k-1},t_k]. 
\end{align}
Note here that $\partial_tu_{kh} = \frac{u_h^k - u_h^{k-1}}{\Delta t}$.
The continuous time Galerkin approximation of the problem \eqref{AFGBHEM} is defined in the following way: Find $u_h^k\in V_h,$ for $k=1,2,\ldots,N$ such that for $\chi\in V_h$:
\begin{equation}\label{fd1}
	\left\{
	\begin{aligned}
		\left(\frac{u_h^k-u_h^{k-1}}{\Delta t},\chi\right)+\nu a(u_h^k,\chi)&+ \alpha b(u_h^k,u_h^k,\chi)+\Big(\sum_{j=1}^k \omega_{kj}(\Delta t) \nabla u_{h}^j,\nabla \chi\Big) \\&=\beta(u_h^k(1-(u_h^k)^{\delta})((u_h^k)^{\delta}-\gamma),\chi)+\langle f^k,\chi\rangle,\\
		(u_h^0(x_i),\chi)&=(u_0(x_i),\chi), \ \text{ for }\  i=1,2,\ldots,N,
	\end{aligned}
	\right.
\end{equation}    
$f^k = (\Delta t)^{-1}\int_{t_{k-1}}^{t_k} f(s)  \mathrm{d}s$ for $f\in \L^2(0,T;\H^{-1}(\Omega))$, $u^0_h$ approximates $u_0$ in $V_h$ and $x_i,$ $i=1,\ldots,n,$ represent the nodes of the triangulation $\mathcal{T}^h$ of $\Omega$. 

The existence of the discrete solution can be proved using an application of Brouwer’s fixed point theorem [Lemma 1.4, \cite{Te}] as done in Theorem 3.1 \cite{KMR}.
To show the error estimates, we need a stability result. It is natural to see that for the quadratic form analogous to the double integral in \eqref{pk}, we have
\begin{align}\label{pkd}
	\sum_{k=1}^N \left(\sum_{j=1}^{k}\omega_{kj}\Delta t \nabla u_h^{j} ,\nabla u_h^k\right) \geq 0,
\end{align}
as discussed in \cite{MMW} and Lemma 4.7 \cite{MTh}. The following lemma is useful for the analysis:
\begin{lemma}\label{lemma3.1}
Assume $ f\in \L^2(0,T;\H^{-1}(\Omega)).$ Then the set $\{f^k\}_{k=1}^N$ defined by $f^k = (\Delta t)^{-1}\int_{t_{k-1}}^{t_k} f(s)  \mathrm{d}s$ satisfies 
	\begin{align*}
		\Delta t \sum_{k=1}^N\|f^k\|_{\H^{-1}}^2 \leq C \|f\|^2_{\L^2(0,T;\H^{-1}(\Omega))}
	\end{align*}
	and 
	\begin{align*}
		\sum_{k=1}^N\int_{t_{k-1}}^{t_k}\|f^k-f(t)\|_{\H^{-1}}^2 \mathrm{d}t \rightarrow 0 \quad \Delta t \rightarrow 0.
	\end{align*}
	If $f\in H^{\gamma}(0,T;\H^{-1}(\Omega))$ for some $\gamma\in [0,1]$, then 
	\begin{align*}
		\sum_{k=1}^N\int_{t_{k-1}}^{t_k}\|f^k-f(t)\|_{\H^{-1}}^2\mathrm{d}t \leq C (\Delta t)^{2\gamma}\|f\|^2_{\H^{\gamma}(0,T;\H^{-1}(\Omega))}.
	\end{align*}
\end{lemma}
	\begin{proof}
		For a proof, we refer to Lemma 3.2, \cite{LSW}.
	\end{proof}

Let us now discuss the stability of the fully-discrete approximation \eqref{fd1}.
\begin{lemma}[Stability]\label{4.1}
	Assume $f\in \L^2(0,T;\H^{-1}(\Omega))$ and $u_h^0 \in \L^2(\Omega)$. Let $\{u_h^k\}_{k=1}^N \subset V_h$ be defined by \eqref{fd1}. Then 
	\begin{align}\label{3p33}
	\sum_{k=1}^{N}	\Delta t\|\nabla u_h^k\|_{\L^2}^2 \leq \left(\frac{1}{\nu }\|f\|^2_{\L^2(0,T;H^{-1}(\Omega))} +\| u_h^0\|_{\L^2}^2\right) \times e^{\beta T(1+\gamma)^2}.\end{align}
\end{lemma}  
\begin{proof}
	Suppose $f\in \L^2(0,T;\H^{-1}(\Omega))$, taking $\chi = u_h^k$ in \eqref{fd1}, we obtain 
	\begin{align*}
		&\frac{1}{2\Delta t}\|u_h^k\|_{\L^2}^2 - \frac{1}{2\Delta t}\|u_h^{k-1}\|_{\L^2}^2 + \frac{1}{2\Delta t}\|u_h^k-u_h^{k-1}\|_{\L^2}^2+\nu a(u_h^k,u_h^k)+ \alpha b(u_h^k,u_h^k,u_h^k) \\&+\eta \left(\sum_{j=1}^{k}\Delta t \omega_{kj}\nabla u_h^{j} ,\nabla u_h^k\right)=\beta(u_h^k(1-(u_h^k)^{\delta})((u_h^k)^{\delta}-\gamma),u_h^k)+\langle f^k,u_h^k\rangle.
	\end{align*}
	Using \eqref{7a}, \eqref{7}, Cauchy-Schwarz and Young’s inequalities,  we have 
	\begin{align*}
		&\frac{1}{2\Delta t}\|u_h^k\|_{\L^2}^2 - \frac{1}{2\Delta t}\|u_h^{k-1}\|_{\L^2}^2 + \frac{1}{2\Delta t}\|u_h^k-u_h^{k-1}\|_{\L^2}^2+\nu \|\nabla u_h^k\|_{\L^2}^2+ \beta\gamma\|u_h^k\|_{\L^2}^2+\beta\|u_h^k\|_{\L^{2(\delta+1)}}^{2(\delta+1)}\\&+\eta \left(\sum_{j=1}^{k}\Delta t \omega_{kj}\nabla u_h^{j} ,\nabla u_h^k\right)\leq \frac{\beta}{2}\| u_h^k\|^{2(\delta+1)}_{\L^{2(\delta+1)}}+\frac{\beta(1+\gamma)^2}{2}\|u_h^k\|_{\L^2}^2+ \frac{1}{2\nu}\|f^k\|_{\H^{-1}}^2 + \frac{\nu}{2} \|\nabla u_h^k\|_{\L^2}^2.
	\end{align*}
	This may be simplified as 
	\begin{align*}
		&\frac{1}{2\Delta t}\|u_h^k\|_{\L^2}^2 - \frac{1}{2\Delta t}\|u_h^{k-1}\|_{\L^2}^2+\frac{\nu}{2} \|\nabla u_h^k\|_{\L^2}^2+\eta  \left(\sum_{j=1}^{k}\omega_{kj}\Delta t \nabla u_h^{j} ,\nabla u_h^k\right)\leq \frac{\beta(1+\gamma)^2}{2}\|u_h^k\|_{\L^2}^2+ \frac{1}{2\nu}\|f^k\|_{\H^{-1}}^2. 
	\end{align*}
	Summing over $k$ for $k = 1, 2, \ldots, N$ and using positivity of the Kernel $K(\cdot)$ \eqref{pkd} and Lemma \ref{lemma3.1}, we have 
	\begin{align*}
		&\frac{1}{\Delta t}\|u_h^N\|_{\L^2}^2 - \frac{1}{\Delta t}\|u_h^0\|_{\L^2}^2+\nu \sum_{k=1}^N\|\nabla u_h^k\|_{\L^2}^2\leq \beta(1+\gamma)^2\sum_{k=1}^N\|u_h^k\|_{\L^2}^2 + \frac{1}{\nu\Delta t }\Delta t\sum_{k=1}^N\|f^k\|_{\H^{-1}}^2. 
	\end{align*}
	Finally, an application of the discrete Gronwall inequality (Lemma 9 \cite{AKP}) yields  the required bound \eqref{3p33}. 
\end{proof}
\begin{lemma}\label{thm7.2} 
	Let  $1\leq \delta < \infty$ for $d=2$ and, $1\leq \delta \leq 2$ for $d=3$ and assume that $u_0\in \H^2(\Omega)\cap \H_0^1(\Omega)$ and $f\in \H^1(0,T;\L^2(\Omega))$. Let $u_h$ be the semi-discrete solution \eqref{7p1} and $u_{kh}$ be the generic fully-discrete solution defined in \eqref{gfds}, then
	\begin{align*}
		\nonumber\|u_h - u_{kh}\|_{\L^{\infty}(0,T;\L^2(\Omega))}^2 + \|u_h - u_{kh}\|_{\L^2(0,T;\H_0^1(\Omega))}^2&\leq C(\Delta t)^2(\|f\|^2_{\H^1(0,T;\L^2(\Omega))} + \|u_0\|^2_{\H^2})\\&\quad+ \eta^2(\Delta t)^2\sup_{k,j} \omega_{kj}^2(\|f\|^2_{\H^1(0,T;\L^2(\Omega))} + \|u_0\|^2_{\H^2}).
	\end{align*}
\end{lemma}
\begin{proof} We divide our proof into two steps: In Step 1, we prove the error estimates on the nodal values and in Step 2, we prove the final estimate. 
	\vskip 2mm
	\noindent\textbf{Step 1:} \emph{Estimates at $t_k :$  }	
	Integrating the semi-discrete scheme \eqref{7p1} from $t_{k-1}$ to $ t_k$, we obtain
	\begin{align}\label{sdi}
		&\nonumber (u_h({t_{k}})-u_h(t_{k-1}),\chi)+\nu \left(\int_{t_{k-1}}^{t_k}\nabla u_h(t) \mathrm{~d} t,\nabla \chi\right)+\alpha \left(\int_{t_{k-1}}^{t_k} B(u_h(t)) \mathrm{~d} t,\chi
		\right)\\&\qquad+\eta\left(\int_{t_{k-1}}^{t_k}(K*\nabla u_h)(t)\mathrm{~d} t,\nabla \chi\right) = \beta\left(\int_{t_{k-1}}^{t_k}c(u_h(t))\mathrm{~d} t,\chi\right)+\left( \int_{t_{k-1}}^{t_k}f(t)\mathrm{~d} t,\chi\right).
	\end{align}
	The fully-discrete scheme \eqref{fd1}  is given as
	\begin{align}\label{fd3}
		&\nonumber	\left(\frac{u_h^k-u_h^{k-1}}{\Delta t},\chi\right)+\nu (\nabla u_h^k,\nabla \chi)+ \alpha b(u_h^k,u_h^k,\chi)+\Big(\sum_{j=1}^k \omega_{kj} \Delta t \nabla u_{h}^k,\nabla \chi\Big) \\&\qquad=\beta(u_h^k(1-(u_h^k)^{\delta})((u_h^k)^{\delta}-\gamma),\chi)+( f^k,\chi).
	\end{align}
	Subtracting \eqref{fd3} from \eqref{sdi}, we have 
	\begin{align*}
		&(u_h(t_k)-u_h^k,\chi)-(u_h(t_{k-1})-u_h^{k-1},\chi)+\nu \left(\int_{t_{k-1}}^{t_k}\nabla u_h(t) \mathrm{~d}t-\Delta t\nabla u_h^k,\nabla \chi\right) +\eta\left(\int_{t_{k-1}}^{t_k}(K*\nabla u_h)(t)\mathrm{~d} t,\nabla \chi\right) \\&- \Delta t \eta \left(\sum_{j=1}^{k}\Delta t \omega_{kj}\nabla u_h^{j} ,\nabla\chi\right) +\alpha \left(\left(\int_{t_{k-1}}^{t_k} B(u_h(t)) \mathrm{~d} t\right) - \Delta t B(u_h^k),\chi
		\right)=\beta\left(\int_{t_{k-1}}^{t_k}c(u_h(t))\mathrm{~d} t - \Delta tc(u_h^k),\chi\right).
	\end{align*}
	Rewriting the above equation and taking $\chi = u_h(t_k)-u_h^k$ , we deduce
	\begin{align*}
		&\|u_h(t_k)-u_h^k\|_{\L^2}^2+\nu \Delta t \|\nabla (u_h(t_k) -u_h^k)\|_{\L^2}^2 + \Delta t \eta \left(\sum_{j=1}^{k}\omega_{kj}\Delta t\nabla (u_h(t_j)-u_h^{j}) ,\nabla(u_h(t_k) -u_h^k)\right) \\&=	(u_h(t_{k-1})-u_h^{k-1},u_h(t_k) -u_h^k) - \nu \left(\int_{t_{k-1}}^{t_k}\nabla u_h(t) \mathrm{~d}t-\Delta t\nabla u_h(t_k),\nabla (u_h(t_k)-u_h^k)\right)\\&\quad + \eta \left(\Delta t \sum_{j=1}^{k}\omega_{kj}\Delta t\nabla u_h(t_j) - \int_{t_{k-1}}^{t_k}(K*\nabla u_h)(t)\mathrm{~d} t,\nabla (u_h(t_k)-u_h^k)\right)  \\&\quad-\alpha \left(\int_{t_{k-1}}^{t_k} B(u_h(t)) \mathrm{~d} t- \Delta t B(u_h(t_k)),u_h(t_k)-u_h^k
		\right)-\alpha \Delta t\left(B(u_h(t_k) ) -  B(u_h^k),u_h(t_k)-u_h^k
		\right)\\&\quad+\beta \left(\int_{t_{k-1}}^{t_k} c(u_h(t)) \mathrm{~d} t - \Delta t c(u_h(t_k)),u_h(t_k)-u_h^k
		\right)+\beta \Delta t\left(c(u_h(t_k) ) -  c(u_h^k),u_h(t_k)-u_h^k
		\right).
	\end{align*}
	Using Cauchy-Schwarz and Young's inequalities, we can estimate the first term on right hand side as  
	\begin{align}
		(u_h(t_{k-1})-u_h^{k-1},u_h(t_k) -u_h^k) \leq \frac{1}{2}\|u_h(t_{k-1})-u_h^{k-1}\|_{\L^2}^2 + \frac{1}{2}\|u_h(t_k) -u_h^k\|_{\L^2}^2\label{fde1},
	\end{align}
	Again an application of the Cauchy-Schwarz and Young's inequalities yield
	\begin{align*}
		&\nonumber\nu \left(\int_{t_{k-1}}^{t_k}\nabla u_h(t) \mathrm{~d}t-\Delta t\nabla u_h(t_k),\nabla (u_h(t_k)-u_h^k)\right)\\&\nonumber\leq \nu\Big\|\int_{t_{k-1}}^{t_k}\nabla u_h(t)\mathrm{~d}t-\Delta t\nabla u_h(t_k) \Big\|_{\L^2}\big\|\nabla (u_h(t_k)-u_h^k)\big\|_{\L^2}\\&\leq \frac{2}{\nu\Delta t}\Big\|\int_{t_{k-1}}^{t_k}\nabla u_h(t)\mathrm{~d}t -\Delta t\nabla u_h(t_k) \Big\|_{\L^2}^2+\frac{\nu\Delta t}{8}\big\|\nabla (u_h(t_k)-u_h^k)\big\|_{\L^2}^2,
	\end{align*}
	where the first term on the right hand side can be estimated as 
	\begin{align}\label{int1}
		\nonumber&\frac{1}{\Delta t}\big\|\int_{t_{k-1}}^{t_k}\nabla u_h(t)\mathrm{~d}t-\Delta t\nabla u_h(t_k) \big\|_{\L^2}^2 = \frac{1}{\Delta t}\big\|\int_{t_{k-1}}^{t_k}(\nabla u_h(t)-\nabla u_h(t_k)) \mathrm{~d}t\big\|_{\L^2}^2 \\&\nonumber= \frac{1}{\Delta t}\big\|\int_{t_{k-1}}^{t_k}\int_{t_k}^{t}\nabla \partial_tu_h(s) \mathrm{~d}s\mathrm{~d}t\big\|_{\L^2}^2\leq \frac{1}{\Delta t}\left(\int_{t_{k-1}}^{t_k}\int_{t_{k-1}}^{t_k}\| \partial_t \nabla u_h(s)\|_{\L^2} \mathrm{~d}s\mathrm{~d}t\right)^2\\&\leq\Delta t\left(\int_{t_{k-1}}^{t_k}\| \partial_t \nabla u_h(s)\|_{\L^2} \mathrm{~d}s\right)^2\leq (\Delta t)^2\int_{t_{k-1}}^{t_k} \| \partial_t \nabla u_h(s)\|_{\L^2}^2 \mathrm{~d}s.
	\end{align}
	The memory term can be estimated as 
	\begin{align*}
		& \eta \left(\Delta t \sum_{j=1}^{k}\Delta t\omega_{kj}\nabla u_h(t_j) - \int_{t_{k-1}}^{t_k}(K*\nabla u_h)(t)\mathrm{~d} t,\nabla (u_h(t_k)-u_h^k)\right) \\&\leq  \eta \Big\|\Delta t \sum_{j=1}^{k}\Delta t\omega_{kj}\nabla u_h(t_j) - \int_{t_{k-1}}^{t_k}(K*\nabla u_h)(t)\mathrm{~d} t\Big\|_{\L^2}\|\nabla (u_h(t_k)-u_h^k)\|_{\L^2} \\& \leq \frac{2\eta^2}{\nu\Delta t} \Big\|(\Delta t)^2 \sum_{j=1}^{k}\omega_{kj}\nabla u_h(t_j) - \int_{t_{k-1}}^{t_k}(K*\nabla u_h)(t)\mathrm{~d} t\Big\|^2_{\L^2}+\frac{\nu\Delta t}{8}\|\nabla (u_h(t_k)-u_h^k)\|^2_{\L^2},
	\end{align*}
	where the term $\frac{2\eta^2}{\nu\Delta t} \big\|(\Delta t)^2 \sum_{j=1}^{k}\omega_{kj}\nabla u_h(t_j) - \int_{t_{k-1}}^{t_k}(K*\nabla u_h)(t)\mathrm{~d} t\big\|^2_{\L^2}$ can be estimated using \eqref{int1} as
	\begin{align*}
		&\frac{1}{\Delta t}\Big\|(\Delta t)^2 \sum_{j=1}^{k}\omega_{kj}\nabla u_h(t_j) - \int_{t_{k-1}}^{t_k}(K*\nabla u_h)(t)\mathrm{~d} t\Big\|_{\L^2}^2 \\&=	\frac{1}{\Delta t}\Big\|\int_{t_{k-1}}^{t_k}\sum_{j=1}^k\left(\int_{t_{j-1}}^{\min(t,t_j)}K(t-s)(\nabla u_h(t_j)- \nabla u_h(s))\mathrm{~d}s  \right)\mathrm{~d} t\Big\|^2_{\L^2} \\&\leq\frac{1}{\Delta t}\Big\|\sum_{j=1}^k\int_{t_{k-1}}^{t_k}\left(\int_{t_{j-1}}^{\min(t,t_j)}|K(t-s)|\int_{t_{j-1}}^{t_j}|\partial_t\nabla u_h(\tau)|\mathrm{~d}\tau  \right)\mathrm{~d}s\mathrm{~d} t\Big\|^2_{\L^2}\\&\leq\frac{1}{\Delta t}\Big\|\sum_{j=1}^k(\Delta t)^2\left(\frac{1}{(\Delta t)^2}\int_{t_{k-1}}^{t_k}\int_{t_{j-1}}^{\min(t,t_j)}|K(t-s)| \mathrm{~d}s\mathrm{~d} t\int_{t_{j-1}}^{t_j}|\partial_t\nabla u_h(\tau)|\mathrm{~d}\tau  \right)\Big\|^2_{\L^2}\\&\leq (\Delta t)^3\Big\|\sum_{j=1}^k \overline{K}_{kj}\int_{t_{j-1}}^{t_j}|\partial_t\nabla u_h(\tau)|\mathrm{~d}\tau\Big\|^2_{\L^2}
		\\&\leq k(\Delta t)^4\sum_{j=1}^k \overline{K}_{kj}^2\int_{t_{j-1}}^{t_j}\|\partial_t\nabla u_h(\tau)\|_{\L^2}^2\mathrm{~d}\tau\\&\leq T(\Delta t)^3\sum_{j=1}^k \overline{K}_{kj}^2\int_{t_{j-1}}^{t_j}\|\partial_t\nabla u_h(\tau)\|_{\L^2}^2\mathrm{~d}\tau,
	\end{align*}
	where $\overline{K}_{kj} =  \frac{1}{(\Delta t)^2}\int_{t_{k-1}}^{t_k}\int_{t_{j-1}}^{\min(t,t_j)}|K(t-s)| \mathrm{~d}s\mathrm{~d} t$.
	An application of integration by parts and Cauchy-Schwarz  inequality gives 
	\begin{align}\label{nl1}
		\nonumber	&\alpha \left(\int_{t_{k-1}}^{t_k} B(u_h(t)) \mathrm{~d} t - \Delta t B(u_h(t_k)),u_h(t_k)-u_h^k
		\right) \\&\nonumber= \frac{\alpha}{\delta+1}\sum_{i=1}^{d}\left(\int_{t_{k-1}}^{t_k} u_h^{\delta+1}(t) \mathrm{d}t - \Delta tu_h^{\delta+1}(t_k), \frac{\partial}{\partial x_i}(u_h(t_k)-u_h^k)\right)\\&\nonumber \leq \frac{\alpha}{\delta+1}\Big\|\int_{t_{k-1}}^{t_k} u_h^{\delta+1}(t) \mathrm{d}t - \Delta t u_h^{\delta+1}(t_k)\Big\|_{\L^2} \big\|\nabla(u_h(t_k)-u_h^k)\big\|_{\L^2}\\&\leq \frac{2\alpha^2}{\nu(\delta+1)^2\Delta t}\Big\|\int_{t_{k-1}}^{t_k} u_h^{\delta+1}(t) \mathrm{d}t - \Delta t u_h^{\delta+1}(t_k)\Big\|_{\L^2}^2+\frac{\nu\Delta t}{8}\big\|\nabla(u_h(t_k)-u_h^k)\big\|^{2}_{\L^2}.
	\end{align}
	Now, we estimate  $\frac{1}{\Delta t}\big\|\int_{t_{k-1}}^{t_k}u_h^{\delta+1}(t)\mathrm{~d}t-\Delta t u_h^{\delta+1}(t_k) \big\|_{\L^2}^2 $ similar to \eqref{int1} as   
	\begin{align*}
		&\frac{1}{\Delta t}\Big\|\int_{t_{k-1}}^{t_k}u_h^{\delta+1}(t)\mathrm{~d}t-\Delta t u_h^{\delta+1}(t_k)\Big\|_{\L^2}^2 = \frac{1}{\Delta t}\Big\|\int_{t_{k-1}}^{t_k}( u_h^{\delta+1}(t)-u_h^{\delta+1}(t_k)) \mathrm{~d}t\Big\|_{\L^2}^2\\&= \frac{1}{\Delta t}\Big\|\int_{t_{k-1}}^{t_k}\int_{t_k}^{t} \partial_tu_h^{\delta+1}(s) \mathrm{~d}s\mathrm{~d}t\Big\|_{\L^2}^2\leq \frac{1}{\Delta t}\left(\int_{t_{k-1}}^{t_k}\int_{t_{k-1}}^{t_k}\| \partial_tu_h^{\delta+1}(s)\|_{\L^2} \mathrm{~d}s\mathrm{~d}t\right)^2\\&\leq\Delta t\left(\int_{t_{k-1}}^{t_k}\| \partial_tu_h^{\delta+1}(s)\|_{\L^2} \mathrm{~d}s\right)^2\leq (\Delta t)^2\int_{t_{k-1}}^{t_k} \| \partial_tu_h^{\delta+1}(s)\|_{\L^2}^2 \mathrm{~d}s \\&=  (\Delta t)^2(\delta+1)^2\int_{t_{k-1}}^{t_k} \|u_h^{\delta}(s)\partial_tu_h(s)\|_{\L^2}^2\mathrm{~d}s.
	\end{align*}
	Finally the reaction term can be simplified as 
	\begin{align*}
		\left(\int_{t_{k-1}}^{t_k} c(u_h(t)) \mathrm{~d} t\right)& - \Delta t c(u_h(t_k))= \beta (1+\gamma)\left(\int_{t_{k-1}}^{t_k} u_h^{\delta+1}(t) \mathrm{~d} t-\Delta t u_h^{\delta+1}(t_k)\right)\\&-\beta\gamma\left(\left(\int_{t_{k-1}}^{t_k} u_h(t) \mathrm{~d} t\right)-\Delta t u_h(t_k)\right)-\beta\left(\int_{t_{k-1}}^{t_k} u_h^{2\delta+1}(t) \mathrm{~d} t-\Delta t u_h^{2\delta+1}(t_k)\right).
	\end{align*}
	Firstly, we estimate the term with coefficient $\gamma$, using Cauchy-Schwarz inequality and the estimate similar to \eqref{int1} as
	\begin{align*}
		\nonumber\beta\gamma\left(\int_{t_{k-1}}^{t_k} u_h(t) \mathrm{~d} t-\Delta t u_h(t_k),u_h(t_k)-u_h^k \right) &\leq \frac{(\beta\gamma)^2}{\nu\Delta t}\Big\|\int_{t_{k-1}}^{t_k} u_h(t) \mathrm{~d} t-\Delta t u_h(t_k)\Big\|_{\L^2}^2 +\frac{\nu\Delta t}{8}\|u_h(t_k)-u_h^k\|^{2}_{\L^2}\\&\leq \frac{(\beta\gamma)^2(\Delta t)^2}{\nu}\int_{t_{k-1}}^{t_k} \| \partial_tu_h(s)\|_{\L^2}^2 \mathrm{~d}s +\frac{\nu\Delta t}{8}\|(u_h(t_k)-u_h^k)\|^{2}_{\L^2}.
	\end{align*}
	Using an estimate similar to \eqref{nl1}, we have
	\begin{align*}
		&\beta(1+\gamma)\left(\int_{t_{k-1}}^{t_k} u_h^{\delta+1}(t) \mathrm{~d} t-\Delta t u_h^{\delta+1}(t_k), u_h(t_k)-u_h^k  \right)\\&\leq  \frac{(\Delta t)^2(1+\gamma)^2\beta^2(\delta+1)^2}{\nu}\int_{t_{k-1}}^{t_k} \|u_h^{\delta}(s)\partial_tu_h(s)\|_{\L^2}^2 + \frac{\nu\Delta t}{8}\|u_h(t_k)-u_h^k\|^{2}_{\L^2}.
	\end{align*}
	Also for the final term, we obtain
	\begin{align*}
		&\Big|\beta\left(\int_{t_{k-1}}^{t_k} u_h^{2\delta+1}(t) \mathrm{~d} t-\Delta t u_h^{2\delta+1}(t_k), u_h(t_k)-u_h^k  \right)\Big|= \Big|\beta\int_{t_{k-1}}^{t_k} \left(u_h^{2\delta+1} (t)-u_h^{2\delta+1} (t_k), u_h(t_k)-u_h^k  \right)\mathrm{~d} t\Big|\\& = \Big|\beta\int_{t_{k-1}}^{t_k}\left(\int_{t_{k}}^{t}\partial_tu_h^{2\delta+1}(s) \mathrm{~d}s,u_h(t_k)-u_h^k  \right)\mathrm{~d} t\Big| \leq \beta(2\delta+1) \int_{t_{k-1}}^{t_k}\int_{t_{k-1}}^{t_k}\left|\left(u_h^{2\delta}(s)\partial_tu_h(s),u_h(t_k)-u_h^k  \right)\right| \mathrm{~d}s\mathrm{~d} t\\&\leq \beta(2\delta+1)\Delta t\int_{t_{k-1}}^{t_k}\left|\left(u_h^{\delta}(t)(u_h^{\delta}(t)\partial_tu_h(t)) ,u_h(t_k)-u_h^k  \right)\right| \mathrm{~d} t\\&\leq \beta(2\delta+1)\Delta t \int_{t_{k-1}}^{t_k}\|u_h^{\delta}(t)\|_{\L^{\frac{2(\delta+1)}{\delta}}}\|u_h^{\delta}(t)\partial_tu_h(t)\|_{\L^2}\|u_h(t_k)-u_h^k\|_{\L^{2(\delta+1)}}   \mathrm{~d} t \\&\leq \beta(2\delta+1)\Delta t\|u_h\|_{\L^{\infty}((t_{k-1},t_k);\L^{2(\delta+1)}(\Omega))}^{\delta}\int_{t_{k-1}}^{t_k}\|u_h^{\delta}(t)\partial_tu_h(t)\|_{\L^2}\|\nabla(u_h(t_k)-u_h^k)\|_{\L^{2}}   \mathrm{~d} t \\&\leq\frac{2\beta^2(2(\delta+1))^2(\Delta t)^2}{\nu}\|u_h\|_{\L^{\infty}(0,T; \L^{2(\delta+1)}(\Omega))}^{2\delta}\int_{t_{k-1}}^{t_k}\|u_h^{\delta}(t)\partial_tu_h(t)\|^2_{\L^2}  \mathrm{~d} t +\frac{\nu\Delta t}{8} \|\nabla(u_h(t_k)-u_h^k)\|^2_{\L^{2}},
	\end{align*}
for $1\leq\delta<\infty$ ($d=2$) and $1\leq \delta\leq 2$ ($d=3$). For, $w= u_h(t_k)-u_h^k$,
we derive a bound for $-\alpha\Delta t(B(u_h(t_k))-B(u_h^k),w)$, using Taylor's formula, H\"older's, Gagliardo-Nirenberg interpolation  and Young's inequalities as 
\begin{align*}
	-\alpha \Delta t\left(B(u_h(t_k) ) -  B(u_h^k),u_h(t_k)-u_h^k
	\right)&=\frac{\alpha}{\delta+1} \left((u_h(t_k)^{\delta+1}-(u_h^k)^{\delta+1})\left(\begin{array}{c}1\\\vdots\\1\end{array}\right),\nabla w\right)\nonumber\\&=\alpha\left((u_h(t_k)-u_h^k)(\theta u_h(t_k)+(1-\theta) u_h^k)^{\delta}\left(\begin{array}{c}1\\\vdots\\1\end{array}\right),\nabla w\right)\nonumber\\&\leq 2^{\delta-1} \alpha(\|u_h(t_k)\|^{\delta}_{\L^{2(\delta+1)}}+\|u_h^k\|^{\delta}_{\L^{2(\delta+1)}})\|w\|_{\L^{2(\delta+1)}}\|\nabla w\|_{\L^2}
	\nonumber\\&\leq 2^{\delta-1}C \alpha(\|u_h(t_k)\|^{\delta}_{\L^{2(\delta+1)}}+\|u_h^k\|^{\delta}_{\L^{2(\delta+1)}})\|\nabla w\|_{\L^{2}}^{\frac{(2+d)\delta+2}{2(\delta+1)}}\| w\|_{\L^2}^{\frac{(2-d)\delta+2}{2(\delta+1)}}\nonumber\\&\leq \frac{\nu}{8}\|\nabla w\|_{\L^2}^2+C(\alpha,\nu)\left(\|u_h(t_k)\|_{\L^{2(\delta+1)}}^{\frac{4\delta(\delta+1)}{(2-d)\delta+2}}+\|u_h^k\|_{\L^{2(\delta+1)}}^{\frac{4\delta(\delta+1)}{(2-d)\delta+2}}\right)\|w\|_{\L^2}^2.
\end{align*}
where $C(\alpha, \nu) = \left(\frac{2((2+d)\delta+2)}{\nu(\delta+1)}\right)^{\frac{(2-d)\delta+2}{(2+d)\delta+2}}\times\left(\frac{(2-d)\delta+2}{4(\delta+1)}\right)(2^{\delta-1}\alpha)^{\frac{4(\delta+1)}{(2-d)\delta+2}}$.
	Combining the above estimates and using the calculations similar to \eqref{26}, we deduce
	\begin{align*}
		&\frac{1}{2}\|u_h(t_k)-u_h^k\|_{\L^2}^2-\frac{1}{2}\|u_h(t_{k-1})-u_h^{k-1}\|_{\L^2}^2+\frac{\nu}{2} \Delta t \|\nabla (u_h(t_k) -u_h^k)\|_{\L^2}^2\\&\quad+ \Delta t \eta \left(\sum_{j=1}^{k}\omega_{kj}\Delta t \nabla (u_h(t_j)-u_h^{j}) ,\nabla(u_h(t_k) -u_h^k)\right) +\frac{\beta\Delta t}{2}\|u_h(t_k)^{\delta}(u_h(t_k)-u_h^k)\|_{\L^2}^2\\&\quad+\frac{\beta\Delta t}{2}\|(u_h^k)^{\delta}(u_h(t_k)-u_h^k)\|_{\L^2}^2+\beta\gamma\Delta t\|(u_h(t_k)-u_h^k)\|_{\L^2}^2\\& \leq\frac{2\alpha^2(\Delta t)^2}{\nu}\int_{t_{k-1}}^{t_k} \| \partial_t \nabla u_h(s)\|_{\L^2}^2 \mathrm{~d}s 
		+ \frac{(\Delta t)^2\beta^2\gamma^2}{\nu}
		\int_{t_{k-1}}^{t_k} \|\partial_tu_h(s)\|_{\L^2}^2 \mathrm{d}s \\&\quad+\Delta t\left(C(\alpha,\nu)\left(\|u_h(t_k)\|_{\L^{2(\delta+1)}}^{\frac{4\delta(\delta+1)}{(2-d)\delta+2}}+\|u_h^k\|_{\L^{2(\delta+1)}}^{\frac{4\delta(\delta+1)}{(2-d)\delta+2}}\right)+\frac{\beta}{2}2^{2\delta}(1+\gamma)^2(\delta+1)^2 + \frac{3\nu}{8}\right)\|u_h(t_k)-u_h^k\|_{\L^2}^2\\&\quad +\frac{(\Delta t)^2}{\nu} \left(2\alpha^2 +  \beta^2(1+\gamma)^2(\delta+1)^2\right)\int_{t_{k-1}}^{t_k} \|u_h^{\delta}(s)\partial_tu_h(s)\|_{\L^2}^2  \mathrm{~d}s + \sum_{j=1}^k \frac{2\eta^2T(\Delta t)^3}{\nu}\overline{K}_{kj}^2\int_{t_{j-1}}^{t_j}\|\partial_t\nabla u_h(\tau)\|_{\L^2}^2\mathrm{~d}\tau.
	\end{align*}
	Summing over all $k = 1,2,\ldots, N$ , and using the positivity of the kernel (cf. \eqref{pkd}), we attain
	\begin{align}
		&\|u_h(t_N)-u_h^N\|_{\L^2}^2+\sum_{k=1}^N\nu \Delta t \|\nabla (u_h(t_k) -u_h^k)\|_{\L^2}^2\nonumber\\& \leq\frac{4\alpha^2(\Delta t)^2}{\nu}\int_{0}^{T} \| \partial_t \nabla u_h(s)\|_{\L^2}^2 \mathrm{~d}s
		+\frac{2(\Delta t)^2\beta^2\gamma^2}{\nu}
		\int_{0}^{T} \|\partial_tu_h(s)\|_{\L^2}^2 \mathrm{~d}s   \nonumber \\&\quad+2\Delta t \sum_{k=1}^N\left(C(\alpha,\nu)\left(\|u_h(t_k)\|_{\L^{2(\delta+1)}}^{\frac{4\delta(\delta+1)}{(2-d)\delta+2}}+\|u_h^k\|_{\L^{2(\delta+1)}}^{\frac{4\delta(\delta+1)}{(2-d)\delta+2}}\right)+\frac{\beta}{2}2^{2\delta}(1+\gamma)^2(\delta+1)^2 + \frac{3\nu}{8}\right)\|u_h(t_k)-u_h^k\|_{\L^2}^2\nonumber\\&\quad+\frac{2(\Delta t)^2}{\nu} \left(2\alpha^2 + C \beta^2(1+\gamma)^2(\delta+1)^2\right)\int_{0}^{T} \|u_h(s)^{\delta}\partial_tu_h(s)\|_{\L^2}^2\mathrm{~d}s\nonumber \\&\quad + \frac{4\eta^2T(\Delta t)^3}{\nu}\sum_{k=1}^N\sum_{j=1}^k \overline{K}_{kj}^2\int_{t_{j-1}}^{t_j}\|\partial_t\nabla u_h(\tau)\|_{\L^2}^2\mathrm{~d}\tau.\label{329}
	\end{align}
	For the memory term, we get 
	\begin{align*}
		T(\Delta t)^3\sum_{k=1}^N\sum_{j=1}^k \overline{K}_{kj}^2\int_{t_{j-1}}^{t_j}\|\partial_t\nabla u_h(\tau)\|_{\L^2}^2\mathrm{~d}\tau & \leq \sup_{k,j} \overline{K}_{kj}^2 T(\Delta t)^3\sum_{k=1}^N\sum_{j=1}^k \int_{t_{j-1}}^{t_j}\|\partial_t\nabla u_h(\tau)\|_{\L^2}^2\mathrm{~d}\tau\\& \leq \sup_{k,j}\overline{K}_{kj}^2 T^2(\Delta t)^2 \int_{0}^{T}\|\partial_t\nabla u_h(\tau)\|_{\L^2}^2\mathrm{~d}\tau.
	\end{align*}
	Using the  above estimate  and  an application of Gronwall inequality in \eqref{329} yields 
	\begin{align}\label{fdest}
		&\nonumber\|u_h(t_N)-u_h^N\|_{\L^2}^2+\sum_{k=1}^N\nu \Delta t \|\nabla (u_h(t_k) -u_h^k)\|_{\L^2}^2\\&\nonumber \leq C(\Delta t)^2\left(\int_{0}^{T} \| \partial_t \nabla u_h(s)\|_{\L^2}^2 \mathrm{~d}s +\int_{0}^{T} \|u_h(s)^{\delta}\partial_tu_h(s)\|_{\L^2}^2+\frac{4\eta^2T^2}{\nu}\sup_{k,j} \overline{K}_{kj}^2 \int_{0}^{T} \|\partial_tu_h(s)\|_{\L^2}^2 \mathrm{d}s \right)\\&\quad \times  \exp\left\{2\Delta t\left[C(\alpha,\nu)\left(\|u_h(t_k)\|_{\L^{2(\delta+1)}}^{\frac{4\delta(\delta+1)}{(2-d)\delta+2}}+\|u_h^k\|_{\L^{2(\delta+1)}}^{\frac{4\delta(\delta+1)}{(2-d)\delta+2}}\right)+\frac{\beta}{2}2^{2\delta}(1+\gamma)^2(\delta+1)^2 + \frac{3\nu}{8}\right]\right\}.
	\end{align}
	One can prove the  estimates obtained in \eqref{2M.68} for the semi-discrete solution $u_h$ also, so that the right hand side of \eqref{fdest} is bounded (see Remark \ref{rem3.2} below).	
	\vskip 2mm
	\noindent\textbf{Step 2:} \emph{Estimate for any $t\in [t_{k-1},t_k]$.  }
	Let us now introduce a linear interpolation  $\mathcal{I} u_h$ for the semi-discrete solution $u_h$ [\cite{YGU} (Section  3.1)]  by:
	$$\mathcal{I} u_h(t) =  u_h(t_{k-1})+\left(\frac{t-t_{k-1}}{\Delta t}\right)(u_h(t_{k})-u_h(t_{k-1}))   \qquad \text{for } t\in [t_{k-1},t_k].$$
	Let us write $u_h-u_{kh} = u_h -\mathcal{I}u_h + \mathcal{I}u_h -u_{kh}$, where $\mathcal{I}u_h $ is the linear interpolation as defined above. Using the triangle inequality, we have
	\begin{align*}
		\|u_h - u_{kh}\|_{\L^2(0,T;\H_0^1(\Omega))}^2 \leq  2 \|u_h-\mathcal{I}u_h\|_{\L^2(0,T;\H_0^1(\Omega))}^2 +2 \|\mathcal{I}u_h -u_{kh}\|_{\L^2(0,T;\H_0^1(\Omega))}^2.
	\end{align*}
	The first term can be estimated using (Lemma 3.2 [\cite{YGU}]) and the final term using  \eqref{fdest} as
	\begin{align}\label{sim}
		&\nonumber\|u_h-\mathcal{I} u_h\|_{\L^2(0,T;\H_0^1(\Omega))}^2 = \sum_{i=1}^N\int_{t_{i-1}}^{t_i}\|u_h(t)-\mathcal{I} u_h(t)\|_{\H_0^1(\Omega)}^2 \mathrm{d}t\leq C(\Delta t )^2 \int_{0}^{T}\left\|\partial_tu_h(t)\right\|_{\H_0^1}^2 \mathrm{d}t,\\&
		\|\mathcal{I}u_h -u_{kh}\|_{\L^2(0,T;\H_0^1(\Omega))}^2 = \sum_{i=1}^N\int_{t_{i-1}}^{t_i} \|\mathcal{I}u_h(t) -u_{kh}(t)\|_{\H_0^1(\Omega)}^2\mathrm{d}t \leq C
		  \sum_{i=1}^N\Delta t \|u_h(t_i) -u_h^i\|_{\H_0^1}^2.
	\end{align}
	Again using triangle inequality, we have 
	\begin{align*}
		\|u - u_{kh}\|_{\L^{\infty}(0,T;\L^2(\Omega))}^2 \leq 2 \|u_h-\mathcal{I}u_h\|_{\L^{\infty}(0,T;\L^2(\Omega))}^2 + 2\|\mathcal{I}u_h -u_{kh}\|_{\L^{\infty}(0,T;\L^2(\Omega))}^2.
	\end{align*}
	Similarly as in \eqref{sim}, by using (Corollary 3.1 [\cite{YGU}]) and the estimate \eqref{fdest}, we deduce
	\begin{align*}
		&\|u_h-\mathcal{I} u_h\|_{\L^{\infty}(0,T;\L^2(\Omega))}^2 \leq\sup_{1\leq i \leq N}\Big(\sup_{t_{i-1}\leq t\leq t_{i}}\|u_h-\mathcal{I} u_h\|^2_{\L^2(\Omega)}\Big) \leq C(\Delta t )^2 \|u_h\|_{W^{1,\infty}(0,T ; \L^2(\Omega))}^2,
		\\& \|\mathcal{I}u_h -u_{kh}\|_{\L^{\infty}(0,T;\L^2(\Omega))}^2 \leq\sup_{1\leq i \leq N}\Big(\sup_{t_{i-1}\leq t\leq t_{i}}\|\mathcal{I}u_h -u_{kh}\|^2_{\L^2(\Omega)}\Big) \leq C\sup_{1\leq i \leq N}\Big(\|u_h(t_i) -u_h^i\|^2_{\L^2(\Omega)}\Big), 
	\end{align*}
and the required result follows.
\end{proof}

The main result of this section is 
\begin{theorem}\label{thm7.3}
	Assume that $u   _0\in \H^2(\Omega)\cap \H_0^1(\Omega)$ and $f\in\H^1(0,T;\L^2(\Omega))$ such that $u\in\L^{\infty}(0,T;\H^2(\Omega)),$ $ u_t\in \L^{\infty}(0,T;\L^2(\Omega))\cap \L^{2}(0,T;\H_0^1(\Omega))$ be the solution of \eqref{weaksolution} on the interval $(0,T]$. Then, the finite element approximation $u_{kh}$ converges to $u$ as $\Delta t, h\rightarrow 0$. In addition, there exists a constant $C\geq0$ such that the approximation $u_{kh}$ satisfies the following error estimate:
	\begin{align*}
		& \|u- u_{kh}\|_{\L^{\infty}(0,T;\L^2(\Omega))}^2 + \|u - u_{kh}\|_{\L^2(0,T;\H_0^1(\Omega))}^2\nonumber\\&\leq \eta^2(\Delta t)^2\sup_{k,j} \overline{K}_{kj}^2(\|f\|^2_{\H^1(0,T;\L^2(\Omega))} + \|u_0\|^2_{\H^2})+ C((\Delta t)^2 + h^2)(\|f\|^2_{\H^1(0,T;\L^2(\Omega))} + \|u_0\|^2_{\H^2}).
	\end{align*}
\end{theorem}
\begin{proof}
	The proof is an application of triangle inequality, Theorem \ref{thm7.1} and Lemma \ref{thm7.2}.
\end{proof}
\begin{remark}\label{rem3.2}
	\begin{enumerate}
		\item 	In the literature most of the estimates needs higher regularity  $u_{tt}\in \L^1(0,T;\L^2(\Omega))$ but we have proved estimates under the minimal regularity assumptions.
		\item For the generalised Burgers- Huxley equation without memory (that is, $\eta =0$), we have obtained the optimal rate of convergence. Also, if we assume that $K(\cdot) \in \L^{\infty}(0,T),$ then $\overline{K}_{kj}^2 \leq C$, so that we obtain the optimal convergence again.
		\vspace{-2mm} 
		\item As proved in \cite{MTW} and \cite{MMu}, if we assume $|K(t-s) |\leq |t-s|^{-\alpha}$, for $0<\alpha<1,$ then we have $\overline{K}_{kj}^2 \leq O((\Delta t)^{-2\alpha})$ and in that case the error estimates are given as 
		\begin{align*}
			\|u- u_{kh}\|_{\L^{\infty}(0,T;\L^2(\Omega))}^2 + \|u - u_{kh}\|_{\L^2(0,T;\H_0^1(\Omega))}^2\leq  C((\Delta t)^{2(1-\alpha)}+h^2)(\|f\|^2_{\H^1(0,T;\L^2(\Omega))} + \|u_0\|^2_{\H^2}).
		\end{align*}
		For $\alpha = \frac{1}{2}$, we have sub-optimal convergence rate and as $\alpha\rightarrow0,$ the convergence tends to an optimal one.
		\vspace{-3mm} 
		\item If we assume $K(\cdot)\in \L^{p}(0,T)$, then we have $\overline{K}_{kj}^2\leq (\Delta t)^{\frac{2-2q}{q}}$, where $\frac{1}{p}+ \frac{1}{q} =1 $. Therefore, the error estimates are of order $((\Delta t)^{\frac{2}{q}} + h^2).$ For $p =  q =2$, that is, $K(\cdot)\in \L^2(0,T),$ then we have the error estimates are of order $(\Delta t + h^2)$.
			\vspace{-2mm} 
		\item If we use the time discretization with non uniform refinement as done in \cite{MTW} and \cite{MMu}, we can get the optimal convergence rate.
	\end{enumerate}

\end{remark}

	\section{Numerical studies}\label{sec4}
	In this section, we perform numerical computations to verify the theoretical results obtained in  Theorem \ref{thm7.3}. The results are obtained in the open-source finite element library FEniCS \citep{ABJ, LHL}. 
	
	The error tables shown below present the error estimates and order of convergence for the generalized Burgers-Huxley equation (GBHE) with kernel $(\eta \neq 0)$ and without Kernel $(\eta=0)$ for the equation 
	\begin{align}\label{1}
		\frac{\partial u}{\partial t}+\alpha u^{\delta}\sum_{i=1}^d\frac{\partial u}{\partial x_{i}}-\nu\Delta u-\eta \int_{0}^{t} K(t-s)\Delta u(s)\mathrm{~d}s=\beta u(1-u^{\delta})(u^{\delta}-\gamma)+f,
	\end{align}
	under the following norm 
	$$|\!|\!|u^k- u^k_h|\!|\!| = \|u_h(t_N)-u_h^N\|_{\L^2}^2+\sum_{k=1}^N\nu \Delta t \|\nabla (u_h(t_k) -u_h^k)\|_{\L^2}^2.$$
	
	In all the experiments, we are using a backward Euler discretization in time and FEM in space. Consider the temporal discretization $t_k=k\Delta t$, for a given $M$, $\Delta t=\frac{T}{M}$, $  t_k=k\Delta t$ and $h$ is the spatial discretization parameter. We choose  $\Delta t \propto h$ and the rate of convergence is given by $r = \log(\|u- u^1_h\|/\|u - u_h^2\|) /\log(h^1/h^2)$.

		\subsection{Accuracy verification for the smooth kernel}
		
		Consider the problem \eqref{1} defined on the domain $(0,1)^d$. Table \ref{table1} represents the errors and convergence rates for the numerical solution $u_h$ for the  smooth kernel $K(t) = e^{-\delta t},$ $t\in[0,T]$ with the exact solution $u = \Pi_{i=1}^{d}e^{-t}\sin(\pi x_i),$ where we have used positive quadrature rule for the memory term. 
		
		We choose the value of the parameters as follows:  $\alpha = \beta = \delta = \nu=1,$ and $\gamma= 0.5$, the forcing term has been calculated using the closed-form solution. We are using first-order polynomials to approximate the numerical solution for all the stimulations. Table \ref{table1} contains the error history for a sequence of refined uniform meshes for the numerical solution constructed using CFEM. The convergence results for both the GBHE with and without memory are represented in the table for both 2D and 3D. In the table, we can observe that the error in the energy norm decreases with the mesh size at the rate $O(h)$. We require at most three Newton iterations and three linear solver iterations to reach the prescribed tolerance of $10^{-10}$.

		\begin{table}
				\begin{center}
						\caption{Errors and convergence rates for the numerical solutions $u_h^k$  for the smooth function $K(t)= e^{-\delta t}$ and the given solution $ u = \Pi_{i=1}^{d}e^{-t}\sin(\pi x_i).$ }
						\label{table1}
						{\small
								\begin{tabular}{| c | c | c | c | c | c | c | c |}
										\hline
										\multicolumn{6}{|c|}{Error history in 2D }\\
										\hline
										\multirow{6}{*}{CGFEM}&{Mesh}&{$|\!|\!|\cdot|\!|\!|$-error for $\eta =0$}&{$O(h)$}&{$|\!|\!|\cdot|\!|\!|$-error for $\eta = 1$}&{$O(h)$}\\
										
										\cline{2- 6}
										&{$4\times 4$}&$7.21(-01)$ &$-$&$6.77(-01)$ &$-$ \\
										\cline{2- 6}
										&{$8\times 8$}&$3.54(-01)$ &$ 1.03$  &$3.28(-01)$ &$ 1.04$ \\
										\cline{2- 6}
										&{$16\times 16$}&$1.74(-01)$ &$1.02$ &$1.60(-01)$ &$1.04$ \\
										\cline{2- 6}
										&{$32\times 32$} &$ 8.62(-02)$ &$ 1.01$&$ 7.90(-02)$ &$ 1.02$ \\
										\cline{2- 6}
										&{$64\times 64$} &$ 4.29(-02)$ &$1.01$&$ 3.92(-02)$ &$1.01$ \\
										\cline{2- 6}
										&{$128\times 128$} &$ 2.14(-02)$ &$ 1.00$&$ 1.95(-02)$ &$ 1.00$ \\
										\cline{1- 6}
										
										\multicolumn{ 6}{|c|}{Error history in 3D }\\
										\hline
										\multirow{6}{*}{CGFEM}&{Mesh}&{$|\!|\!|\cdot|\!|\!|$-error for $\eta =0$}&{$O(h)$}&{$|\!|\!|\cdot|\!|\!|$-error for $\eta = 1$}&{$O(h)$}\\
										\cline{2- 6}
										&{$2\times 2 \times 2$}&$1.35(00)$ &$-$&$1.33(00)$ &$-$ \\
										\cline{2- 6}
										&{$4\times 4 \times4$}&$7.40(-01)$ &$0.86$&$7.13(-01)$ &$0.90$ \\
										\cline{2- 6}
										&{$8\times 8 \times 8$}&$3.71(-01)$ &$ 1.00$  &$3.53(-01)$ &$ 1.01$ \\
										\cline{2- 6}
										&{$16\times 16 \times16$}&$1.83(-01)$ &$1.02$ &$1.73(-01)$ &$1.03$ \\
										\cline{2- 6}
										&{$32\times 32 \times 32$} &$ 9.06(-02)$ &$ 1.01$&$ 8.55(-02)$ &$ 1.02$ \\
										\cline{1- 6}
										&{$64\times 64 \times 64$} &$ 4.50(-02)$ &$ 1.01$&$ 4.24(-02)$ &$ 1.01$ \\
										\cline{1- 6}
										\hline
								
								\end{tabular}}
						\vspace{-5mm}
					\end{center}
			\end{table}

	\subsection{Weakly singular kernel}
	Consider the problem (\ref{1}) defined on the domain $(0,1)^d$. We shall approximate $\psi$ in $J_k(\psi) = \int_0^{t_n} K(t_n-s)\psi(s) \mathrm{d}s$ by the piecewise constant function taking value $\psi(t_j)$ in $(t_{j-1},t_{j})$. We have obtained the errors and convergence rates for the numerical solution $u^k_h$ for the weakly singular kernel $K(t) =\frac{1}{\sqrt{t}}$ with two different expressions of the exact solution
	\begin{align*}
		\text{Case 1 : }u = (t^3-t^2+1)\prod\limits_{i=1}^{d}\sin(\pi x_i),  \qquad 	\text{Case 2 : }u = t\sqrt{t}\prod\limits_{i=1}^{d}\sin(\pi x_i).
	\end{align*}
	\begin{table}
		
		\begin{center}
			\caption{Errors and convergence rates for the numerical solutions $u_h^k$ for the weakly singular kernel  $K(t)= \frac{1}{\sqrt{t}}$ and the given solution $ u = \prod\limits_{i=1}^{d}(t^3-t^2+1)\sin(\pi x_i).$ }
			\label{table2}
			{\small
				\begin{tabular}{| c | c | c | c | c | c | c | c |}
					\hline
					\multicolumn{ 6}{|c|}{Error history in 2D }\\
					\hline
					\multirow{6}{*}{CGFEM}&{Mesh}&{$|\!|\!|\cdot|\!|\!|$-error for $\eta =0$}&{$O(h)$}&{$|\!|\!|\cdot|\!|\!|$-error for $\eta = 1$}&{$O(h)$}\\
					\cline{2- 6}
					&{$4\times 4$}&$8.04(-01)$ &$-$&$8.05(-01)$ &$-$ \\
					\cline{2- 6}
					&{$8\times 8$}&$4.14(-01)$ &$0.96$  &$4.14(-01)$ &$0.96$ \\
					\cline{2- 6}
					&{$16\times 16$}&$2.09(-01)$ &$0.99$ &$2.09(-01)$ &$0.99$ \\
					\cline{2- 6}
					&{$32\times 32$} &$ 1.05(-01)$ &$0.99$&$ 1.05(-01)$ &$ 0.99$ \\
					\cline{2- 6}
					&{$64\times 64$} &$ 5.26(-02)$ &$1.00$&$ 5.25(-02)$ &$1.00$ \\
					\cline{2- 6}
					&{$128\times 128$} &$ 2.63(-02)$ &$ 1.00$&$ 2.62(-02)$ &$ 1.00$ \\
					\cline{1- 6}
					\multicolumn{ 6}{|c|}{Error history in 3D }\\
					\hline
					\multirow{6}{*}{CGFEM}&{Mesh}&{$|\!|\!|\cdot|\!|\!|$-error for $\eta =0$}&{$O(h)$}&{$|\!|\!|\cdot|\!|\!|$-error for $\eta = 1$}&{$O(h)$}\\
					\cline{2- 6}
					&{$2\times 2 \times 2$}&$1.48(00)$ &$-$&$1.50(00)$ &$-$ \\
					\cline{2- 6}
					&{$4\times 4 \times 4$}&$8.62(-01)$ &$0.78$&$8.64(-01)$ &$0.80$ \\
					\cline{2- 6}
					&{$8\times 8 \times 8$}&$4.51(-01)$ &$ 0.93$  &$4.51(-01)$ &$ 0.94$ \\
					\cline{2- 6}
					&{$16\times 16 \times 16$}&$2.23(-01)$ &$1.01$ &$2.28(-01)$ &$0.98$ \\
					\cline{2- 6}
					&{$32\times 32 \times 32$} &$ 1.15(-01)$ &$ 0.96$&$ 1.15(-01)$ &$ 0.99$ \\
					\cline{2- 6}
					&{$64\times 64 \times 64$}&$ 5.74(-02)$ &$ 1.00$  &$ 5.74(-02)$ &$ 1.00$ \\
					\cline{2- 6}
					\hline
				
			\end{tabular}}
		\end{center}
		\vspace{-5mm}
	\end{table}
	The parameters are chosen as in \cite{KMR}, that is,  $\alpha = \beta = \delta = \nu=1,$ and $\gamma= 0.5$. Table \ref{table2} represents the convergence results related to the Case 1 and Table \ref{table3} corresponds to the Case 2. In both cases, we can observe that the error in the energy norm decreases with the mesh size at the rate $O(h)$. We require at most three Newton iterations and three linear solver iterations to reach the prescribed tolerance of $10^{-10}$.	
	\begin{table}
		\begin{center}
			{\small
				\caption{Errors and convergence rates for the numerical solutions $u^k_h$ for the weakly singular kernel  $K(t)= \frac{1}{\sqrt{t}}$ and the given solution $ u = t\sqrt{t}\sin(2\pi x)\sin(2\pi y).$ }
				\label{table3}
				\begin{tabular}{| c | c | c | c | c | c  |}
					\hline
					\multicolumn{ 6}{|c|}{Error history in 2D with weakly singular kernel,  $ u=t\sqrt{t}\sin(2\pi x)\sin(2\pi y)$ }\\
					\hline
					\multirow{6}{*}{CGFEM}&{Mesh}&{$|\!|\!|\cdot|\!|\!|$-error for $\eta =0$}&{$O(h)$}&{$|\!|\!|\cdot|\!|\!|$-error for $\eta = 1$}&{$O(h)$}\\
					\cline{2- 6}
					&{$4\times 4$}&$1.44(00)$ &$-$&$1.47(00)$ &$-$ \\
					\cline{2- 6}
					&{$8\times 8$} &$9.24(-01)$ &$ 0.64$ &$9.35(-01)$ &$ 0.65$ \\
					\cline{2- 6}
					&{$16\times 16$} &$5.02(-01)$ &$0.88$ &$5.06(-01)$ &$0.89$ \\
					\cline{2- 6}
					&{$32\times 32$} &$ 2.59(-01)$ &$ 0.95$ &$ 2.60(-01)$ &$ 0.96$ \\
					\cline{2- 6}
					&{$64\times 64$} &$ 1.31(-01)$ &$0.98$&$ 1.32(-01)$ &$0.99$ \\
					\cline{2- 6}
					&{$128\times 128$} &$ 6.59(-02)$ &$ 0.99$ &$ 6.61(-02)$ &$ 1.00$ \\
					\cline{1- 6}
					\multicolumn{ 6}{|c|}{Error history in 3D with weakly singular kernel,  $ u=t\sqrt{t}\sin(2\pi x)\sin(2\pi y)\sin(2\pi z)$ }\\
					\hline
					\multirow{6}{*}{CGFEM}&{Mesh}&{$|\!|\!|\cdot|\!|\!|$-error for $\eta =0$}&{$O(h)$}&{$|\!|\!|\cdot|\!|\!|$-error for $\eta = 1$}&{$O(h)$}\\
					\cline{2- 6}
					&{$4\times 4$}&$1.31(00)$ &$-$&$1.32(00)$ &$-$ \\
					\cline{2- 6}
					&{$8\times 8$} &$9.20(-01)$ &$ 0.51$ &$9.26(-01)$ &$ 0.51$ \\
					\cline{2- 6}
					&{$16\times 16$} &$5.15(-01)$ &$0.84$ &$5.18(-01)$ &$0.84$ \\
					\cline{2- 6}
					&{$32\times 32$} &$ 2.68(-01)$ &$ 0.94$ &$ 2.61(-01)$ &$ 0.99$ \\
					\cline{2- 6}
					&{$64\times 64$} &$ 1.36(-01)$ &$0.98$&$ 1.37(-01)$ &$0.93$ \\
					\cline{1- 6}
			\end{tabular}}
		\end{center}
	\end{table}
			\subsection{Application to Nerve Pulse Propagation}
		We conclude this section by showing a difference of qualitative illustration between a classical bistable 
		equation with a weakly singular kernel (without advection, $\alpha = 0$ and with a simplified cubic nonlinearity induced by $\delta = 1$) 
		and the generalized Burgers-Huxley equation with a weakly singular kernel. In this study, we perform a straightforward simulation to a transient problem which incorporates an additional ordinary differential equation (ODE). This inclusion allows for the emergence of self-sustained patterns, as discussed in \cite{DCS}.
		The system reads 
		\begin{align}\label{TBH}
			\nonumber\partial_t u + \alpha u^\delta \sum_{i=1}^d\partial_i u - \nu \Delta u -\eta\int_{0}^{t} K(t-s)\Delta u(s) \mathrm{~d}s &= \beta u (1-u^\delta)(u^\delta-\gamma) -v ,\\\partial_t v &= \varepsilon(u-\rho v),
		\end{align}
	where $\varepsilon$ and $\rho$ are parameters that can change the rest state and dynamics.	Setting $\delta = 1$, $\eta =0$ and $\alpha = 0$, one recovers the well-known FitzHugh-Nagumo equations 
		\[ \partial_t u - \nu \Delta u - \beta u (1-u)(u-\gamma) + v = 0, \qquad \partial_t v = \varepsilon(u-\rho v).\] 
		To obtain the weak formulation similar to \eqref{fd1} for the CFEM, we apply a simple backward Euler time discretization with constant time step $\Delta t = 0.2$ and other parameters similar to \cite{KMR}.
		For this example, we prescribe Neumann boundary conditions for $u$ on $\partial\Omega$. Figure~\ref{fig:ex3} depicts three snapshots of the evolution of $u$ (representing the action potential propagation in a 
		piece of nerve tissue, cardiac muscle, or any excitable media) for the classical FitzHugh-Nagumo system with weakly singular kernel
		vs. the modified delayed generalized Burgers-Huxley system with weakly singular kernel \eqref{TBH}. All numerical solutions computed using the CFEM setting. The difference in the spiral dynamics appears to be more responsive to the degree of non-linearity (governed by $\delta$) rather than the additional advection term (influenced by $\alpha$). 
		\begin{figure}[ht!]
			\caption{Example 3. Snapshots at $t =200$ of $u_h$ with memory coefficient $\eta = 0.0001, 0.01,1$ respectively for the FitzHugh-Nagumo model using $\delta = 1$, $\alpha = 0$ (top panels) and 
				for the modified generalized Burgers-Huxley system with $\delta = 1$, $\alpha = 0.1$ (middle row) and 
				with $\delta = 2$, $\alpha = 0.1$ (bottom).}\label{fig:ex3}
			\begin{center}
				\includegraphics[width=0.325\textwidth]{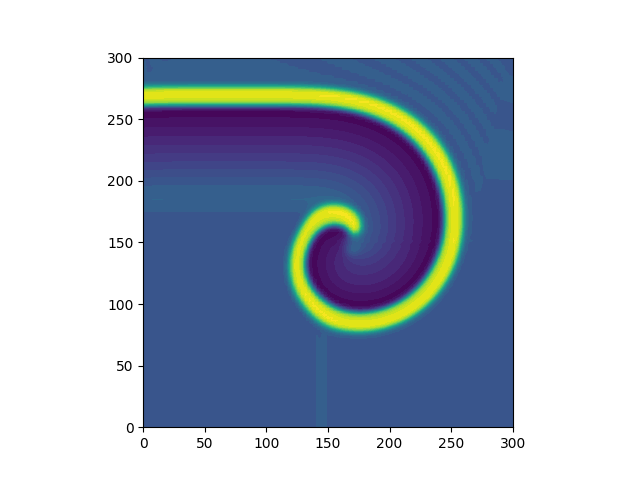}
				\includegraphics[width=0.325\textwidth]{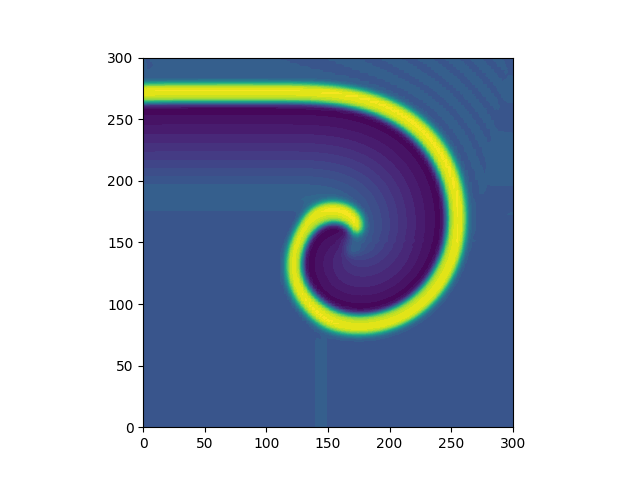}
				\includegraphics[width=0.325\textwidth]{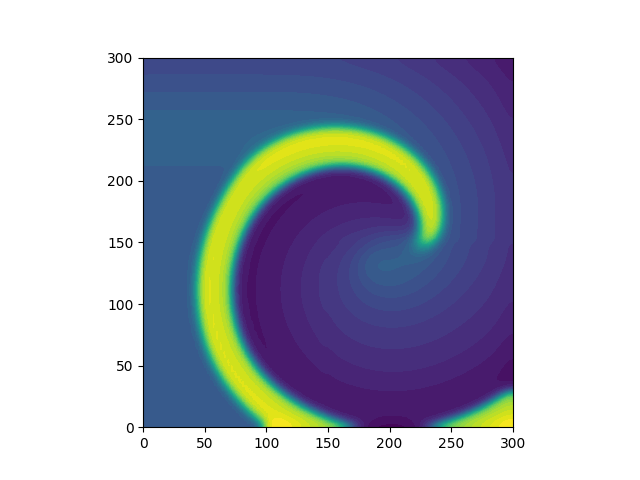}\\
				\includegraphics[width=0.325\textwidth]{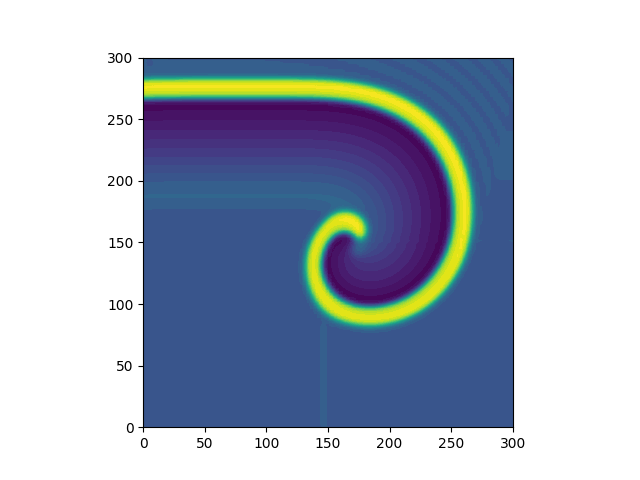}
				\includegraphics[width=0.325\textwidth]{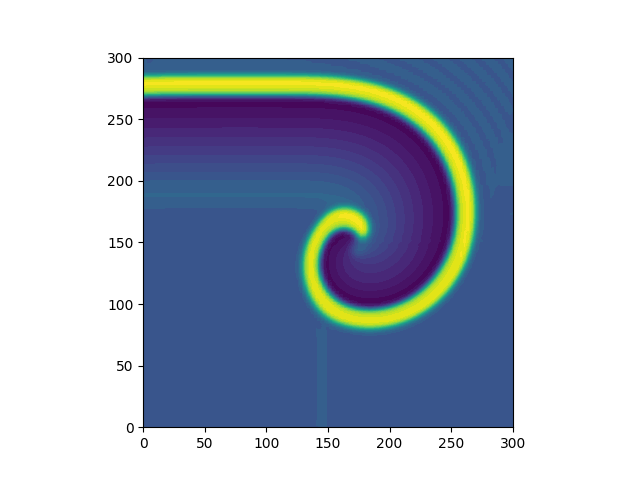}
				\includegraphics[width=0.325\textwidth]{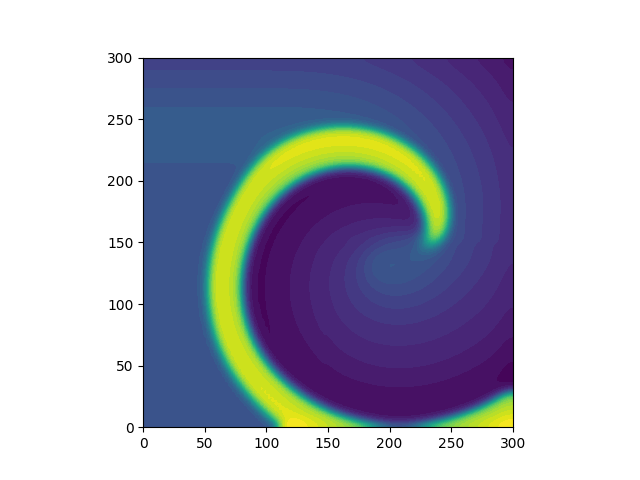}\\
				\includegraphics[width=0.325\textwidth]{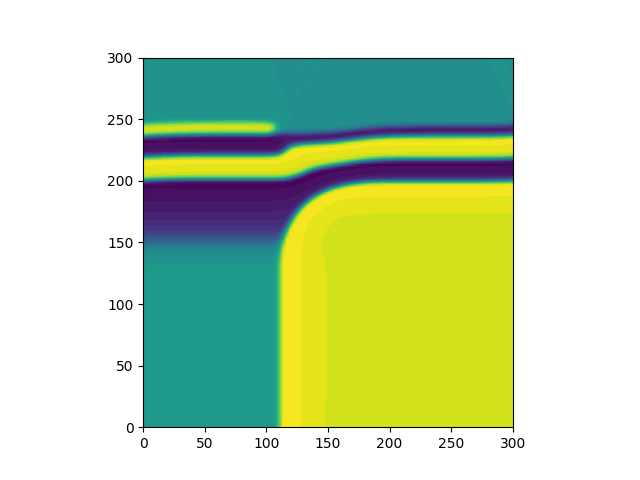}
				\includegraphics[width=0.325\textwidth]{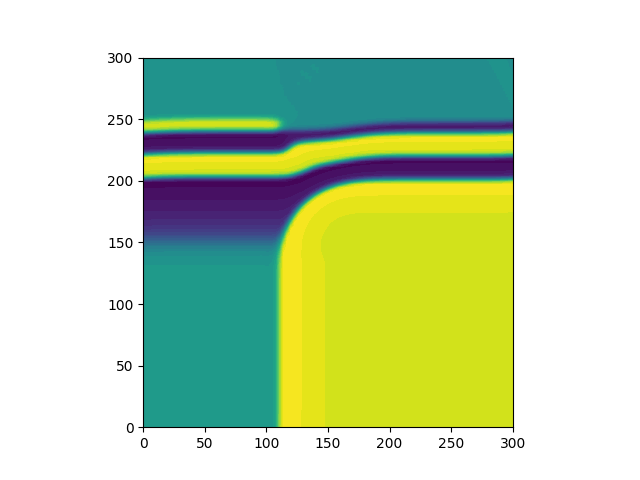}
				\includegraphics[width=0.325\textwidth]{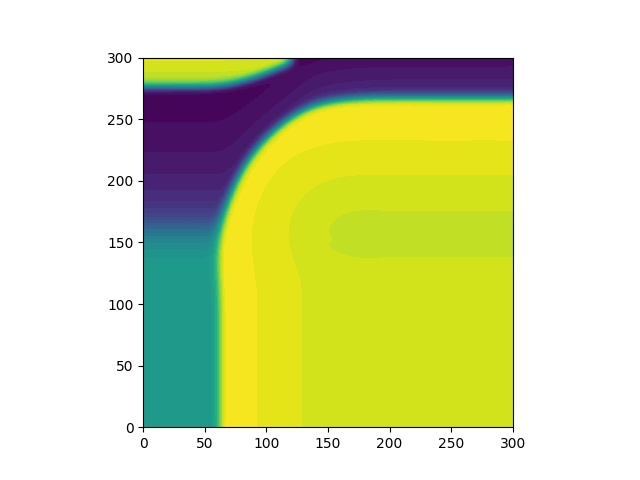}
			\end{center}
		\end{figure}

		\section{Acknowledgement}
		The first author would like to thank Ministry of Education, Government of India (Prime Minister Research Fellowship, PMRF ID
		: 2801816), for financial support to carry out his research work.
		
		\bibliographystyle{IMANUM-BIB}
		\bibliography{IMANUM-refs}
		
		\clearpage
		
	\end{document}